%% file: UQNotes.tex
\documentclass[preprint,3p,sort&compress]{elsarticle}
\journal{{\tt ArXiv}}

\usepackage[dvips]{epsfig}
\usepackage{graphicx} 
\usepackage{pdfpages}
\usepackage{latexsym}
\usepackage{verbatim}
\usepackage{amsmath,amsthm}
\usepackage[utf8]{inputenc}
\usepackage{amssymb}
\usepackage{bbm}
\usepackage{fonts}
\usepackage{enumerate}
\usepackage{placeins}
\usepackage{standalone}
\usepackage{paralist}
\usepackage{subcaption}

\usepackage{diagbox}

\usepackage{algorithm,algorithmicx,algpseudocode}
\usepackage{caption}
\usepackage[]{hyperref} 

\usepackage{pgfplots}
\pgfplotsset{compat=newest}       
\usepackage[]{externaltikz}
 \usepackage{relinput}
\newtheorem{theorem}{Theorem}[section]
\newtheorem{definition}[theorem]{Definition}
\newtheorem{remark}[theorem]{Remark}

\newtheorem{assumption}[theorem]{Assumption}

\newtheorem{lemma}[theorem]{Lemma}

\newcounter{tikzsubfigcounter}[figure]
\renewcommand{\thetikzsubfigcounter}{\the\numexpr\value{figure}+1\relax\alph{tikzsubfigcounter}}

\newcommand{\tikztitle}[1]{ %
	\refstepcounter{tikzsubfigcounter}
	\textbf{(\alph{tikzsubfigcounter})}\space\space #1 
}

\newcounter{tikzsubfigcounterinvisible}[figure]
\renewcommand{\thetikzsubfigcounterinvisible}{\the\numexpr\value{figure}+1\relax\alph{tikzsubfigcounterinvisible}}

\newcommand{\settikzlabel}[1]{ %
	\refstepcounter{tikzsubfigcounterinvisible} \label{#1} 
}

\newcommand{\refone}[1]{\textcolor{black}{#1}}
\newcommand{\reftwo}[1]{\textcolor{black}{#1}}

\numberwithin{equation}{section}

\title{A \hyperbolicity-preserving discontinuous stochastic Galerkin scheme for uncertain hyperbolic systems of equations}

\author[jd]{Jakob D\"{u}rrw\"{a}chter}
\address[jd]{Institut f\"{u}r  Aerodynamik und Gasdynamik, Universit\"{a}t Stuttgart, Pfaffenwaldring 21, 70569 Stuttgart, Germany, {\tt jakob.duerrwaechter@iag.uni-stuttgart.de}}

\author[tk]{Thomas  Kuhn}
\address[tk]{Institut f\"{u}r  Aerodynamik und Gasdynamik, Universit\"{a}t Stuttgart, Pfaffenwaldring 21, 70569 Stuttgart, Germany, {\tt thomas.kuhn@iag.uni-stuttgart.de}}

\author[fm]{Fabian Meyer}
\address[fm]{Institut f\"{u}r Angewandte Analysis und numerische Simulation, Universit\"{a}t Stuttgart, Pfaffenwaldring 57, 70569 Stuttgart, Germany, {\tt fabian.meyer@mathematik.uni-stuttgart.de}}

\author[ls]{Louisa Schlachter}
\address[ls]{Fachbereich Mathematik, TU Kaiserslautern, Erwin-Schr\"odinger-Str., 67663 Kaiserslautern, Germany, {\tt schlacht@mathematik.uni-kl.de}}

\author[fs]{Florian Schneider}
\address[fs]{Fachbereich Mathematik, TU Kaiserslautern, Erwin-Schr\"odinger-Str., 67663 Kaiserslautern, Germany, {\tt schneider@mathematik.uni-kl.de}}

\date{}

\input{usermacros}

\usepackage{media9}

\begin{document}

\begin{abstract}
Intrusive Uncertainty Quantification methods such as stochastic Galerkin are gaining popularity, whereas the classical stochastic Galerkin approach is not ensured to preserve hyperbolicity of the underlying hyperbolic system. We apply a modification of this method that uses a slope limiter to retain admissible solutions of the system, while providing high-order approximations in the physical and stochastic space. This is done using a spatial discontinuous Galerkin scheme and a Multi-Element stochastic Galerkin ansatz in the random space. We analyze the convergence of the resulting scheme and apply it to the compressible Euler equations with various uncertain initial states in one and two spatial domains with up to three uncertainties. The performance in multiple stochastic dimensions is compared to the non-intrusive Stochastic Collocation method. The numerical results underline the strength of our method, especially if discontinuities are present in the uncertainty of the solution.
\end{abstract}

\maketitle

\noindent

% {\bf Key words.}

\input{Sections/introduction}
\input{Sections/model}
\input{Sections/dg}
\input{Sections/results}
\input{Sections/conclusions}

% Bibliography
%%%%%%%%%%%%%%
%\newpage
%\section*{References}
\bibliographystyle{siam}
\bibliography{library,bibliography,libfabian}
%bibliography{library,libfabian}
\end{document}

%% file: usermacros.tex
\definecolor{greenyellow}   {cmyk}{0.15, 0   , 0.69, 0   }
\definecolor{yellow}        {cmyk}{0   , 0   , 1   , 0   }
\definecolor{goldenrod}     {cmyk}{0   , 0.10, 0.84, 0   }
\definecolor{dandelion}     {cmyk}{0   , 0.29, 0.84, 0   }
\definecolor{apricot}       {cmyk}{0   , 0.32, 0.52, 0   }
\definecolor{peach}         {cmyk}{0   , 0.50, 0.70, 0   }
\definecolor{melon}         {cmyk}{0   , 0.46, 0.50, 0   }
\definecolor{yelloworange}  {cmyk}{0   , 0.42, 1   , 0   }
\definecolor{orange}        {cmyk}{0   , 0.61, 0.87, 0   }
\definecolor{burntorange}   {cmyk}{0   , 0.51, 1   , 0   }
\definecolor{bittersweet}   {cmyk}{0   , 0.75, 1   , 0.24}
\definecolor{redorange}     {cmyk}{0   , 0.77, 0.87, 0   }
\definecolor{mahogany}      {cmyk}{0   , 0.85, 0.87, 0.35}
\definecolor{maroon}        {cmyk}{0   , 0.87, 0.68, 0.32}
\definecolor{brickred}      {cmyk}{0   , 0.89, 0.94, 0.28}
\definecolor{red}           {cmyk}{0   , 1   , 1   , 0   }
\definecolor{orangered}     {cmyk}{0   , 1   , 0.50, 0   }
\definecolor{rubinered}     {cmyk}{0   , 1   , 0.13, 0   }
\definecolor{wildstrawberry}{cmyk}{0   , 0.96, 0.39, 0   }
\definecolor{salmon}        {cmyk}{0   , 0.53, 0.38, 0   }
\definecolor{carnationpink} {cmyk}{0   , 0.63, 0   , 0   }
\definecolor{magenta}       {cmyk}{0   , 1   , 0   , 0   }
\definecolor{violetred}     {cmyk}{0   , 0.81, 0   , 0   }
\definecolor{rhodamine}     {cmyk}{0   , 0.82, 0   , 0   }
\definecolor{mulberry}      {cmyk}{0.34, 0.90, 0   , 0.02}
\definecolor{redviolet}     {cmyk}{0.07, 0.90, 0   , 0.34}
\definecolor{fuchsia}       {cmyk}{0.47, 0.91, 0   , 0.08}
\definecolor{lavender}      {cmyk}{0   , 0.48, 0   , 0   }
\definecolor{thistle}       {cmyk}{0.12, 0.59, 0   , 0   }
\definecolor{orchid}        {cmyk}{0.32, 0.64, 0   , 0   }
\definecolor{darkorchid}    {cmyk}{0.40, 0.80, 0.20, 0   }
\definecolor{purple}        {cmyk}{0.45, 0.86, 0   , 0   }
\definecolor{plum}          {cmyk}{0.50, 1   , 0   , 0   }
\definecolor{violet}        {cmyk}{0.79, 0.88, 0   , 0   }
\definecolor{royalpurple}   {cmyk}{0.75, 0.90, 0   , 0   }
\definecolor{blueviolet}    {cmyk}{0.86, 0.91, 0   , 0.04}
\definecolor{periwinkle}    {cmyk}{0.57, 0.55, 0   , 0   }
\definecolor{cadetblue}     {cmyk}{0.62, 0.57, 0.23, 0   }
\definecolor{cornflowerblue}{cmyk}{0.65, 0.13, 0   , 0   }
\definecolor{midnightblue}  {cmyk}{0.98, 0.13, 0   , 0.43}
\definecolor{navyblue}      {cmyk}{0.94, 0.54, 0   , 0   }
\definecolor{royalblue}     {cmyk}{1   , 0.50, 0   , 0   }
\definecolor{blue}          {cmyk}{1   , 1   , 0   , 0   }
\definecolor{cerulean}      {cmyk}{0.94, 0.11, 0   , 0   }
\definecolor{cyan}          {cmyk}{1   , 0   , 0   , 0   }
\definecolor{processblue}   {cmyk}{0.96, 0   , 0   , 0   }
\definecolor{skyblue}       {cmyk}{0.62, 0   , 0.12, 0   }
\definecolor{turquoise}     {cmyk}{0.85, 0   , 0.20, 0   }
\definecolor{tealblue}      {cmyk}{0.86, 0   , 0.34, 0.02}
\definecolor{aquamarine}    {cmyk}{0.82, 0   , 0.30, 0   }
\definecolor{bluegreen}     {cmyk}{0.85, 0   , 0.33, 0   }
\definecolor{emerald}       {cmyk}{1   , 0   , 0.50, 0   }
\definecolor{junglegreen}   {cmyk}{0.99, 0   , 0.52, 0   }
\definecolor{seagreen}      {cmyk}{0.69, 0   , 0.50, 0   }
\definecolor{green}         {cmyk}{1   , 0   , 1   , 0   }
\definecolor{forestgreen}   {cmyk}{0.91, 0   , 0.88, 0.12}
\definecolor{pinegreen}     {cmyk}{0.92, 0   , 0.59, 0.25}
\definecolor{limegreen}     {cmyk}{0.50, 0   , 1   , 0   }
\definecolor{yellowgreen}   {cmyk}{0.44, 0   , 0.74, 0   }
\definecolor{springgreen}   {cmyk}{0.26, 0   , 0.76, 0   }
\definecolor{olivegreen}    {cmyk}{0.64, 0   , 0.95, 0.40}
\definecolor{rawsienna}     {cmyk}{0   , 0.72, 1   , 0.45}
\definecolor{sepia}         {cmyk}{0   , 0.83, 1   , 0.70}
\definecolor{brown}         {cmyk}{0   , 0.81, 1   , 0.60}
\definecolor{tan}           {cmyk}{0.14, 0.42, 0.56, 0   }
\definecolor{gray}          {cmyk}{0   , 0   , 0   , 0.50}
\definecolor{black}         {cmyk}{0   , 0   , 0   , 1   }
\definecolor{white}         {cmyk}{0   , 0   , 0   , 0   } 

\usepgfplotslibrary{groupplots}
\usepackage{pgfplotstable}

\pgfplotsset{
	colormap={jet}{
rgb(0.000000 pt)=(0.000000,0.000000,0.504000);
rgb(1.000000 pt)=(0.000000,0.000000,0.508000);
rgb(2.000000 pt)=(0.000000,0.000000,0.512000);
rgb(3.000000 pt)=(0.000000,0.000000,0.516000);
rgb(4.000000 pt)=(0.000000,0.000000,0.520000);
rgb(5.000000 pt)=(0.000000,0.000000,0.524000);
rgb(6.000000 pt)=(0.000000,0.000000,0.528000);
rgb(7.000000 pt)=(0.000000,0.000000,0.532000);
rgb(8.000000 pt)=(0.000000,0.000000,0.536000);
rgb(9.000000 pt)=(0.000000,0.000000,0.540000);
rgb(10.000000 pt)=(0.000000,0.000000,0.544000);
rgb(11.000000 pt)=(0.000000,0.000000,0.548000);
rgb(12.000000 pt)=(0.000000,0.000000,0.552000);
rgb(13.000000 pt)=(0.000000,0.000000,0.556000);
rgb(14.000000 pt)=(0.000000,0.000000,0.560000);
rgb(15.000000 pt)=(0.000000,0.000000,0.564000);
rgb(16.000000 pt)=(0.000000,0.000000,0.568000);
rgb(17.000000 pt)=(0.000000,0.000000,0.572000);
rgb(18.000000 pt)=(0.000000,0.000000,0.576000);
rgb(19.000000 pt)=(0.000000,0.000000,0.580000);
rgb(20.000000 pt)=(0.000000,0.000000,0.584000);
rgb(21.000000 pt)=(0.000000,0.000000,0.588000);
rgb(22.000000 pt)=(0.000000,0.000000,0.592000);
rgb(23.000000 pt)=(0.000000,0.000000,0.596000);
rgb(24.000000 pt)=(0.000000,0.000000,0.600000);
rgb(25.000000 pt)=(0.000000,0.000000,0.604000);
rgb(26.000000 pt)=(0.000000,0.000000,0.608000);
rgb(27.000000 pt)=(0.000000,0.000000,0.612000);
rgb(28.000000 pt)=(0.000000,0.000000,0.616000);
rgb(29.000000 pt)=(0.000000,0.000000,0.620000);
rgb(30.000000 pt)=(0.000000,0.000000,0.624000);
rgb(31.000000 pt)=(0.000000,0.000000,0.628000);
rgb(32.000000 pt)=(0.000000,0.000000,0.632000);
rgb(33.000000 pt)=(0.000000,0.000000,0.636000);
rgb(34.000000 pt)=(0.000000,0.000000,0.640000);
rgb(35.000000 pt)=(0.000000,0.000000,0.644000);
rgb(36.000000 pt)=(0.000000,0.000000,0.648000);
rgb(37.000000 pt)=(0.000000,0.000000,0.652000);
rgb(38.000000 pt)=(0.000000,0.000000,0.656000);
rgb(39.000000 pt)=(0.000000,0.000000,0.660000);
rgb(40.000000 pt)=(0.000000,0.000000,0.664000);
rgb(41.000000 pt)=(0.000000,0.000000,0.668000);
rgb(42.000000 pt)=(0.000000,0.000000,0.672000);
rgb(43.000000 pt)=(0.000000,0.000000,0.676000);
rgb(44.000000 pt)=(0.000000,0.000000,0.680000);
rgb(45.000000 pt)=(0.000000,0.000000,0.684000);
rgb(46.000000 pt)=(0.000000,0.000000,0.688000);
rgb(47.000000 pt)=(0.000000,0.000000,0.692000);
rgb(48.000000 pt)=(0.000000,0.000000,0.696000);
rgb(49.000000 pt)=(0.000000,0.000000,0.700000);
rgb(50.000000 pt)=(0.000000,0.000000,0.704000);
rgb(51.000000 pt)=(0.000000,0.000000,0.708000);
rgb(52.000000 pt)=(0.000000,0.000000,0.712000);
rgb(53.000000 pt)=(0.000000,0.000000,0.716000);
rgb(54.000000 pt)=(0.000000,0.000000,0.720000);
rgb(55.000000 pt)=(0.000000,0.000000,0.724000);
rgb(56.000000 pt)=(0.000000,0.000000,0.728000);
rgb(57.000000 pt)=(0.000000,0.000000,0.732000);
rgb(58.000000 pt)=(0.000000,0.000000,0.736000);
rgb(59.000000 pt)=(0.000000,0.000000,0.740000);
rgb(60.000000 pt)=(0.000000,0.000000,0.744000);
rgb(61.000000 pt)=(0.000000,0.000000,0.748000);
rgb(62.000000 pt)=(0.000000,0.000000,0.752000);
rgb(63.000000 pt)=(0.000000,0.000000,0.756000);
rgb(64.000000 pt)=(0.000000,0.000000,0.760000);
rgb(65.000000 pt)=(0.000000,0.000000,0.764000);
rgb(66.000000 pt)=(0.000000,0.000000,0.768000);
rgb(67.000000 pt)=(0.000000,0.000000,0.772000);
rgb(68.000000 pt)=(0.000000,0.000000,0.776000);
rgb(69.000000 pt)=(0.000000,0.000000,0.780000);
rgb(70.000000 pt)=(0.000000,0.000000,0.784000);
rgb(71.000000 pt)=(0.000000,0.000000,0.788000);
rgb(72.000000 pt)=(0.000000,0.000000,0.792000);
rgb(73.000000 pt)=(0.000000,0.000000,0.796000);
rgb(74.000000 pt)=(0.000000,0.000000,0.800000);
rgb(75.000000 pt)=(0.000000,0.000000,0.804000);
rgb(76.000000 pt)=(0.000000,0.000000,0.808000);
rgb(77.000000 pt)=(0.000000,0.000000,0.812000);
rgb(78.000000 pt)=(0.000000,0.000000,0.816000);
rgb(79.000000 pt)=(0.000000,0.000000,0.820000);
rgb(80.000000 pt)=(0.000000,0.000000,0.824000);
rgb(81.000000 pt)=(0.000000,0.000000,0.828000);
rgb(82.000000 pt)=(0.000000,0.000000,0.832000);
rgb(83.000000 pt)=(0.000000,0.000000,0.836000);
rgb(84.000000 pt)=(0.000000,0.000000,0.840000);
rgb(85.000000 pt)=(0.000000,0.000000,0.844000);
rgb(86.000000 pt)=(0.000000,0.000000,0.848000);
rgb(87.000000 pt)=(0.000000,0.000000,0.852000);
rgb(88.000000 pt)=(0.000000,0.000000,0.856000);
rgb(89.000000 pt)=(0.000000,0.000000,0.860000);
rgb(90.000000 pt)=(0.000000,0.000000,0.864000);
rgb(91.000000 pt)=(0.000000,0.000000,0.868000);
rgb(92.000000 pt)=(0.000000,0.000000,0.872000);
rgb(93.000000 pt)=(0.000000,0.000000,0.876000);
rgb(94.000000 pt)=(0.000000,0.000000,0.880000);
rgb(95.000000 pt)=(0.000000,0.000000,0.884000);
rgb(96.000000 pt)=(0.000000,0.000000,0.888000);
rgb(97.000000 pt)=(0.000000,0.000000,0.892000);
rgb(98.000000 pt)=(0.000000,0.000000,0.896000);
rgb(99.000000 pt)=(0.000000,0.000000,0.900000);
rgb(100.000000 pt)=(0.000000,0.000000,0.904000);
rgb(101.000000 pt)=(0.000000,0.000000,0.908000);
rgb(102.000000 pt)=(0.000000,0.000000,0.912000);
rgb(103.000000 pt)=(0.000000,0.000000,0.916000);
rgb(104.000000 pt)=(0.000000,0.000000,0.920000);
rgb(105.000000 pt)=(0.000000,0.000000,0.924000);
rgb(106.000000 pt)=(0.000000,0.000000,0.928000);
rgb(107.000000 pt)=(0.000000,0.000000,0.932000);
rgb(108.000000 pt)=(0.000000,0.000000,0.936000);
rgb(109.000000 pt)=(0.000000,0.000000,0.940000);
rgb(110.000000 pt)=(0.000000,0.000000,0.944000);
rgb(111.000000 pt)=(0.000000,0.000000,0.948000);
rgb(112.000000 pt)=(0.000000,0.000000,0.952000);
rgb(113.000000 pt)=(0.000000,0.000000,0.956000);
rgb(114.000000 pt)=(0.000000,0.000000,0.960000);
rgb(115.000000 pt)=(0.000000,0.000000,0.964000);
rgb(116.000000 pt)=(0.000000,0.000000,0.968000);
rgb(117.000000 pt)=(0.000000,0.000000,0.972000);
rgb(118.000000 pt)=(0.000000,0.000000,0.976000);
rgb(119.000000 pt)=(0.000000,0.000000,0.980000);
rgb(120.000000 pt)=(0.000000,0.000000,0.984000);
rgb(121.000000 pt)=(0.000000,0.000000,0.988000);
rgb(122.000000 pt)=(0.000000,0.000000,0.992000);
rgb(123.000000 pt)=(0.000000,0.000000,0.996000);
rgb(124.000000 pt)=(0.000000,0.000000,1.000000);
rgb(125.000000 pt)=(0.000000,0.004000,1.000000);
rgb(126.000000 pt)=(0.000000,0.008000,1.000000);
rgb(127.000000 pt)=(0.000000,0.012000,1.000000);
rgb(128.000000 pt)=(0.000000,0.016000,1.000000);
rgb(129.000000 pt)=(0.000000,0.020000,1.000000);
rgb(130.000000 pt)=(0.000000,0.024000,1.000000);
rgb(131.000000 pt)=(0.000000,0.028000,1.000000);
rgb(132.000000 pt)=(0.000000,0.032000,1.000000);
rgb(133.000000 pt)=(0.000000,0.036000,1.000000);
rgb(134.000000 pt)=(0.000000,0.040000,1.000000);
rgb(135.000000 pt)=(0.000000,0.044000,1.000000);
rgb(136.000000 pt)=(0.000000,0.048000,1.000000);
rgb(137.000000 pt)=(0.000000,0.052000,1.000000);
rgb(138.000000 pt)=(0.000000,0.056000,1.000000);
rgb(139.000000 pt)=(0.000000,0.060000,1.000000);
rgb(140.000000 pt)=(0.000000,0.064000,1.000000);
rgb(141.000000 pt)=(0.000000,0.068000,1.000000);
rgb(142.000000 pt)=(0.000000,0.072000,1.000000);
rgb(143.000000 pt)=(0.000000,0.076000,1.000000);
rgb(144.000000 pt)=(0.000000,0.080000,1.000000);
rgb(145.000000 pt)=(0.000000,0.084000,1.000000);
rgb(146.000000 pt)=(0.000000,0.088000,1.000000);
rgb(147.000000 pt)=(0.000000,0.092000,1.000000);
rgb(148.000000 pt)=(0.000000,0.096000,1.000000);
rgb(149.000000 pt)=(0.000000,0.100000,1.000000);
rgb(150.000000 pt)=(0.000000,0.104000,1.000000);
rgb(151.000000 pt)=(0.000000,0.108000,1.000000);
rgb(152.000000 pt)=(0.000000,0.112000,1.000000);
rgb(153.000000 pt)=(0.000000,0.116000,1.000000);
rgb(154.000000 pt)=(0.000000,0.120000,1.000000);
rgb(155.000000 pt)=(0.000000,0.124000,1.000000);
rgb(156.000000 pt)=(0.000000,0.128000,1.000000);
rgb(157.000000 pt)=(0.000000,0.132000,1.000000);
rgb(158.000000 pt)=(0.000000,0.136000,1.000000);
rgb(159.000000 pt)=(0.000000,0.140000,1.000000);
rgb(160.000000 pt)=(0.000000,0.144000,1.000000);
rgb(161.000000 pt)=(0.000000,0.148000,1.000000);
rgb(162.000000 pt)=(0.000000,0.152000,1.000000);
rgb(163.000000 pt)=(0.000000,0.156000,1.000000);
rgb(164.000000 pt)=(0.000000,0.160000,1.000000);
rgb(165.000000 pt)=(0.000000,0.164000,1.000000);
rgb(166.000000 pt)=(0.000000,0.168000,1.000000);
rgb(167.000000 pt)=(0.000000,0.172000,1.000000);
rgb(168.000000 pt)=(0.000000,0.176000,1.000000);
rgb(169.000000 pt)=(0.000000,0.180000,1.000000);
rgb(170.000000 pt)=(0.000000,0.184000,1.000000);
rgb(171.000000 pt)=(0.000000,0.188000,1.000000);
rgb(172.000000 pt)=(0.000000,0.192000,1.000000);
rgb(173.000000 pt)=(0.000000,0.196000,1.000000);
rgb(174.000000 pt)=(0.000000,0.200000,1.000000);
rgb(175.000000 pt)=(0.000000,0.204000,1.000000);
rgb(176.000000 pt)=(0.000000,0.208000,1.000000);
rgb(177.000000 pt)=(0.000000,0.212000,1.000000);
rgb(178.000000 pt)=(0.000000,0.216000,1.000000);
rgb(179.000000 pt)=(0.000000,0.220000,1.000000);
rgb(180.000000 pt)=(0.000000,0.224000,1.000000);
rgb(181.000000 pt)=(0.000000,0.228000,1.000000);
rgb(182.000000 pt)=(0.000000,0.232000,1.000000);
rgb(183.000000 pt)=(0.000000,0.236000,1.000000);
rgb(184.000000 pt)=(0.000000,0.240000,1.000000);
rgb(185.000000 pt)=(0.000000,0.244000,1.000000);
rgb(186.000000 pt)=(0.000000,0.248000,1.000000);
rgb(187.000000 pt)=(0.000000,0.252000,1.000000);
rgb(188.000000 pt)=(0.000000,0.256000,1.000000);
rgb(189.000000 pt)=(0.000000,0.260000,1.000000);
rgb(190.000000 pt)=(0.000000,0.264000,1.000000);
rgb(191.000000 pt)=(0.000000,0.268000,1.000000);
rgb(192.000000 pt)=(0.000000,0.272000,1.000000);
rgb(193.000000 pt)=(0.000000,0.276000,1.000000);
rgb(194.000000 pt)=(0.000000,0.280000,1.000000);
rgb(195.000000 pt)=(0.000000,0.284000,1.000000);
rgb(196.000000 pt)=(0.000000,0.288000,1.000000);
rgb(197.000000 pt)=(0.000000,0.292000,1.000000);
rgb(198.000000 pt)=(0.000000,0.296000,1.000000);
rgb(199.000000 pt)=(0.000000,0.300000,1.000000);
rgb(200.000000 pt)=(0.000000,0.304000,1.000000);
rgb(201.000000 pt)=(0.000000,0.308000,1.000000);
rgb(202.000000 pt)=(0.000000,0.312000,1.000000);
rgb(203.000000 pt)=(0.000000,0.316000,1.000000);
rgb(204.000000 pt)=(0.000000,0.320000,1.000000);
rgb(205.000000 pt)=(0.000000,0.324000,1.000000);
rgb(206.000000 pt)=(0.000000,0.328000,1.000000);
rgb(207.000000 pt)=(0.000000,0.332000,1.000000);
rgb(208.000000 pt)=(0.000000,0.336000,1.000000);
rgb(209.000000 pt)=(0.000000,0.340000,1.000000);
rgb(210.000000 pt)=(0.000000,0.344000,1.000000);
rgb(211.000000 pt)=(0.000000,0.348000,1.000000);
rgb(212.000000 pt)=(0.000000,0.352000,1.000000);
rgb(213.000000 pt)=(0.000000,0.356000,1.000000);
rgb(214.000000 pt)=(0.000000,0.360000,1.000000);
rgb(215.000000 pt)=(0.000000,0.364000,1.000000);
rgb(216.000000 pt)=(0.000000,0.368000,1.000000);
rgb(217.000000 pt)=(0.000000,0.372000,1.000000);
rgb(218.000000 pt)=(0.000000,0.376000,1.000000);
rgb(219.000000 pt)=(0.000000,0.380000,1.000000);
rgb(220.000000 pt)=(0.000000,0.384000,1.000000);
rgb(221.000000 pt)=(0.000000,0.388000,1.000000);
rgb(222.000000 pt)=(0.000000,0.392000,1.000000);
rgb(223.000000 pt)=(0.000000,0.396000,1.000000);
rgb(224.000000 pt)=(0.000000,0.400000,1.000000);
rgb(225.000000 pt)=(0.000000,0.404000,1.000000);
rgb(226.000000 pt)=(0.000000,0.408000,1.000000);
rgb(227.000000 pt)=(0.000000,0.412000,1.000000);
rgb(228.000000 pt)=(0.000000,0.416000,1.000000);
rgb(229.000000 pt)=(0.000000,0.420000,1.000000);
rgb(230.000000 pt)=(0.000000,0.424000,1.000000);
rgb(231.000000 pt)=(0.000000,0.428000,1.000000);
rgb(232.000000 pt)=(0.000000,0.432000,1.000000);
rgb(233.000000 pt)=(0.000000,0.436000,1.000000);
rgb(234.000000 pt)=(0.000000,0.440000,1.000000);
rgb(235.000000 pt)=(0.000000,0.444000,1.000000);
rgb(236.000000 pt)=(0.000000,0.448000,1.000000);
rgb(237.000000 pt)=(0.000000,0.452000,1.000000);
rgb(238.000000 pt)=(0.000000,0.456000,1.000000);
rgb(239.000000 pt)=(0.000000,0.460000,1.000000);
rgb(240.000000 pt)=(0.000000,0.464000,1.000000);
rgb(241.000000 pt)=(0.000000,0.468000,1.000000);
rgb(242.000000 pt)=(0.000000,0.472000,1.000000);
rgb(243.000000 pt)=(0.000000,0.476000,1.000000);
rgb(244.000000 pt)=(0.000000,0.480000,1.000000);
rgb(245.000000 pt)=(0.000000,0.484000,1.000000);
rgb(246.000000 pt)=(0.000000,0.488000,1.000000);
rgb(247.000000 pt)=(0.000000,0.492000,1.000000);
rgb(248.000000 pt)=(0.000000,0.496000,1.000000);
rgb(249.000000 pt)=(0.000000,0.500000,1.000000);
rgb(250.000000 pt)=(0.000000,0.504000,1.000000);
rgb(251.000000 pt)=(0.000000,0.508000,1.000000);
rgb(252.000000 pt)=(0.000000,0.512000,1.000000);
rgb(253.000000 pt)=(0.000000,0.516000,1.000000);
rgb(254.000000 pt)=(0.000000,0.520000,1.000000);
rgb(255.000000 pt)=(0.000000,0.524000,1.000000);
rgb(256.000000 pt)=(0.000000,0.528000,1.000000);
rgb(257.000000 pt)=(0.000000,0.532000,1.000000);
rgb(258.000000 pt)=(0.000000,0.536000,1.000000);
rgb(259.000000 pt)=(0.000000,0.540000,1.000000);
rgb(260.000000 pt)=(0.000000,0.544000,1.000000);
rgb(261.000000 pt)=(0.000000,0.548000,1.000000);
rgb(262.000000 pt)=(0.000000,0.552000,1.000000);
rgb(263.000000 pt)=(0.000000,0.556000,1.000000);
rgb(264.000000 pt)=(0.000000,0.560000,1.000000);
rgb(265.000000 pt)=(0.000000,0.564000,1.000000);
rgb(266.000000 pt)=(0.000000,0.568000,1.000000);
rgb(267.000000 pt)=(0.000000,0.572000,1.000000);
rgb(268.000000 pt)=(0.000000,0.576000,1.000000);
rgb(269.000000 pt)=(0.000000,0.580000,1.000000);
rgb(270.000000 pt)=(0.000000,0.584000,1.000000);
rgb(271.000000 pt)=(0.000000,0.588000,1.000000);
rgb(272.000000 pt)=(0.000000,0.592000,1.000000);
rgb(273.000000 pt)=(0.000000,0.596000,1.000000);
rgb(274.000000 pt)=(0.000000,0.600000,1.000000);
rgb(275.000000 pt)=(0.000000,0.604000,1.000000);
rgb(276.000000 pt)=(0.000000,0.608000,1.000000);
rgb(277.000000 pt)=(0.000000,0.612000,1.000000);
rgb(278.000000 pt)=(0.000000,0.616000,1.000000);
rgb(279.000000 pt)=(0.000000,0.620000,1.000000);
rgb(280.000000 pt)=(0.000000,0.624000,1.000000);
rgb(281.000000 pt)=(0.000000,0.628000,1.000000);
rgb(282.000000 pt)=(0.000000,0.632000,1.000000);
rgb(283.000000 pt)=(0.000000,0.636000,1.000000);
rgb(284.000000 pt)=(0.000000,0.640000,1.000000);
rgb(285.000000 pt)=(0.000000,0.644000,1.000000);
rgb(286.000000 pt)=(0.000000,0.648000,1.000000);
rgb(287.000000 pt)=(0.000000,0.652000,1.000000);
rgb(288.000000 pt)=(0.000000,0.656000,1.000000);
rgb(289.000000 pt)=(0.000000,0.660000,1.000000);
rgb(290.000000 pt)=(0.000000,0.664000,1.000000);
rgb(291.000000 pt)=(0.000000,0.668000,1.000000);
rgb(292.000000 pt)=(0.000000,0.672000,1.000000);
rgb(293.000000 pt)=(0.000000,0.676000,1.000000);
rgb(294.000000 pt)=(0.000000,0.680000,1.000000);
rgb(295.000000 pt)=(0.000000,0.684000,1.000000);
rgb(296.000000 pt)=(0.000000,0.688000,1.000000);
rgb(297.000000 pt)=(0.000000,0.692000,1.000000);
rgb(298.000000 pt)=(0.000000,0.696000,1.000000);
rgb(299.000000 pt)=(0.000000,0.700000,1.000000);
rgb(300.000000 pt)=(0.000000,0.704000,1.000000);
rgb(301.000000 pt)=(0.000000,0.708000,1.000000);
rgb(302.000000 pt)=(0.000000,0.712000,1.000000);
rgb(303.000000 pt)=(0.000000,0.716000,1.000000);
rgb(304.000000 pt)=(0.000000,0.720000,1.000000);
rgb(305.000000 pt)=(0.000000,0.724000,1.000000);
rgb(306.000000 pt)=(0.000000,0.728000,1.000000);
rgb(307.000000 pt)=(0.000000,0.732000,1.000000);
rgb(308.000000 pt)=(0.000000,0.736000,1.000000);
rgb(309.000000 pt)=(0.000000,0.740000,1.000000);
rgb(310.000000 pt)=(0.000000,0.744000,1.000000);
rgb(311.000000 pt)=(0.000000,0.748000,1.000000);
rgb(312.000000 pt)=(0.000000,0.752000,1.000000);
rgb(313.000000 pt)=(0.000000,0.756000,1.000000);
rgb(314.000000 pt)=(0.000000,0.760000,1.000000);
rgb(315.000000 pt)=(0.000000,0.764000,1.000000);
rgb(316.000000 pt)=(0.000000,0.768000,1.000000);
rgb(317.000000 pt)=(0.000000,0.772000,1.000000);
rgb(318.000000 pt)=(0.000000,0.776000,1.000000);
rgb(319.000000 pt)=(0.000000,0.780000,1.000000);
rgb(320.000000 pt)=(0.000000,0.784000,1.000000);
rgb(321.000000 pt)=(0.000000,0.788000,1.000000);
rgb(322.000000 pt)=(0.000000,0.792000,1.000000);
rgb(323.000000 pt)=(0.000000,0.796000,1.000000);
rgb(324.000000 pt)=(0.000000,0.800000,1.000000);
rgb(325.000000 pt)=(0.000000,0.804000,1.000000);
rgb(326.000000 pt)=(0.000000,0.808000,1.000000);
rgb(327.000000 pt)=(0.000000,0.812000,1.000000);
rgb(328.000000 pt)=(0.000000,0.816000,1.000000);
rgb(329.000000 pt)=(0.000000,0.820000,1.000000);
rgb(330.000000 pt)=(0.000000,0.824000,1.000000);
rgb(331.000000 pt)=(0.000000,0.828000,1.000000);
rgb(332.000000 pt)=(0.000000,0.832000,1.000000);
rgb(333.000000 pt)=(0.000000,0.836000,1.000000);
rgb(334.000000 pt)=(0.000000,0.840000,1.000000);
rgb(335.000000 pt)=(0.000000,0.844000,1.000000);
rgb(336.000000 pt)=(0.000000,0.848000,1.000000);
rgb(337.000000 pt)=(0.000000,0.852000,1.000000);
rgb(338.000000 pt)=(0.000000,0.856000,1.000000);
rgb(339.000000 pt)=(0.000000,0.860000,1.000000);
rgb(340.000000 pt)=(0.000000,0.864000,1.000000);
rgb(341.000000 pt)=(0.000000,0.868000,1.000000);
rgb(342.000000 pt)=(0.000000,0.872000,1.000000);
rgb(343.000000 pt)=(0.000000,0.876000,1.000000);
rgb(344.000000 pt)=(0.000000,0.880000,1.000000);
rgb(345.000000 pt)=(0.000000,0.884000,1.000000);
rgb(346.000000 pt)=(0.000000,0.888000,1.000000);
rgb(347.000000 pt)=(0.000000,0.892000,1.000000);
rgb(348.000000 pt)=(0.000000,0.896000,1.000000);
rgb(349.000000 pt)=(0.000000,0.900000,1.000000);
rgb(350.000000 pt)=(0.000000,0.904000,1.000000);
rgb(351.000000 pt)=(0.000000,0.908000,1.000000);
rgb(352.000000 pt)=(0.000000,0.912000,1.000000);
rgb(353.000000 pt)=(0.000000,0.916000,1.000000);
rgb(354.000000 pt)=(0.000000,0.920000,1.000000);
rgb(355.000000 pt)=(0.000000,0.924000,1.000000);
rgb(356.000000 pt)=(0.000000,0.928000,1.000000);
rgb(357.000000 pt)=(0.000000,0.932000,1.000000);
rgb(358.000000 pt)=(0.000000,0.936000,1.000000);
rgb(359.000000 pt)=(0.000000,0.940000,1.000000);
rgb(360.000000 pt)=(0.000000,0.944000,1.000000);
rgb(361.000000 pt)=(0.000000,0.948000,1.000000);
rgb(362.000000 pt)=(0.000000,0.952000,1.000000);
rgb(363.000000 pt)=(0.000000,0.956000,1.000000);
rgb(364.000000 pt)=(0.000000,0.960000,1.000000);
rgb(365.000000 pt)=(0.000000,0.964000,1.000000);
rgb(366.000000 pt)=(0.000000,0.968000,1.000000);
rgb(367.000000 pt)=(0.000000,0.972000,1.000000);
rgb(368.000000 pt)=(0.000000,0.976000,1.000000);
rgb(369.000000 pt)=(0.000000,0.980000,1.000000);
rgb(370.000000 pt)=(0.000000,0.984000,1.000000);
rgb(371.000000 pt)=(0.000000,0.988000,1.000000);
rgb(372.000000 pt)=(0.000000,0.992000,1.000000);
rgb(373.000000 pt)=(0.000000,0.996000,1.000000);
rgb(374.000000 pt)=(0.000000,1.000000,1.000000);
rgb(375.000000 pt)=(0.004000,1.000000,0.996000);
rgb(376.000000 pt)=(0.008000,1.000000,0.992000);
rgb(377.000000 pt)=(0.012000,1.000000,0.988000);
rgb(378.000000 pt)=(0.016000,1.000000,0.984000);
rgb(379.000000 pt)=(0.020000,1.000000,0.980000);
rgb(380.000000 pt)=(0.024000,1.000000,0.976000);
rgb(381.000000 pt)=(0.028000,1.000000,0.972000);
rgb(382.000000 pt)=(0.032000,1.000000,0.968000);
rgb(383.000000 pt)=(0.036000,1.000000,0.964000);
rgb(384.000000 pt)=(0.040000,1.000000,0.960000);
rgb(385.000000 pt)=(0.044000,1.000000,0.956000);
rgb(386.000000 pt)=(0.048000,1.000000,0.952000);
rgb(387.000000 pt)=(0.052000,1.000000,0.948000);
rgb(388.000000 pt)=(0.056000,1.000000,0.944000);
rgb(389.000000 pt)=(0.060000,1.000000,0.940000);
rgb(390.000000 pt)=(0.064000,1.000000,0.936000);
rgb(391.000000 pt)=(0.068000,1.000000,0.932000);
rgb(392.000000 pt)=(0.072000,1.000000,0.928000);
rgb(393.000000 pt)=(0.076000,1.000000,0.924000);
rgb(394.000000 pt)=(0.080000,1.000000,0.920000);
rgb(395.000000 pt)=(0.084000,1.000000,0.916000);
rgb(396.000000 pt)=(0.088000,1.000000,0.912000);
rgb(397.000000 pt)=(0.092000,1.000000,0.908000);
rgb(398.000000 pt)=(0.096000,1.000000,0.904000);
rgb(399.000000 pt)=(0.100000,1.000000,0.900000);
rgb(400.000000 pt)=(0.104000,1.000000,0.896000);
rgb(401.000000 pt)=(0.108000,1.000000,0.892000);
rgb(402.000000 pt)=(0.112000,1.000000,0.888000);
rgb(403.000000 pt)=(0.116000,1.000000,0.884000);
rgb(404.000000 pt)=(0.120000,1.000000,0.880000);
rgb(405.000000 pt)=(0.124000,1.000000,0.876000);
rgb(406.000000 pt)=(0.128000,1.000000,0.872000);
rgb(407.000000 pt)=(0.132000,1.000000,0.868000);
rgb(408.000000 pt)=(0.136000,1.000000,0.864000);
rgb(409.000000 pt)=(0.140000,1.000000,0.860000);
rgb(410.000000 pt)=(0.144000,1.000000,0.856000);
rgb(411.000000 pt)=(0.148000,1.000000,0.852000);
rgb(412.000000 pt)=(0.152000,1.000000,0.848000);
rgb(413.000000 pt)=(0.156000,1.000000,0.844000);
rgb(414.000000 pt)=(0.160000,1.000000,0.840000);
rgb(415.000000 pt)=(0.164000,1.000000,0.836000);
rgb(416.000000 pt)=(0.168000,1.000000,0.832000);
rgb(417.000000 pt)=(0.172000,1.000000,0.828000);
rgb(418.000000 pt)=(0.176000,1.000000,0.824000);
rgb(419.000000 pt)=(0.180000,1.000000,0.820000);
rgb(420.000000 pt)=(0.184000,1.000000,0.816000);
rgb(421.000000 pt)=(0.188000,1.000000,0.812000);
rgb(422.000000 pt)=(0.192000,1.000000,0.808000);
rgb(423.000000 pt)=(0.196000,1.000000,0.804000);
rgb(424.000000 pt)=(0.200000,1.000000,0.800000);
rgb(425.000000 pt)=(0.204000,1.000000,0.796000);
rgb(426.000000 pt)=(0.208000,1.000000,0.792000);
rgb(427.000000 pt)=(0.212000,1.000000,0.788000);
rgb(428.000000 pt)=(0.216000,1.000000,0.784000);
rgb(429.000000 pt)=(0.220000,1.000000,0.780000);
rgb(430.000000 pt)=(0.224000,1.000000,0.776000);
rgb(431.000000 pt)=(0.228000,1.000000,0.772000);
rgb(432.000000 pt)=(0.232000,1.000000,0.768000);
rgb(433.000000 pt)=(0.236000,1.000000,0.764000);
rgb(434.000000 pt)=(0.240000,1.000000,0.760000);
rgb(435.000000 pt)=(0.244000,1.000000,0.756000);
rgb(436.000000 pt)=(0.248000,1.000000,0.752000);
rgb(437.000000 pt)=(0.252000,1.000000,0.748000);
rgb(438.000000 pt)=(0.256000,1.000000,0.744000);
rgb(439.000000 pt)=(0.260000,1.000000,0.740000);
rgb(440.000000 pt)=(0.264000,1.000000,0.736000);
rgb(441.000000 pt)=(0.268000,1.000000,0.732000);
rgb(442.000000 pt)=(0.272000,1.000000,0.728000);
rgb(443.000000 pt)=(0.276000,1.000000,0.724000);
rgb(444.000000 pt)=(0.280000,1.000000,0.720000);
rgb(445.000000 pt)=(0.284000,1.000000,0.716000);
rgb(446.000000 pt)=(0.288000,1.000000,0.712000);
rgb(447.000000 pt)=(0.292000,1.000000,0.708000);
rgb(448.000000 pt)=(0.296000,1.000000,0.704000);
rgb(449.000000 pt)=(0.300000,1.000000,0.700000);
rgb(450.000000 pt)=(0.304000,1.000000,0.696000);
rgb(451.000000 pt)=(0.308000,1.000000,0.692000);
rgb(452.000000 pt)=(0.312000,1.000000,0.688000);
rgb(453.000000 pt)=(0.316000,1.000000,0.684000);
rgb(454.000000 pt)=(0.320000,1.000000,0.680000);
rgb(455.000000 pt)=(0.324000,1.000000,0.676000);
rgb(456.000000 pt)=(0.328000,1.000000,0.672000);
rgb(457.000000 pt)=(0.332000,1.000000,0.668000);
rgb(458.000000 pt)=(0.336000,1.000000,0.664000);
rgb(459.000000 pt)=(0.340000,1.000000,0.660000);
rgb(460.000000 pt)=(0.344000,1.000000,0.656000);
rgb(461.000000 pt)=(0.348000,1.000000,0.652000);
rgb(462.000000 pt)=(0.352000,1.000000,0.648000);
rgb(463.000000 pt)=(0.356000,1.000000,0.644000);
rgb(464.000000 pt)=(0.360000,1.000000,0.640000);
rgb(465.000000 pt)=(0.364000,1.000000,0.636000);
rgb(466.000000 pt)=(0.368000,1.000000,0.632000);
rgb(467.000000 pt)=(0.372000,1.000000,0.628000);
rgb(468.000000 pt)=(0.376000,1.000000,0.624000);
rgb(469.000000 pt)=(0.380000,1.000000,0.620000);
rgb(470.000000 pt)=(0.384000,1.000000,0.616000);
rgb(471.000000 pt)=(0.388000,1.000000,0.612000);
rgb(472.000000 pt)=(0.392000,1.000000,0.608000);
rgb(473.000000 pt)=(0.396000,1.000000,0.604000);
rgb(474.000000 pt)=(0.400000,1.000000,0.600000);
rgb(475.000000 pt)=(0.404000,1.000000,0.596000);
rgb(476.000000 pt)=(0.408000,1.000000,0.592000);
rgb(477.000000 pt)=(0.412000,1.000000,0.588000);
rgb(478.000000 pt)=(0.416000,1.000000,0.584000);
rgb(479.000000 pt)=(0.420000,1.000000,0.580000);
rgb(480.000000 pt)=(0.424000,1.000000,0.576000);
rgb(481.000000 pt)=(0.428000,1.000000,0.572000);
rgb(482.000000 pt)=(0.432000,1.000000,0.568000);
rgb(483.000000 pt)=(0.436000,1.000000,0.564000);
rgb(484.000000 pt)=(0.440000,1.000000,0.560000);
rgb(485.000000 pt)=(0.444000,1.000000,0.556000);
rgb(486.000000 pt)=(0.448000,1.000000,0.552000);
rgb(487.000000 pt)=(0.452000,1.000000,0.548000);
rgb(488.000000 pt)=(0.456000,1.000000,0.544000);
rgb(489.000000 pt)=(0.460000,1.000000,0.540000);
rgb(490.000000 pt)=(0.464000,1.000000,0.536000);
rgb(491.000000 pt)=(0.468000,1.000000,0.532000);
rgb(492.000000 pt)=(0.472000,1.000000,0.528000);
rgb(493.000000 pt)=(0.476000,1.000000,0.524000);
rgb(494.000000 pt)=(0.480000,1.000000,0.520000);
rgb(495.000000 pt)=(0.484000,1.000000,0.516000);
rgb(496.000000 pt)=(0.488000,1.000000,0.512000);
rgb(497.000000 pt)=(0.492000,1.000000,0.508000);
rgb(498.000000 pt)=(0.496000,1.000000,0.504000);
rgb(499.000000 pt)=(0.500000,1.000000,0.500000);
rgb(500.000000 pt)=(0.504000,1.000000,0.496000);
rgb(501.000000 pt)=(0.508000,1.000000,0.492000);
rgb(502.000000 pt)=(0.512000,1.000000,0.488000);
rgb(503.000000 pt)=(0.516000,1.000000,0.484000);
rgb(504.000000 pt)=(0.520000,1.000000,0.480000);
rgb(505.000000 pt)=(0.524000,1.000000,0.476000);
rgb(506.000000 pt)=(0.528000,1.000000,0.472000);
rgb(507.000000 pt)=(0.532000,1.000000,0.468000);
rgb(508.000000 pt)=(0.536000,1.000000,0.464000);
rgb(509.000000 pt)=(0.540000,1.000000,0.460000);
rgb(510.000000 pt)=(0.544000,1.000000,0.456000);
rgb(511.000000 pt)=(0.548000,1.000000,0.452000);
rgb(512.000000 pt)=(0.552000,1.000000,0.448000);
rgb(513.000000 pt)=(0.556000,1.000000,0.444000);
rgb(514.000000 pt)=(0.560000,1.000000,0.440000);
rgb(515.000000 pt)=(0.564000,1.000000,0.436000);
rgb(516.000000 pt)=(0.568000,1.000000,0.432000);
rgb(517.000000 pt)=(0.572000,1.000000,0.428000);
rgb(518.000000 pt)=(0.576000,1.000000,0.424000);
rgb(519.000000 pt)=(0.580000,1.000000,0.420000);
rgb(520.000000 pt)=(0.584000,1.000000,0.416000);
rgb(521.000000 pt)=(0.588000,1.000000,0.412000);
rgb(522.000000 pt)=(0.592000,1.000000,0.408000);
rgb(523.000000 pt)=(0.596000,1.000000,0.404000);
rgb(524.000000 pt)=(0.600000,1.000000,0.400000);
rgb(525.000000 pt)=(0.604000,1.000000,0.396000);
rgb(526.000000 pt)=(0.608000,1.000000,0.392000);
rgb(527.000000 pt)=(0.612000,1.000000,0.388000);
rgb(528.000000 pt)=(0.616000,1.000000,0.384000);
rgb(529.000000 pt)=(0.620000,1.000000,0.380000);
rgb(530.000000 pt)=(0.624000,1.000000,0.376000);
rgb(531.000000 pt)=(0.628000,1.000000,0.372000);
rgb(532.000000 pt)=(0.632000,1.000000,0.368000);
rgb(533.000000 pt)=(0.636000,1.000000,0.364000);
rgb(534.000000 pt)=(0.640000,1.000000,0.360000);
rgb(535.000000 pt)=(0.644000,1.000000,0.356000);
rgb(536.000000 pt)=(0.648000,1.000000,0.352000);
rgb(537.000000 pt)=(0.652000,1.000000,0.348000);
rgb(538.000000 pt)=(0.656000,1.000000,0.344000);
rgb(539.000000 pt)=(0.660000,1.000000,0.340000);
rgb(540.000000 pt)=(0.664000,1.000000,0.336000);
rgb(541.000000 pt)=(0.668000,1.000000,0.332000);
rgb(542.000000 pt)=(0.672000,1.000000,0.328000);
rgb(543.000000 pt)=(0.676000,1.000000,0.324000);
rgb(544.000000 pt)=(0.680000,1.000000,0.320000);
rgb(545.000000 pt)=(0.684000,1.000000,0.316000);
rgb(546.000000 pt)=(0.688000,1.000000,0.312000);
rgb(547.000000 pt)=(0.692000,1.000000,0.308000);
rgb(548.000000 pt)=(0.696000,1.000000,0.304000);
rgb(549.000000 pt)=(0.700000,1.000000,0.300000);
rgb(550.000000 pt)=(0.704000,1.000000,0.296000);
rgb(551.000000 pt)=(0.708000,1.000000,0.292000);
rgb(552.000000 pt)=(0.712000,1.000000,0.288000);
rgb(553.000000 pt)=(0.716000,1.000000,0.284000);
rgb(554.000000 pt)=(0.720000,1.000000,0.280000);
rgb(555.000000 pt)=(0.724000,1.000000,0.276000);
rgb(556.000000 pt)=(0.728000,1.000000,0.272000);
rgb(557.000000 pt)=(0.732000,1.000000,0.268000);
rgb(558.000000 pt)=(0.736000,1.000000,0.264000);
rgb(559.000000 pt)=(0.740000,1.000000,0.260000);
rgb(560.000000 pt)=(0.744000,1.000000,0.256000);
rgb(561.000000 pt)=(0.748000,1.000000,0.252000);
rgb(562.000000 pt)=(0.752000,1.000000,0.248000);
rgb(563.000000 pt)=(0.756000,1.000000,0.244000);
rgb(564.000000 pt)=(0.760000,1.000000,0.240000);
rgb(565.000000 pt)=(0.764000,1.000000,0.236000);
rgb(566.000000 pt)=(0.768000,1.000000,0.232000);
rgb(567.000000 pt)=(0.772000,1.000000,0.228000);
rgb(568.000000 pt)=(0.776000,1.000000,0.224000);
rgb(569.000000 pt)=(0.780000,1.000000,0.220000);
rgb(570.000000 pt)=(0.784000,1.000000,0.216000);
rgb(571.000000 pt)=(0.788000,1.000000,0.212000);
rgb(572.000000 pt)=(0.792000,1.000000,0.208000);
rgb(573.000000 pt)=(0.796000,1.000000,0.204000);
rgb(574.000000 pt)=(0.800000,1.000000,0.200000);
rgb(575.000000 pt)=(0.804000,1.000000,0.196000);
rgb(576.000000 pt)=(0.808000,1.000000,0.192000);
rgb(577.000000 pt)=(0.812000,1.000000,0.188000);
rgb(578.000000 pt)=(0.816000,1.000000,0.184000);
rgb(579.000000 pt)=(0.820000,1.000000,0.180000);
rgb(580.000000 pt)=(0.824000,1.000000,0.176000);
rgb(581.000000 pt)=(0.828000,1.000000,0.172000);
rgb(582.000000 pt)=(0.832000,1.000000,0.168000);
rgb(583.000000 pt)=(0.836000,1.000000,0.164000);
rgb(584.000000 pt)=(0.840000,1.000000,0.160000);
rgb(585.000000 pt)=(0.844000,1.000000,0.156000);
rgb(586.000000 pt)=(0.848000,1.000000,0.152000);
rgb(587.000000 pt)=(0.852000,1.000000,0.148000);
rgb(588.000000 pt)=(0.856000,1.000000,0.144000);
rgb(589.000000 pt)=(0.860000,1.000000,0.140000);
rgb(590.000000 pt)=(0.864000,1.000000,0.136000);
rgb(591.000000 pt)=(0.868000,1.000000,0.132000);
rgb(592.000000 pt)=(0.872000,1.000000,0.128000);
rgb(593.000000 pt)=(0.876000,1.000000,0.124000);
rgb(594.000000 pt)=(0.880000,1.000000,0.120000);
rgb(595.000000 pt)=(0.884000,1.000000,0.116000);
rgb(596.000000 pt)=(0.888000,1.000000,0.112000);
rgb(597.000000 pt)=(0.892000,1.000000,0.108000);
rgb(598.000000 pt)=(0.896000,1.000000,0.104000);
rgb(599.000000 pt)=(0.900000,1.000000,0.100000);
rgb(600.000000 pt)=(0.904000,1.000000,0.096000);
rgb(601.000000 pt)=(0.908000,1.000000,0.092000);
rgb(602.000000 pt)=(0.912000,1.000000,0.088000);
rgb(603.000000 pt)=(0.916000,1.000000,0.084000);
rgb(604.000000 pt)=(0.920000,1.000000,0.080000);
rgb(605.000000 pt)=(0.924000,1.000000,0.076000);
rgb(606.000000 pt)=(0.928000,1.000000,0.072000);
rgb(607.000000 pt)=(0.932000,1.000000,0.068000);
rgb(608.000000 pt)=(0.936000,1.000000,0.064000);
rgb(609.000000 pt)=(0.940000,1.000000,0.060000);
rgb(610.000000 pt)=(0.944000,1.000000,0.056000);
rgb(611.000000 pt)=(0.948000,1.000000,0.052000);
rgb(612.000000 pt)=(0.952000,1.000000,0.048000);
rgb(613.000000 pt)=(0.956000,1.000000,0.044000);
rgb(614.000000 pt)=(0.960000,1.000000,0.040000);
rgb(615.000000 pt)=(0.964000,1.000000,0.036000);
rgb(616.000000 pt)=(0.968000,1.000000,0.032000);
rgb(617.000000 pt)=(0.972000,1.000000,0.028000);
rgb(618.000000 pt)=(0.976000,1.000000,0.024000);
rgb(619.000000 pt)=(0.980000,1.000000,0.020000);
rgb(620.000000 pt)=(0.984000,1.000000,0.016000);
rgb(621.000000 pt)=(0.988000,1.000000,0.012000);
rgb(622.000000 pt)=(0.992000,1.000000,0.008000);
rgb(623.000000 pt)=(0.996000,1.000000,0.004000);
rgb(624.000000 pt)=(1.000000,1.000000,0.000000);
rgb(625.000000 pt)=(1.000000,0.996000,0.000000);
rgb(626.000000 pt)=(1.000000,0.992000,0.000000);
rgb(627.000000 pt)=(1.000000,0.988000,0.000000);
rgb(628.000000 pt)=(1.000000,0.984000,0.000000);
rgb(629.000000 pt)=(1.000000,0.980000,0.000000);
rgb(630.000000 pt)=(1.000000,0.976000,0.000000);
rgb(631.000000 pt)=(1.000000,0.972000,0.000000);
rgb(632.000000 pt)=(1.000000,0.968000,0.000000);
rgb(633.000000 pt)=(1.000000,0.964000,0.000000);
rgb(634.000000 pt)=(1.000000,0.960000,0.000000);
rgb(635.000000 pt)=(1.000000,0.956000,0.000000);
rgb(636.000000 pt)=(1.000000,0.952000,0.000000);
rgb(637.000000 pt)=(1.000000,0.948000,0.000000);
rgb(638.000000 pt)=(1.000000,0.944000,0.000000);
rgb(639.000000 pt)=(1.000000,0.940000,0.000000);
rgb(640.000000 pt)=(1.000000,0.936000,0.000000);
rgb(641.000000 pt)=(1.000000,0.932000,0.000000);
rgb(642.000000 pt)=(1.000000,0.928000,0.000000);
rgb(643.000000 pt)=(1.000000,0.924000,0.000000);
rgb(644.000000 pt)=(1.000000,0.920000,0.000000);
rgb(645.000000 pt)=(1.000000,0.916000,0.000000);
rgb(646.000000 pt)=(1.000000,0.912000,0.000000);
rgb(647.000000 pt)=(1.000000,0.908000,0.000000);
rgb(648.000000 pt)=(1.000000,0.904000,0.000000);
rgb(649.000000 pt)=(1.000000,0.900000,0.000000);
rgb(650.000000 pt)=(1.000000,0.896000,0.000000);
rgb(651.000000 pt)=(1.000000,0.892000,0.000000);
rgb(652.000000 pt)=(1.000000,0.888000,0.000000);
rgb(653.000000 pt)=(1.000000,0.884000,0.000000);
rgb(654.000000 pt)=(1.000000,0.880000,0.000000);
rgb(655.000000 pt)=(1.000000,0.876000,0.000000);
rgb(656.000000 pt)=(1.000000,0.872000,0.000000);
rgb(657.000000 pt)=(1.000000,0.868000,0.000000);
rgb(658.000000 pt)=(1.000000,0.864000,0.000000);
rgb(659.000000 pt)=(1.000000,0.860000,0.000000);
rgb(660.000000 pt)=(1.000000,0.856000,0.000000);
rgb(661.000000 pt)=(1.000000,0.852000,0.000000);
rgb(662.000000 pt)=(1.000000,0.848000,0.000000);
rgb(663.000000 pt)=(1.000000,0.844000,0.000000);
rgb(664.000000 pt)=(1.000000,0.840000,0.000000);
rgb(665.000000 pt)=(1.000000,0.836000,0.000000);
rgb(666.000000 pt)=(1.000000,0.832000,0.000000);
rgb(667.000000 pt)=(1.000000,0.828000,0.000000);
rgb(668.000000 pt)=(1.000000,0.824000,0.000000);
rgb(669.000000 pt)=(1.000000,0.820000,0.000000);
rgb(670.000000 pt)=(1.000000,0.816000,0.000000);
rgb(671.000000 pt)=(1.000000,0.812000,0.000000);
rgb(672.000000 pt)=(1.000000,0.808000,0.000000);
rgb(673.000000 pt)=(1.000000,0.804000,0.000000);
rgb(674.000000 pt)=(1.000000,0.800000,0.000000);
rgb(675.000000 pt)=(1.000000,0.796000,0.000000);
rgb(676.000000 pt)=(1.000000,0.792000,0.000000);
rgb(677.000000 pt)=(1.000000,0.788000,0.000000);
rgb(678.000000 pt)=(1.000000,0.784000,0.000000);
rgb(679.000000 pt)=(1.000000,0.780000,0.000000);
rgb(680.000000 pt)=(1.000000,0.776000,0.000000);
rgb(681.000000 pt)=(1.000000,0.772000,0.000000);
rgb(682.000000 pt)=(1.000000,0.768000,0.000000);
rgb(683.000000 pt)=(1.000000,0.764000,0.000000);
rgb(684.000000 pt)=(1.000000,0.760000,0.000000);
rgb(685.000000 pt)=(1.000000,0.756000,0.000000);
rgb(686.000000 pt)=(1.000000,0.752000,0.000000);
rgb(687.000000 pt)=(1.000000,0.748000,0.000000);
rgb(688.000000 pt)=(1.000000,0.744000,0.000000);
rgb(689.000000 pt)=(1.000000,0.740000,0.000000);
rgb(690.000000 pt)=(1.000000,0.736000,0.000000);
rgb(691.000000 pt)=(1.000000,0.732000,0.000000);
rgb(692.000000 pt)=(1.000000,0.728000,0.000000);
rgb(693.000000 pt)=(1.000000,0.724000,0.000000);
rgb(694.000000 pt)=(1.000000,0.720000,0.000000);
rgb(695.000000 pt)=(1.000000,0.716000,0.000000);
rgb(696.000000 pt)=(1.000000,0.712000,0.000000);
rgb(697.000000 pt)=(1.000000,0.708000,0.000000);
rgb(698.000000 pt)=(1.000000,0.704000,0.000000);
rgb(699.000000 pt)=(1.000000,0.700000,0.000000);
rgb(700.000000 pt)=(1.000000,0.696000,0.000000);
rgb(701.000000 pt)=(1.000000,0.692000,0.000000);
rgb(702.000000 pt)=(1.000000,0.688000,0.000000);
rgb(703.000000 pt)=(1.000000,0.684000,0.000000);
rgb(704.000000 pt)=(1.000000,0.680000,0.000000);
rgb(705.000000 pt)=(1.000000,0.676000,0.000000);
rgb(706.000000 pt)=(1.000000,0.672000,0.000000);
rgb(707.000000 pt)=(1.000000,0.668000,0.000000);
rgb(708.000000 pt)=(1.000000,0.664000,0.000000);
rgb(709.000000 pt)=(1.000000,0.660000,0.000000);
rgb(710.000000 pt)=(1.000000,0.656000,0.000000);
rgb(711.000000 pt)=(1.000000,0.652000,0.000000);
rgb(712.000000 pt)=(1.000000,0.648000,0.000000);
rgb(713.000000 pt)=(1.000000,0.644000,0.000000);
rgb(714.000000 pt)=(1.000000,0.640000,0.000000);
rgb(715.000000 pt)=(1.000000,0.636000,0.000000);
rgb(716.000000 pt)=(1.000000,0.632000,0.000000);
rgb(717.000000 pt)=(1.000000,0.628000,0.000000);
rgb(718.000000 pt)=(1.000000,0.624000,0.000000);
rgb(719.000000 pt)=(1.000000,0.620000,0.000000);
rgb(720.000000 pt)=(1.000000,0.616000,0.000000);
rgb(721.000000 pt)=(1.000000,0.612000,0.000000);
rgb(722.000000 pt)=(1.000000,0.608000,0.000000);
rgb(723.000000 pt)=(1.000000,0.604000,0.000000);
rgb(724.000000 pt)=(1.000000,0.600000,0.000000);
rgb(725.000000 pt)=(1.000000,0.596000,0.000000);
rgb(726.000000 pt)=(1.000000,0.592000,0.000000);
rgb(727.000000 pt)=(1.000000,0.588000,0.000000);
rgb(728.000000 pt)=(1.000000,0.584000,0.000000);
rgb(729.000000 pt)=(1.000000,0.580000,0.000000);
rgb(730.000000 pt)=(1.000000,0.576000,0.000000);
rgb(731.000000 pt)=(1.000000,0.572000,0.000000);
rgb(732.000000 pt)=(1.000000,0.568000,0.000000);
rgb(733.000000 pt)=(1.000000,0.564000,0.000000);
rgb(734.000000 pt)=(1.000000,0.560000,0.000000);
rgb(735.000000 pt)=(1.000000,0.556000,0.000000);
rgb(736.000000 pt)=(1.000000,0.552000,0.000000);
rgb(737.000000 pt)=(1.000000,0.548000,0.000000);
rgb(738.000000 pt)=(1.000000,0.544000,0.000000);
rgb(739.000000 pt)=(1.000000,0.540000,0.000000);
rgb(740.000000 pt)=(1.000000,0.536000,0.000000);
rgb(741.000000 pt)=(1.000000,0.532000,0.000000);
rgb(742.000000 pt)=(1.000000,0.528000,0.000000);
rgb(743.000000 pt)=(1.000000,0.524000,0.000000);
rgb(744.000000 pt)=(1.000000,0.520000,0.000000);
rgb(745.000000 pt)=(1.000000,0.516000,0.000000);
rgb(746.000000 pt)=(1.000000,0.512000,0.000000);
rgb(747.000000 pt)=(1.000000,0.508000,0.000000);
rgb(748.000000 pt)=(1.000000,0.504000,0.000000);
rgb(749.000000 pt)=(1.000000,0.500000,0.000000);
rgb(750.000000 pt)=(1.000000,0.496000,0.000000);
rgb(751.000000 pt)=(1.000000,0.492000,0.000000);
rgb(752.000000 pt)=(1.000000,0.488000,0.000000);
rgb(753.000000 pt)=(1.000000,0.484000,0.000000);
rgb(754.000000 pt)=(1.000000,0.480000,0.000000);
rgb(755.000000 pt)=(1.000000,0.476000,0.000000);
rgb(756.000000 pt)=(1.000000,0.472000,0.000000);
rgb(757.000000 pt)=(1.000000,0.468000,0.000000);
rgb(758.000000 pt)=(1.000000,0.464000,0.000000);
rgb(759.000000 pt)=(1.000000,0.460000,0.000000);
rgb(760.000000 pt)=(1.000000,0.456000,0.000000);
rgb(761.000000 pt)=(1.000000,0.452000,0.000000);
rgb(762.000000 pt)=(1.000000,0.448000,0.000000);
rgb(763.000000 pt)=(1.000000,0.444000,0.000000);
rgb(764.000000 pt)=(1.000000,0.440000,0.000000);
rgb(765.000000 pt)=(1.000000,0.436000,0.000000);
rgb(766.000000 pt)=(1.000000,0.432000,0.000000);
rgb(767.000000 pt)=(1.000000,0.428000,0.000000);
rgb(768.000000 pt)=(1.000000,0.424000,0.000000);
rgb(769.000000 pt)=(1.000000,0.420000,0.000000);
rgb(770.000000 pt)=(1.000000,0.416000,0.000000);
rgb(771.000000 pt)=(1.000000,0.412000,0.000000);
rgb(772.000000 pt)=(1.000000,0.408000,0.000000);
rgb(773.000000 pt)=(1.000000,0.404000,0.000000);
rgb(774.000000 pt)=(1.000000,0.400000,0.000000);
rgb(775.000000 pt)=(1.000000,0.396000,0.000000);
rgb(776.000000 pt)=(1.000000,0.392000,0.000000);
rgb(777.000000 pt)=(1.000000,0.388000,0.000000);
rgb(778.000000 pt)=(1.000000,0.384000,0.000000);
rgb(779.000000 pt)=(1.000000,0.380000,0.000000);
rgb(780.000000 pt)=(1.000000,0.376000,0.000000);
rgb(781.000000 pt)=(1.000000,0.372000,0.000000);
rgb(782.000000 pt)=(1.000000,0.368000,0.000000);
rgb(783.000000 pt)=(1.000000,0.364000,0.000000);
rgb(784.000000 pt)=(1.000000,0.360000,0.000000);
rgb(785.000000 pt)=(1.000000,0.356000,0.000000);
rgb(786.000000 pt)=(1.000000,0.352000,0.000000);
rgb(787.000000 pt)=(1.000000,0.348000,0.000000);
rgb(788.000000 pt)=(1.000000,0.344000,0.000000);
rgb(789.000000 pt)=(1.000000,0.340000,0.000000);
rgb(790.000000 pt)=(1.000000,0.336000,0.000000);
rgb(791.000000 pt)=(1.000000,0.332000,0.000000);
rgb(792.000000 pt)=(1.000000,0.328000,0.000000);
rgb(793.000000 pt)=(1.000000,0.324000,0.000000);
rgb(794.000000 pt)=(1.000000,0.320000,0.000000);
rgb(795.000000 pt)=(1.000000,0.316000,0.000000);
rgb(796.000000 pt)=(1.000000,0.312000,0.000000);
rgb(797.000000 pt)=(1.000000,0.308000,0.000000);
rgb(798.000000 pt)=(1.000000,0.304000,0.000000);
rgb(799.000000 pt)=(1.000000,0.300000,0.000000);
rgb(800.000000 pt)=(1.000000,0.296000,0.000000);
rgb(801.000000 pt)=(1.000000,0.292000,0.000000);
rgb(802.000000 pt)=(1.000000,0.288000,0.000000);
rgb(803.000000 pt)=(1.000000,0.284000,0.000000);
rgb(804.000000 pt)=(1.000000,0.280000,0.000000);
rgb(805.000000 pt)=(1.000000,0.276000,0.000000);
rgb(806.000000 pt)=(1.000000,0.272000,0.000000);
rgb(807.000000 pt)=(1.000000,0.268000,0.000000);
rgb(808.000000 pt)=(1.000000,0.264000,0.000000);
rgb(809.000000 pt)=(1.000000,0.260000,0.000000);
rgb(810.000000 pt)=(1.000000,0.256000,0.000000);
rgb(811.000000 pt)=(1.000000,0.252000,0.000000);
rgb(812.000000 pt)=(1.000000,0.248000,0.000000);
rgb(813.000000 pt)=(1.000000,0.244000,0.000000);
rgb(814.000000 pt)=(1.000000,0.240000,0.000000);
rgb(815.000000 pt)=(1.000000,0.236000,0.000000);
rgb(816.000000 pt)=(1.000000,0.232000,0.000000);
rgb(817.000000 pt)=(1.000000,0.228000,0.000000);
rgb(818.000000 pt)=(1.000000,0.224000,0.000000);
rgb(819.000000 pt)=(1.000000,0.220000,0.000000);
rgb(820.000000 pt)=(1.000000,0.216000,0.000000);
rgb(821.000000 pt)=(1.000000,0.212000,0.000000);
rgb(822.000000 pt)=(1.000000,0.208000,0.000000);
rgb(823.000000 pt)=(1.000000,0.204000,0.000000);
rgb(824.000000 pt)=(1.000000,0.200000,0.000000);
rgb(825.000000 pt)=(1.000000,0.196000,0.000000);
rgb(826.000000 pt)=(1.000000,0.192000,0.000000);
rgb(827.000000 pt)=(1.000000,0.188000,0.000000);
rgb(828.000000 pt)=(1.000000,0.184000,0.000000);
rgb(829.000000 pt)=(1.000000,0.180000,0.000000);
rgb(830.000000 pt)=(1.000000,0.176000,0.000000);
rgb(831.000000 pt)=(1.000000,0.172000,0.000000);
rgb(832.000000 pt)=(1.000000,0.168000,0.000000);
rgb(833.000000 pt)=(1.000000,0.164000,0.000000);
rgb(834.000000 pt)=(1.000000,0.160000,0.000000);
rgb(835.000000 pt)=(1.000000,0.156000,0.000000);
rgb(836.000000 pt)=(1.000000,0.152000,0.000000);
rgb(837.000000 pt)=(1.000000,0.148000,0.000000);
rgb(838.000000 pt)=(1.000000,0.144000,0.000000);
rgb(839.000000 pt)=(1.000000,0.140000,0.000000);
rgb(840.000000 pt)=(1.000000,0.136000,0.000000);
rgb(841.000000 pt)=(1.000000,0.132000,0.000000);
rgb(842.000000 pt)=(1.000000,0.128000,0.000000);
rgb(843.000000 pt)=(1.000000,0.124000,0.000000);
rgb(844.000000 pt)=(1.000000,0.120000,0.000000);
rgb(845.000000 pt)=(1.000000,0.116000,0.000000);
rgb(846.000000 pt)=(1.000000,0.112000,0.000000);
rgb(847.000000 pt)=(1.000000,0.108000,0.000000);
rgb(848.000000 pt)=(1.000000,0.104000,0.000000);
rgb(849.000000 pt)=(1.000000,0.100000,0.000000);
rgb(850.000000 pt)=(1.000000,0.096000,0.000000);
rgb(851.000000 pt)=(1.000000,0.092000,0.000000);
rgb(852.000000 pt)=(1.000000,0.088000,0.000000);
rgb(853.000000 pt)=(1.000000,0.084000,0.000000);
rgb(854.000000 pt)=(1.000000,0.080000,0.000000);
rgb(855.000000 pt)=(1.000000,0.076000,0.000000);
rgb(856.000000 pt)=(1.000000,0.072000,0.000000);
rgb(857.000000 pt)=(1.000000,0.068000,0.000000);
rgb(858.000000 pt)=(1.000000,0.064000,0.000000);
rgb(859.000000 pt)=(1.000000,0.060000,0.000000);
rgb(860.000000 pt)=(1.000000,0.056000,0.000000);
rgb(861.000000 pt)=(1.000000,0.052000,0.000000);
rgb(862.000000 pt)=(1.000000,0.048000,0.000000);
rgb(863.000000 pt)=(1.000000,0.044000,0.000000);
rgb(864.000000 pt)=(1.000000,0.040000,0.000000);
rgb(865.000000 pt)=(1.000000,0.036000,0.000000);
rgb(866.000000 pt)=(1.000000,0.032000,0.000000);
rgb(867.000000 pt)=(1.000000,0.028000,0.000000);
rgb(868.000000 pt)=(1.000000,0.024000,0.000000);
rgb(869.000000 pt)=(1.000000,0.020000,0.000000);
rgb(870.000000 pt)=(1.000000,0.016000,0.000000);
rgb(871.000000 pt)=(1.000000,0.012000,0.000000);
rgb(872.000000 pt)=(1.000000,0.008000,0.000000);
rgb(873.000000 pt)=(1.000000,0.004000,0.000000);
rgb(874.000000 pt)=(1.000000,0.000000,0.000000);
rgb(875.000000 pt)=(0.996000,0.000000,0.000000);
rgb(876.000000 pt)=(0.992000,0.000000,0.000000);
rgb(877.000000 pt)=(0.988000,0.000000,0.000000);
rgb(878.000000 pt)=(0.984000,0.000000,0.000000);
rgb(879.000000 pt)=(0.980000,0.000000,0.000000);
rgb(880.000000 pt)=(0.976000,0.000000,0.000000);
rgb(881.000000 pt)=(0.972000,0.000000,0.000000);
rgb(882.000000 pt)=(0.968000,0.000000,0.000000);
rgb(883.000000 pt)=(0.964000,0.000000,0.000000);
rgb(884.000000 pt)=(0.960000,0.000000,0.000000);
rgb(885.000000 pt)=(0.956000,0.000000,0.000000);
rgb(886.000000 pt)=(0.952000,0.000000,0.000000);
rgb(887.000000 pt)=(0.948000,0.000000,0.000000);
rgb(888.000000 pt)=(0.944000,0.000000,0.000000);
rgb(889.000000 pt)=(0.940000,0.000000,0.000000);
rgb(890.000000 pt)=(0.936000,0.000000,0.000000);
rgb(891.000000 pt)=(0.932000,0.000000,0.000000);
rgb(892.000000 pt)=(0.928000,0.000000,0.000000);
rgb(893.000000 pt)=(0.924000,0.000000,0.000000);
rgb(894.000000 pt)=(0.920000,0.000000,0.000000);
rgb(895.000000 pt)=(0.916000,0.000000,0.000000);
rgb(896.000000 pt)=(0.912000,0.000000,0.000000);
rgb(897.000000 pt)=(0.908000,0.000000,0.000000);
rgb(898.000000 pt)=(0.904000,0.000000,0.000000);
rgb(899.000000 pt)=(0.900000,0.000000,0.000000);
rgb(900.000000 pt)=(0.896000,0.000000,0.000000);
rgb(901.000000 pt)=(0.892000,0.000000,0.000000);
rgb(902.000000 pt)=(0.888000,0.000000,0.000000);
rgb(903.000000 pt)=(0.884000,0.000000,0.000000);
rgb(904.000000 pt)=(0.880000,0.000000,0.000000);
rgb(905.000000 pt)=(0.876000,0.000000,0.000000);
rgb(906.000000 pt)=(0.872000,0.000000,0.000000);
rgb(907.000000 pt)=(0.868000,0.000000,0.000000);
rgb(908.000000 pt)=(0.864000,0.000000,0.000000);
rgb(909.000000 pt)=(0.860000,0.000000,0.000000);
rgb(910.000000 pt)=(0.856000,0.000000,0.000000);
rgb(911.000000 pt)=(0.852000,0.000000,0.000000);
rgb(912.000000 pt)=(0.848000,0.000000,0.000000);
rgb(913.000000 pt)=(0.844000,0.000000,0.000000);
rgb(914.000000 pt)=(0.840000,0.000000,0.000000);
rgb(915.000000 pt)=(0.836000,0.000000,0.000000);
rgb(916.000000 pt)=(0.832000,0.000000,0.000000);
rgb(917.000000 pt)=(0.828000,0.000000,0.000000);
rgb(918.000000 pt)=(0.824000,0.000000,0.000000);
rgb(919.000000 pt)=(0.820000,0.000000,0.000000);
rgb(920.000000 pt)=(0.816000,0.000000,0.000000);
rgb(921.000000 pt)=(0.812000,0.000000,0.000000);
rgb(922.000000 pt)=(0.808000,0.000000,0.000000);
rgb(923.000000 pt)=(0.804000,0.000000,0.000000);
rgb(924.000000 pt)=(0.800000,0.000000,0.000000);
rgb(925.000000 pt)=(0.796000,0.000000,0.000000);
rgb(926.000000 pt)=(0.792000,0.000000,0.000000);
rgb(927.000000 pt)=(0.788000,0.000000,0.000000);
rgb(928.000000 pt)=(0.784000,0.000000,0.000000);
rgb(929.000000 pt)=(0.780000,0.000000,0.000000);
rgb(930.000000 pt)=(0.776000,0.000000,0.000000);
rgb(931.000000 pt)=(0.772000,0.000000,0.000000);
rgb(932.000000 pt)=(0.768000,0.000000,0.000000);
rgb(933.000000 pt)=(0.764000,0.000000,0.000000);
rgb(934.000000 pt)=(0.760000,0.000000,0.000000);
rgb(935.000000 pt)=(0.756000,0.000000,0.000000);
rgb(936.000000 pt)=(0.752000,0.000000,0.000000);
rgb(937.000000 pt)=(0.748000,0.000000,0.000000);
rgb(938.000000 pt)=(0.744000,0.000000,0.000000);
rgb(939.000000 pt)=(0.740000,0.000000,0.000000);
rgb(940.000000 pt)=(0.736000,0.000000,0.000000);
rgb(941.000000 pt)=(0.732000,0.000000,0.000000);
rgb(942.000000 pt)=(0.728000,0.000000,0.000000);
rgb(943.000000 pt)=(0.724000,0.000000,0.000000);
rgb(944.000000 pt)=(0.720000,0.000000,0.000000);
rgb(945.000000 pt)=(0.716000,0.000000,0.000000);
rgb(946.000000 pt)=(0.712000,0.000000,0.000000);
rgb(947.000000 pt)=(0.708000,0.000000,0.000000);
rgb(948.000000 pt)=(0.704000,0.000000,0.000000);
rgb(949.000000 pt)=(0.700000,0.000000,0.000000);
rgb(950.000000 pt)=(0.696000,0.000000,0.000000);
rgb(951.000000 pt)=(0.692000,0.000000,0.000000);
rgb(952.000000 pt)=(0.688000,0.000000,0.000000);
rgb(953.000000 pt)=(0.684000,0.000000,0.000000);
rgb(954.000000 pt)=(0.680000,0.000000,0.000000);
rgb(955.000000 pt)=(0.676000,0.000000,0.000000);
rgb(956.000000 pt)=(0.672000,0.000000,0.000000);
rgb(957.000000 pt)=(0.668000,0.000000,0.000000);
rgb(958.000000 pt)=(0.664000,0.000000,0.000000);
rgb(959.000000 pt)=(0.660000,0.000000,0.000000);
rgb(960.000000 pt)=(0.656000,0.000000,0.000000);
rgb(961.000000 pt)=(0.652000,0.000000,0.000000);
rgb(962.000000 pt)=(0.648000,0.000000,0.000000);
rgb(963.000000 pt)=(0.644000,0.000000,0.000000);
rgb(964.000000 pt)=(0.640000,0.000000,0.000000);
rgb(965.000000 pt)=(0.636000,0.000000,0.000000);
rgb(966.000000 pt)=(0.632000,0.000000,0.000000);
rgb(967.000000 pt)=(0.628000,0.000000,0.000000);
rgb(968.000000 pt)=(0.624000,0.000000,0.000000);
rgb(969.000000 pt)=(0.620000,0.000000,0.000000);
rgb(970.000000 pt)=(0.616000,0.000000,0.000000);
rgb(971.000000 pt)=(0.612000,0.000000,0.000000);
rgb(972.000000 pt)=(0.608000,0.000000,0.000000);
rgb(973.000000 pt)=(0.604000,0.000000,0.000000);
rgb(974.000000 pt)=(0.600000,0.000000,0.000000);
rgb(975.000000 pt)=(0.596000,0.000000,0.000000);
rgb(976.000000 pt)=(0.592000,0.000000,0.000000);
rgb(977.000000 pt)=(0.588000,0.000000,0.000000);
rgb(978.000000 pt)=(0.584000,0.000000,0.000000);
rgb(979.000000 pt)=(0.580000,0.000000,0.000000);
rgb(980.000000 pt)=(0.576000,0.000000,0.000000);
rgb(981.000000 pt)=(0.572000,0.000000,0.000000);
rgb(982.000000 pt)=(0.568000,0.000000,0.000000);
rgb(983.000000 pt)=(0.564000,0.000000,0.000000);
rgb(984.000000 pt)=(0.560000,0.000000,0.000000);
rgb(985.000000 pt)=(0.556000,0.000000,0.000000);
rgb(986.000000 pt)=(0.552000,0.000000,0.000000);
rgb(987.000000 pt)=(0.548000,0.000000,0.000000);
rgb(988.000000 pt)=(0.544000,0.000000,0.000000);
rgb(989.000000 pt)=(0.540000,0.000000,0.000000);
rgb(990.000000 pt)=(0.536000,0.000000,0.000000);
rgb(991.000000 pt)=(0.532000,0.000000,0.000000);
rgb(992.000000 pt)=(0.528000,0.000000,0.000000);
rgb(993.000000 pt)=(0.524000,0.000000,0.000000);
rgb(994.000000 pt)=(0.520000,0.000000,0.000000);
rgb(995.000000 pt)=(0.516000,0.000000,0.000000);
rgb(996.000000 pt)=(0.512000,0.000000,0.000000);
rgb(997.000000 pt)=(0.508000,0.000000,0.000000);
rgb(998.000000 pt)=(0.504000,0.000000,0.000000);
rgb(999.000000 pt)=(0.500000,0.000000,0.000000);
}}

\pgfplotsset{
	colormap={myhot}{
		rgb(0pt)=(1.000000,1.000000,1.000000);
		rgb(1pt)=(1.000000,1.000000,0.996000);
		rgb(2pt)=(1.000000,1.000000,0.992000);
		rgb(3pt)=(1.000000,1.000000,0.988000);
		rgb(4pt)=(1.000000,1.000000,0.984000);
		rgb(5pt)=(1.000000,1.000000,0.980000);
		rgb(6pt)=(1.000000,1.000000,0.976000);
		rgb(7pt)=(1.000000,1.000000,0.972000);
		rgb(8pt)=(1.000000,1.000000,0.968000);
		rgb(9pt)=(1.000000,1.000000,0.964000);
		rgb(10pt)=(1.000000,1.000000,0.960000);
		rgb(11pt)=(1.000000,1.000000,0.956000);
		rgb(12pt)=(1.000000,1.000000,0.952000);
		rgb(13pt)=(1.000000,1.000000,0.948000);
		rgb(14pt)=(1.000000,1.000000,0.944000);
		rgb(15pt)=(1.000000,1.000000,0.940000);
		rgb(16pt)=(1.000000,1.000000,0.936000);
		rgb(17pt)=(1.000000,1.000000,0.932000);
		rgb(18pt)=(1.000000,1.000000,0.928000);
		rgb(19pt)=(1.000000,1.000000,0.924000);
		rgb(20pt)=(1.000000,1.000000,0.920000);
		rgb(21pt)=(1.000000,1.000000,0.916000);
		rgb(22pt)=(1.000000,1.000000,0.912000);
		rgb(23pt)=(1.000000,1.000000,0.908000);
		rgb(24pt)=(1.000000,1.000000,0.904000);
		rgb(25pt)=(1.000000,1.000000,0.900000);
		rgb(26pt)=(1.000000,1.000000,0.896000);
		rgb(27pt)=(1.000000,1.000000,0.892000);
		rgb(28pt)=(1.000000,1.000000,0.888000);
		rgb(29pt)=(1.000000,1.000000,0.884000);
		rgb(30pt)=(1.000000,1.000000,0.880000);
		rgb(31pt)=(1.000000,1.000000,0.876000);
		rgb(32pt)=(1.000000,1.000000,0.872000);
		rgb(33pt)=(1.000000,1.000000,0.868000);
		rgb(34pt)=(1.000000,1.000000,0.864000);
		rgb(35pt)=(1.000000,1.000000,0.860000);
		rgb(36pt)=(1.000000,1.000000,0.856000);
		rgb(37pt)=(1.000000,1.000000,0.852000);
		rgb(38pt)=(1.000000,1.000000,0.848000);
		rgb(39pt)=(1.000000,1.000000,0.844000);
		rgb(40pt)=(1.000000,1.000000,0.840000);
		rgb(41pt)=(1.000000,1.000000,0.836000);
		rgb(42pt)=(1.000000,1.000000,0.832000);
		rgb(43pt)=(1.000000,1.000000,0.828000);
		rgb(44pt)=(1.000000,1.000000,0.824000);
		rgb(45pt)=(1.000000,1.000000,0.820000);
		rgb(46pt)=(1.000000,1.000000,0.816000);
		rgb(47pt)=(1.000000,1.000000,0.812000);
		rgb(48pt)=(1.000000,1.000000,0.808000);
		rgb(49pt)=(1.000000,1.000000,0.804000);
		rgb(50pt)=(1.000000,1.000000,0.800000);
		rgb(51pt)=(1.000000,1.000000,0.796000);
		rgb(52pt)=(1.000000,1.000000,0.792000);
		rgb(53pt)=(1.000000,1.000000,0.788000);
		rgb(54pt)=(1.000000,1.000000,0.784000);
		rgb(55pt)=(1.000000,1.000000,0.780000);
		rgb(56pt)=(1.000000,1.000000,0.776000);
		rgb(57pt)=(1.000000,1.000000,0.772000);
		rgb(58pt)=(1.000000,1.000000,0.768000);
		rgb(59pt)=(1.000000,1.000000,0.764000);
		rgb(60pt)=(1.000000,1.000000,0.760000);
		rgb(61pt)=(1.000000,1.000000,0.756000);
		rgb(62pt)=(1.000000,1.000000,0.752000);
		rgb(63pt)=(1.000000,1.000000,0.748000);
		rgb(64pt)=(1.000000,1.000000,0.744000);
		rgb(65pt)=(1.000000,1.000000,0.740000);
		rgb(66pt)=(1.000000,1.000000,0.736000);
		rgb(67pt)=(1.000000,1.000000,0.732000);
		rgb(68pt)=(1.000000,1.000000,0.728000);
		rgb(69pt)=(1.000000,1.000000,0.724000);
		rgb(70pt)=(1.000000,1.000000,0.720000);
		rgb(71pt)=(1.000000,1.000000,0.716000);
		rgb(72pt)=(1.000000,1.000000,0.712000);
		rgb(73pt)=(1.000000,1.000000,0.708000);
		rgb(74pt)=(1.000000,1.000000,0.704000);
		rgb(75pt)=(1.000000,1.000000,0.700000);
		rgb(76pt)=(1.000000,1.000000,0.696000);
		rgb(77pt)=(1.000000,1.000000,0.692000);
		rgb(78pt)=(1.000000,1.000000,0.688000);
		rgb(79pt)=(1.000000,1.000000,0.684000);
		rgb(80pt)=(1.000000,1.000000,0.680000);
		rgb(81pt)=(1.000000,1.000000,0.676000);
		rgb(82pt)=(1.000000,1.000000,0.672000);
		rgb(83pt)=(1.000000,1.000000,0.668000);
		rgb(84pt)=(1.000000,1.000000,0.664000);
		rgb(85pt)=(1.000000,1.000000,0.660000);
		rgb(86pt)=(1.000000,1.000000,0.656000);
		rgb(87pt)=(1.000000,1.000000,0.652000);
		rgb(88pt)=(1.000000,1.000000,0.648000);
		rgb(89pt)=(1.000000,1.000000,0.644000);
		rgb(90pt)=(1.000000,1.000000,0.640000);
		rgb(91pt)=(1.000000,1.000000,0.636000);
		rgb(92pt)=(1.000000,1.000000,0.632000);
		rgb(93pt)=(1.000000,1.000000,0.628000);
		rgb(94pt)=(1.000000,1.000000,0.624000);
		rgb(95pt)=(1.000000,1.000000,0.620000);
		rgb(96pt)=(1.000000,1.000000,0.616000);
		rgb(97pt)=(1.000000,1.000000,0.612000);
		rgb(98pt)=(1.000000,1.000000,0.608000);
		rgb(99pt)=(1.000000,1.000000,0.604000);
		rgb(100pt)=(1.000000,1.000000,0.600000);
		rgb(101pt)=(1.000000,1.000000,0.596000);
		rgb(102pt)=(1.000000,1.000000,0.592000);
		rgb(103pt)=(1.000000,1.000000,0.588000);
		rgb(104pt)=(1.000000,1.000000,0.584000);
		rgb(105pt)=(1.000000,1.000000,0.580000);
		rgb(106pt)=(1.000000,1.000000,0.576000);
		rgb(107pt)=(1.000000,1.000000,0.572000);
		rgb(108pt)=(1.000000,1.000000,0.568000);
		rgb(109pt)=(1.000000,1.000000,0.564000);
		rgb(110pt)=(1.000000,1.000000,0.560000);
		rgb(111pt)=(1.000000,1.000000,0.556000);
		rgb(112pt)=(1.000000,1.000000,0.552000);
		rgb(113pt)=(1.000000,1.000000,0.548000);
		rgb(114pt)=(1.000000,1.000000,0.544000);
		rgb(115pt)=(1.000000,1.000000,0.540000);
		rgb(116pt)=(1.000000,1.000000,0.536000);
		rgb(117pt)=(1.000000,1.000000,0.532000);
		rgb(118pt)=(1.000000,1.000000,0.528000);
		rgb(119pt)=(1.000000,1.000000,0.524000);
		rgb(120pt)=(1.000000,1.000000,0.520000);
		rgb(121pt)=(1.000000,1.000000,0.516000);
		rgb(122pt)=(1.000000,1.000000,0.512000);
		rgb(123pt)=(1.000000,1.000000,0.508000);
		rgb(124pt)=(1.000000,1.000000,0.504000);
		rgb(125pt)=(1.000000,1.000000,0.500000);
		rgb(126pt)=(1.000000,1.000000,0.496000);
		rgb(127pt)=(1.000000,1.000000,0.492000);
		rgb(128pt)=(1.000000,1.000000,0.488000);
		rgb(129pt)=(1.000000,1.000000,0.484000);
		rgb(130pt)=(1.000000,1.000000,0.480000);
		rgb(131pt)=(1.000000,1.000000,0.476000);
		rgb(132pt)=(1.000000,1.000000,0.472000);
		rgb(133pt)=(1.000000,1.000000,0.468000);
		rgb(134pt)=(1.000000,1.000000,0.464000);
		rgb(135pt)=(1.000000,1.000000,0.460000);
		rgb(136pt)=(1.000000,1.000000,0.456000);
		rgb(137pt)=(1.000000,1.000000,0.452000);
		rgb(138pt)=(1.000000,1.000000,0.448000);
		rgb(139pt)=(1.000000,1.000000,0.444000);
		rgb(140pt)=(1.000000,1.000000,0.440000);
		rgb(141pt)=(1.000000,1.000000,0.436000);
		rgb(142pt)=(1.000000,1.000000,0.432000);
		rgb(143pt)=(1.000000,1.000000,0.428000);
		rgb(144pt)=(1.000000,1.000000,0.424000);
		rgb(145pt)=(1.000000,1.000000,0.420000);
		rgb(146pt)=(1.000000,1.000000,0.416000);
		rgb(147pt)=(1.000000,1.000000,0.412000);
		rgb(148pt)=(1.000000,1.000000,0.408000);
		rgb(149pt)=(1.000000,1.000000,0.404000);
		rgb(150pt)=(1.000000,1.000000,0.400000);
		rgb(151pt)=(1.000000,1.000000,0.396000);
		rgb(152pt)=(1.000000,1.000000,0.392000);
		rgb(153pt)=(1.000000,1.000000,0.388000);
		rgb(154pt)=(1.000000,1.000000,0.384000);
		rgb(155pt)=(1.000000,1.000000,0.380000);
		rgb(156pt)=(1.000000,1.000000,0.376000);
		rgb(157pt)=(1.000000,1.000000,0.372000);
		rgb(158pt)=(1.000000,1.000000,0.368000);
		rgb(159pt)=(1.000000,1.000000,0.364000);
		rgb(160pt)=(1.000000,1.000000,0.360000);
		rgb(161pt)=(1.000000,1.000000,0.356000);
		rgb(162pt)=(1.000000,1.000000,0.352000);
		rgb(163pt)=(1.000000,1.000000,0.348000);
		rgb(164pt)=(1.000000,1.000000,0.344000);
		rgb(165pt)=(1.000000,1.000000,0.340000);
		rgb(166pt)=(1.000000,1.000000,0.336000);
		rgb(167pt)=(1.000000,1.000000,0.332000);
		rgb(168pt)=(1.000000,1.000000,0.328000);
		rgb(169pt)=(1.000000,1.000000,0.324000);
		rgb(170pt)=(1.000000,1.000000,0.320000);
		rgb(171pt)=(1.000000,1.000000,0.316000);
		rgb(172pt)=(1.000000,1.000000,0.312000);
		rgb(173pt)=(1.000000,1.000000,0.308000);
		rgb(174pt)=(1.000000,1.000000,0.304000);
		rgb(175pt)=(1.000000,1.000000,0.300000);
		rgb(176pt)=(1.000000,1.000000,0.296000);
		rgb(177pt)=(1.000000,1.000000,0.292000);
		rgb(178pt)=(1.000000,1.000000,0.288000);
		rgb(179pt)=(1.000000,1.000000,0.284000);
		rgb(180pt)=(1.000000,1.000000,0.280000);
		rgb(181pt)=(1.000000,1.000000,0.276000);
		rgb(182pt)=(1.000000,1.000000,0.272000);
		rgb(183pt)=(1.000000,1.000000,0.268000);
		rgb(184pt)=(1.000000,1.000000,0.264000);
		rgb(185pt)=(1.000000,1.000000,0.260000);
		rgb(186pt)=(1.000000,1.000000,0.256000);
		rgb(187pt)=(1.000000,1.000000,0.252000);
		rgb(188pt)=(1.000000,1.000000,0.248000);
		rgb(189pt)=(1.000000,1.000000,0.244000);
		rgb(190pt)=(1.000000,1.000000,0.240000);
		rgb(191pt)=(1.000000,1.000000,0.236000);
		rgb(192pt)=(1.000000,1.000000,0.232000);
		rgb(193pt)=(1.000000,1.000000,0.228000);
		rgb(194pt)=(1.000000,1.000000,0.224000);
		rgb(195pt)=(1.000000,1.000000,0.220000);
		rgb(196pt)=(1.000000,1.000000,0.216000);
		rgb(197pt)=(1.000000,1.000000,0.212000);
		rgb(198pt)=(1.000000,1.000000,0.208000);
		rgb(199pt)=(1.000000,1.000000,0.204000);
		rgb(200pt)=(1.000000,1.000000,0.200000);
		rgb(201pt)=(1.000000,1.000000,0.196000);
		rgb(202pt)=(1.000000,1.000000,0.192000);
		rgb(203pt)=(1.000000,1.000000,0.188000);
		rgb(204pt)=(1.000000,1.000000,0.184000);
		rgb(205pt)=(1.000000,1.000000,0.180000);
		rgb(206pt)=(1.000000,1.000000,0.176000);
		rgb(207pt)=(1.000000,1.000000,0.172000);
		rgb(208pt)=(1.000000,1.000000,0.168000);
		rgb(209pt)=(1.000000,1.000000,0.164000);
		rgb(210pt)=(1.000000,1.000000,0.160000);
		rgb(211pt)=(1.000000,1.000000,0.156000);
		rgb(212pt)=(1.000000,1.000000,0.152000);
		rgb(213pt)=(1.000000,1.000000,0.148000);
		rgb(214pt)=(1.000000,1.000000,0.144000);
		rgb(215pt)=(1.000000,1.000000,0.140000);
		rgb(216pt)=(1.000000,1.000000,0.136000);
		rgb(217pt)=(1.000000,1.000000,0.132000);
		rgb(218pt)=(1.000000,1.000000,0.128000);
		rgb(219pt)=(1.000000,1.000000,0.124000);
		rgb(220pt)=(1.000000,1.000000,0.120000);
		rgb(221pt)=(1.000000,1.000000,0.116000);
		rgb(222pt)=(1.000000,1.000000,0.112000);
		rgb(223pt)=(1.000000,1.000000,0.108000);
		rgb(224pt)=(1.000000,1.000000,0.104000);
		rgb(225pt)=(1.000000,1.000000,0.100000);
		rgb(226pt)=(1.000000,1.000000,0.096000);
		rgb(227pt)=(1.000000,1.000000,0.092000);
		rgb(228pt)=(1.000000,1.000000,0.088000);
		rgb(229pt)=(1.000000,1.000000,0.084000);
		rgb(230pt)=(1.000000,1.000000,0.080000);
		rgb(231pt)=(1.000000,1.000000,0.076000);
		rgb(232pt)=(1.000000,1.000000,0.072000);
		rgb(233pt)=(1.000000,1.000000,0.068000);
		rgb(234pt)=(1.000000,1.000000,0.064000);
		rgb(235pt)=(1.000000,1.000000,0.060000);
		rgb(236pt)=(1.000000,1.000000,0.056000);
		rgb(237pt)=(1.000000,1.000000,0.052000);
		rgb(238pt)=(1.000000,1.000000,0.048000);
		rgb(239pt)=(1.000000,1.000000,0.044000);
		rgb(240pt)=(1.000000,1.000000,0.040000);
		rgb(241pt)=(1.000000,1.000000,0.036000);
		rgb(242pt)=(1.000000,1.000000,0.032000);
		rgb(243pt)=(1.000000,1.000000,0.028000);
		rgb(244pt)=(1.000000,1.000000,0.024000);
		rgb(245pt)=(1.000000,1.000000,0.020000);
		rgb(246pt)=(1.000000,1.000000,0.016000);
		rgb(247pt)=(1.000000,1.000000,0.012000);
		rgb(248pt)=(1.000000,1.000000,0.008000);
		rgb(249pt)=(1.000000,1.000000,0.004000);
		rgb(250pt)=(1.000000,1.000000,0.000000);
		rgb(251pt)=(1.000000,0.997333,0.000000);
		rgb(252pt)=(1.000000,0.994667,0.000000);
		rgb(253pt)=(1.000000,0.992000,0.000000);
		rgb(254pt)=(1.000000,0.989333,0.000000);
		rgb(255pt)=(1.000000,0.986667,0.000000);
		rgb(256pt)=(1.000000,0.984000,0.000000);
		rgb(257pt)=(1.000000,0.981333,0.000000);
		rgb(258pt)=(1.000000,0.978667,0.000000);
		rgb(259pt)=(1.000000,0.976000,0.000000);
		rgb(260pt)=(1.000000,0.973333,0.000000);
		rgb(261pt)=(1.000000,0.970667,0.000000);
		rgb(262pt)=(1.000000,0.968000,0.000000);
		rgb(263pt)=(1.000000,0.965333,0.000000);
		rgb(264pt)=(1.000000,0.962667,0.000000);
		rgb(265pt)=(1.000000,0.960000,0.000000);
		rgb(266pt)=(1.000000,0.957333,0.000000);
		rgb(267pt)=(1.000000,0.954667,0.000000);
		rgb(268pt)=(1.000000,0.952000,0.000000);
		rgb(269pt)=(1.000000,0.949333,0.000000);
		rgb(270pt)=(1.000000,0.946667,0.000000);
		rgb(271pt)=(1.000000,0.944000,0.000000);
		rgb(272pt)=(1.000000,0.941333,0.000000);
		rgb(273pt)=(1.000000,0.938667,0.000000);
		rgb(274pt)=(1.000000,0.936000,0.000000);
		rgb(275pt)=(1.000000,0.933333,0.000000);
		rgb(276pt)=(1.000000,0.930667,0.000000);
		rgb(277pt)=(1.000000,0.928000,0.000000);
		rgb(278pt)=(1.000000,0.925333,0.000000);
		rgb(279pt)=(1.000000,0.922667,0.000000);
		rgb(280pt)=(1.000000,0.920000,0.000000);
		rgb(281pt)=(1.000000,0.917333,0.000000);
		rgb(282pt)=(1.000000,0.914667,0.000000);
		rgb(283pt)=(1.000000,0.912000,0.000000);
		rgb(284pt)=(1.000000,0.909333,0.000000);
		rgb(285pt)=(1.000000,0.906667,0.000000);
		rgb(286pt)=(1.000000,0.904000,0.000000);
		rgb(287pt)=(1.000000,0.901333,0.000000);
		rgb(288pt)=(1.000000,0.898667,0.000000);
		rgb(289pt)=(1.000000,0.896000,0.000000);
		rgb(290pt)=(1.000000,0.893333,0.000000);
		rgb(291pt)=(1.000000,0.890667,0.000000);
		rgb(292pt)=(1.000000,0.888000,0.000000);
		rgb(293pt)=(1.000000,0.885333,0.000000);
		rgb(294pt)=(1.000000,0.882667,0.000000);
		rgb(295pt)=(1.000000,0.880000,0.000000);
		rgb(296pt)=(1.000000,0.877333,0.000000);
		rgb(297pt)=(1.000000,0.874667,0.000000);
		rgb(298pt)=(1.000000,0.872000,0.000000);
		rgb(299pt)=(1.000000,0.869333,0.000000);
		rgb(300pt)=(1.000000,0.866667,0.000000);
		rgb(301pt)=(1.000000,0.864000,0.000000);
		rgb(302pt)=(1.000000,0.861333,0.000000);
		rgb(303pt)=(1.000000,0.858667,0.000000);
		rgb(304pt)=(1.000000,0.856000,0.000000);
		rgb(305pt)=(1.000000,0.853333,0.000000);
		rgb(306pt)=(1.000000,0.850667,0.000000);
		rgb(307pt)=(1.000000,0.848000,0.000000);
		rgb(308pt)=(1.000000,0.845333,0.000000);
		rgb(309pt)=(1.000000,0.842667,0.000000);
		rgb(310pt)=(1.000000,0.840000,0.000000);
		rgb(311pt)=(1.000000,0.837333,0.000000);
		rgb(312pt)=(1.000000,0.834667,0.000000);
		rgb(313pt)=(1.000000,0.832000,0.000000);
		rgb(314pt)=(1.000000,0.829333,0.000000);
		rgb(315pt)=(1.000000,0.826667,0.000000);
		rgb(316pt)=(1.000000,0.824000,0.000000);
		rgb(317pt)=(1.000000,0.821333,0.000000);
		rgb(318pt)=(1.000000,0.818667,0.000000);
		rgb(319pt)=(1.000000,0.816000,0.000000);
		rgb(320pt)=(1.000000,0.813333,0.000000);
		rgb(321pt)=(1.000000,0.810667,0.000000);
		rgb(322pt)=(1.000000,0.808000,0.000000);
		rgb(323pt)=(1.000000,0.805333,0.000000);
		rgb(324pt)=(1.000000,0.802667,0.000000);
		rgb(325pt)=(1.000000,0.800000,0.000000);
		rgb(326pt)=(1.000000,0.797333,0.000000);
		rgb(327pt)=(1.000000,0.794667,0.000000);
		rgb(328pt)=(1.000000,0.792000,0.000000);
		rgb(329pt)=(1.000000,0.789333,0.000000);
		rgb(330pt)=(1.000000,0.786667,0.000000);
		rgb(331pt)=(1.000000,0.784000,0.000000);
		rgb(332pt)=(1.000000,0.781333,0.000000);
		rgb(333pt)=(1.000000,0.778667,0.000000);
		rgb(334pt)=(1.000000,0.776000,0.000000);
		rgb(335pt)=(1.000000,0.773333,0.000000);
		rgb(336pt)=(1.000000,0.770667,0.000000);
		rgb(337pt)=(1.000000,0.768000,0.000000);
		rgb(338pt)=(1.000000,0.765333,0.000000);
		rgb(339pt)=(1.000000,0.762667,0.000000);
		rgb(340pt)=(1.000000,0.760000,0.000000);
		rgb(341pt)=(1.000000,0.757333,0.000000);
		rgb(342pt)=(1.000000,0.754667,0.000000);
		rgb(343pt)=(1.000000,0.752000,0.000000);
		rgb(344pt)=(1.000000,0.749333,0.000000);
		rgb(345pt)=(1.000000,0.746667,0.000000);
		rgb(346pt)=(1.000000,0.744000,0.000000);
		rgb(347pt)=(1.000000,0.741333,0.000000);
		rgb(348pt)=(1.000000,0.738667,0.000000);
		rgb(349pt)=(1.000000,0.736000,0.000000);
		rgb(350pt)=(1.000000,0.733333,0.000000);
		rgb(351pt)=(1.000000,0.730667,0.000000);
		rgb(352pt)=(1.000000,0.728000,0.000000);
		rgb(353pt)=(1.000000,0.725333,0.000000);
		rgb(354pt)=(1.000000,0.722667,0.000000);
		rgb(355pt)=(1.000000,0.720000,0.000000);
		rgb(356pt)=(1.000000,0.717333,0.000000);
		rgb(357pt)=(1.000000,0.714667,0.000000);
		rgb(358pt)=(1.000000,0.712000,0.000000);
		rgb(359pt)=(1.000000,0.709333,0.000000);
		rgb(360pt)=(1.000000,0.706667,0.000000);
		rgb(361pt)=(1.000000,0.704000,0.000000);
		rgb(362pt)=(1.000000,0.701333,0.000000);
		rgb(363pt)=(1.000000,0.698667,0.000000);
		rgb(364pt)=(1.000000,0.696000,0.000000);
		rgb(365pt)=(1.000000,0.693333,0.000000);
		rgb(366pt)=(1.000000,0.690667,0.000000);
		rgb(367pt)=(1.000000,0.688000,0.000000);
		rgb(368pt)=(1.000000,0.685333,0.000000);
		rgb(369pt)=(1.000000,0.682667,0.000000);
		rgb(370pt)=(1.000000,0.680000,0.000000);
		rgb(371pt)=(1.000000,0.677333,0.000000);
		rgb(372pt)=(1.000000,0.674667,0.000000);
		rgb(373pt)=(1.000000,0.672000,0.000000);
		rgb(374pt)=(1.000000,0.669333,0.000000);
		rgb(375pt)=(1.000000,0.666667,0.000000);
		rgb(376pt)=(1.000000,0.664000,0.000000);
		rgb(377pt)=(1.000000,0.661333,0.000000);
		rgb(378pt)=(1.000000,0.658667,0.000000);
		rgb(379pt)=(1.000000,0.656000,0.000000);
		rgb(380pt)=(1.000000,0.653333,0.000000);
		rgb(381pt)=(1.000000,0.650667,0.000000);
		rgb(382pt)=(1.000000,0.648000,0.000000);
		rgb(383pt)=(1.000000,0.645333,0.000000);
		rgb(384pt)=(1.000000,0.642667,0.000000);
		rgb(385pt)=(1.000000,0.640000,0.000000);
		rgb(386pt)=(1.000000,0.637333,0.000000);
		rgb(387pt)=(1.000000,0.634667,0.000000);
		rgb(388pt)=(1.000000,0.632000,0.000000);
		rgb(389pt)=(1.000000,0.629333,0.000000);
		rgb(390pt)=(1.000000,0.626667,0.000000);
		rgb(391pt)=(1.000000,0.624000,0.000000);
		rgb(392pt)=(1.000000,0.621333,0.000000);
		rgb(393pt)=(1.000000,0.618667,0.000000);
		rgb(394pt)=(1.000000,0.616000,0.000000);
		rgb(395pt)=(1.000000,0.613333,0.000000);
		rgb(396pt)=(1.000000,0.610667,0.000000);
		rgb(397pt)=(1.000000,0.608000,0.000000);
		rgb(398pt)=(1.000000,0.605333,0.000000);
		rgb(399pt)=(1.000000,0.602667,0.000000);
		rgb(400pt)=(1.000000,0.600000,0.000000);
		rgb(401pt)=(1.000000,0.597333,0.000000);
		rgb(402pt)=(1.000000,0.594667,0.000000);
		rgb(403pt)=(1.000000,0.592000,0.000000);
		rgb(404pt)=(1.000000,0.589333,0.000000);
		rgb(405pt)=(1.000000,0.586667,0.000000);
		rgb(406pt)=(1.000000,0.584000,0.000000);
		rgb(407pt)=(1.000000,0.581333,0.000000);
		rgb(408pt)=(1.000000,0.578667,0.000000);
		rgb(409pt)=(1.000000,0.576000,0.000000);
		rgb(410pt)=(1.000000,0.573333,0.000000);
		rgb(411pt)=(1.000000,0.570667,0.000000);
		rgb(412pt)=(1.000000,0.568000,0.000000);
		rgb(413pt)=(1.000000,0.565333,0.000000);
		rgb(414pt)=(1.000000,0.562667,0.000000);
		rgb(415pt)=(1.000000,0.560000,0.000000);
		rgb(416pt)=(1.000000,0.557333,0.000000);
		rgb(417pt)=(1.000000,0.554667,0.000000);
		rgb(418pt)=(1.000000,0.552000,0.000000);
		rgb(419pt)=(1.000000,0.549333,0.000000);
		rgb(420pt)=(1.000000,0.546667,0.000000);
		rgb(421pt)=(1.000000,0.544000,0.000000);
		rgb(422pt)=(1.000000,0.541333,0.000000);
		rgb(423pt)=(1.000000,0.538667,0.000000);
		rgb(424pt)=(1.000000,0.536000,0.000000);
		rgb(425pt)=(1.000000,0.533333,0.000000);
		rgb(426pt)=(1.000000,0.530667,0.000000);
		rgb(427pt)=(1.000000,0.528000,0.000000);
		rgb(428pt)=(1.000000,0.525333,0.000000);
		rgb(429pt)=(1.000000,0.522667,0.000000);
		rgb(430pt)=(1.000000,0.520000,0.000000);
		rgb(431pt)=(1.000000,0.517333,0.000000);
		rgb(432pt)=(1.000000,0.514667,0.000000);
		rgb(433pt)=(1.000000,0.512000,0.000000);
		rgb(434pt)=(1.000000,0.509333,0.000000);
		rgb(435pt)=(1.000000,0.506667,0.000000);
		rgb(436pt)=(1.000000,0.504000,0.000000);
		rgb(437pt)=(1.000000,0.501333,0.000000);
		rgb(438pt)=(1.000000,0.498667,0.000000);
		rgb(439pt)=(1.000000,0.496000,0.000000);
		rgb(440pt)=(1.000000,0.493333,0.000000);
		rgb(441pt)=(1.000000,0.490667,0.000000);
		rgb(442pt)=(1.000000,0.488000,0.000000);
		rgb(443pt)=(1.000000,0.485333,0.000000);
		rgb(444pt)=(1.000000,0.482667,0.000000);
		rgb(445pt)=(1.000000,0.480000,0.000000);
		rgb(446pt)=(1.000000,0.477333,0.000000);
		rgb(447pt)=(1.000000,0.474667,0.000000);
		rgb(448pt)=(1.000000,0.472000,0.000000);
		rgb(449pt)=(1.000000,0.469333,0.000000);
		rgb(450pt)=(1.000000,0.466667,0.000000);
		rgb(451pt)=(1.000000,0.464000,0.000000);
		rgb(452pt)=(1.000000,0.461333,0.000000);
		rgb(453pt)=(1.000000,0.458667,0.000000);
		rgb(454pt)=(1.000000,0.456000,0.000000);
		rgb(455pt)=(1.000000,0.453333,0.000000);
		rgb(456pt)=(1.000000,0.450667,0.000000);
		rgb(457pt)=(1.000000,0.448000,0.000000);
		rgb(458pt)=(1.000000,0.445333,0.000000);
		rgb(459pt)=(1.000000,0.442667,0.000000);
		rgb(460pt)=(1.000000,0.440000,0.000000);
		rgb(461pt)=(1.000000,0.437333,0.000000);
		rgb(462pt)=(1.000000,0.434667,0.000000);
		rgb(463pt)=(1.000000,0.432000,0.000000);
		rgb(464pt)=(1.000000,0.429333,0.000000);
		rgb(465pt)=(1.000000,0.426667,0.000000);
		rgb(466pt)=(1.000000,0.424000,0.000000);
		rgb(467pt)=(1.000000,0.421333,0.000000);
		rgb(468pt)=(1.000000,0.418667,0.000000);
		rgb(469pt)=(1.000000,0.416000,0.000000);
		rgb(470pt)=(1.000000,0.413333,0.000000);
		rgb(471pt)=(1.000000,0.410667,0.000000);
		rgb(472pt)=(1.000000,0.408000,0.000000);
		rgb(473pt)=(1.000000,0.405333,0.000000);
		rgb(474pt)=(1.000000,0.402667,0.000000);
		rgb(475pt)=(1.000000,0.400000,0.000000);
		rgb(476pt)=(1.000000,0.397333,0.000000);
		rgb(477pt)=(1.000000,0.394667,0.000000);
		rgb(478pt)=(1.000000,0.392000,0.000000);
		rgb(479pt)=(1.000000,0.389333,0.000000);
		rgb(480pt)=(1.000000,0.386667,0.000000);
		rgb(481pt)=(1.000000,0.384000,0.000000);
		rgb(482pt)=(1.000000,0.381333,0.000000);
		rgb(483pt)=(1.000000,0.378667,0.000000);
		rgb(484pt)=(1.000000,0.376000,0.000000);
		rgb(485pt)=(1.000000,0.373333,0.000000);
		rgb(486pt)=(1.000000,0.370667,0.000000);
		rgb(487pt)=(1.000000,0.368000,0.000000);
		rgb(488pt)=(1.000000,0.365333,0.000000);
		rgb(489pt)=(1.000000,0.362667,0.000000);
		rgb(490pt)=(1.000000,0.360000,0.000000);
		rgb(491pt)=(1.000000,0.357333,0.000000);
		rgb(492pt)=(1.000000,0.354667,0.000000);
		rgb(493pt)=(1.000000,0.352000,0.000000);
		rgb(494pt)=(1.000000,0.349333,0.000000);
		rgb(495pt)=(1.000000,0.346667,0.000000);
		rgb(496pt)=(1.000000,0.344000,0.000000);
		rgb(497pt)=(1.000000,0.341333,0.000000);
		rgb(498pt)=(1.000000,0.338667,0.000000);
		rgb(499pt)=(1.000000,0.336000,0.000000);
		rgb(500pt)=(1.000000,0.333333,0.000000);
		rgb(501pt)=(1.000000,0.330667,0.000000);
		rgb(502pt)=(1.000000,0.328000,0.000000);
		rgb(503pt)=(1.000000,0.325333,0.000000);
		rgb(504pt)=(1.000000,0.322667,0.000000);
		rgb(505pt)=(1.000000,0.320000,0.000000);
		rgb(506pt)=(1.000000,0.317333,0.000000);
		rgb(507pt)=(1.000000,0.314667,0.000000);
		rgb(508pt)=(1.000000,0.312000,0.000000);
		rgb(509pt)=(1.000000,0.309333,0.000000);
		rgb(510pt)=(1.000000,0.306667,0.000000);
		rgb(511pt)=(1.000000,0.304000,0.000000);
		rgb(512pt)=(1.000000,0.301333,0.000000);
		rgb(513pt)=(1.000000,0.298667,0.000000);
		rgb(514pt)=(1.000000,0.296000,0.000000);
		rgb(515pt)=(1.000000,0.293333,0.000000);
		rgb(516pt)=(1.000000,0.290667,0.000000);
		rgb(517pt)=(1.000000,0.288000,0.000000);
		rgb(518pt)=(1.000000,0.285333,0.000000);
		rgb(519pt)=(1.000000,0.282667,0.000000);
		rgb(520pt)=(1.000000,0.280000,0.000000);
		rgb(521pt)=(1.000000,0.277333,0.000000);
		rgb(522pt)=(1.000000,0.274667,0.000000);
		rgb(523pt)=(1.000000,0.272000,0.000000);
		rgb(524pt)=(1.000000,0.269333,0.000000);
		rgb(525pt)=(1.000000,0.266667,0.000000);
		rgb(526pt)=(1.000000,0.264000,0.000000);
		rgb(527pt)=(1.000000,0.261333,0.000000);
		rgb(528pt)=(1.000000,0.258667,0.000000);
		rgb(529pt)=(1.000000,0.256000,0.000000);
		rgb(530pt)=(1.000000,0.253333,0.000000);
		rgb(531pt)=(1.000000,0.250667,0.000000);
		rgb(532pt)=(1.000000,0.248000,0.000000);
		rgb(533pt)=(1.000000,0.245333,0.000000);
		rgb(534pt)=(1.000000,0.242667,0.000000);
		rgb(535pt)=(1.000000,0.240000,0.000000);
		rgb(536pt)=(1.000000,0.237333,0.000000);
		rgb(537pt)=(1.000000,0.234667,0.000000);
		rgb(538pt)=(1.000000,0.232000,0.000000);
		rgb(539pt)=(1.000000,0.229333,0.000000);
		rgb(540pt)=(1.000000,0.226667,0.000000);
		rgb(541pt)=(1.000000,0.224000,0.000000);
		rgb(542pt)=(1.000000,0.221333,0.000000);
		rgb(543pt)=(1.000000,0.218667,0.000000);
		rgb(544pt)=(1.000000,0.216000,0.000000);
		rgb(545pt)=(1.000000,0.213333,0.000000);
		rgb(546pt)=(1.000000,0.210667,0.000000);
		rgb(547pt)=(1.000000,0.208000,0.000000);
		rgb(548pt)=(1.000000,0.205333,0.000000);
		rgb(549pt)=(1.000000,0.202667,0.000000);
		rgb(550pt)=(1.000000,0.200000,0.000000);
		rgb(551pt)=(1.000000,0.197333,0.000000);
		rgb(552pt)=(1.000000,0.194667,0.000000);
		rgb(553pt)=(1.000000,0.192000,0.000000);
		rgb(554pt)=(1.000000,0.189333,0.000000);
		rgb(555pt)=(1.000000,0.186667,0.000000);
		rgb(556pt)=(1.000000,0.184000,0.000000);
		rgb(557pt)=(1.000000,0.181333,0.000000);
		rgb(558pt)=(1.000000,0.178667,0.000000);
		rgb(559pt)=(1.000000,0.176000,0.000000);
		rgb(560pt)=(1.000000,0.173333,0.000000);
		rgb(561pt)=(1.000000,0.170667,0.000000);
		rgb(562pt)=(1.000000,0.168000,0.000000);
		rgb(563pt)=(1.000000,0.165333,0.000000);
		rgb(564pt)=(1.000000,0.162667,0.000000);
		rgb(565pt)=(1.000000,0.160000,0.000000);
		rgb(566pt)=(1.000000,0.157333,0.000000);
		rgb(567pt)=(1.000000,0.154667,0.000000);
		rgb(568pt)=(1.000000,0.152000,0.000000);
		rgb(569pt)=(1.000000,0.149333,0.000000);
		rgb(570pt)=(1.000000,0.146667,0.000000);
		rgb(571pt)=(1.000000,0.144000,0.000000);
		rgb(572pt)=(1.000000,0.141333,0.000000);
		rgb(573pt)=(1.000000,0.138667,0.000000);
		rgb(574pt)=(1.000000,0.136000,0.000000);
		rgb(575pt)=(1.000000,0.133333,0.000000);
		rgb(576pt)=(1.000000,0.130667,0.000000);
		rgb(577pt)=(1.000000,0.128000,0.000000);
		rgb(578pt)=(1.000000,0.125333,0.000000);
		rgb(579pt)=(1.000000,0.122667,0.000000);
		rgb(580pt)=(1.000000,0.120000,0.000000);
		rgb(581pt)=(1.000000,0.117333,0.000000);
		rgb(582pt)=(1.000000,0.114667,0.000000);
		rgb(583pt)=(1.000000,0.112000,0.000000);
		rgb(584pt)=(1.000000,0.109333,0.000000);
		rgb(585pt)=(1.000000,0.106667,0.000000);
		rgb(586pt)=(1.000000,0.104000,0.000000);
		rgb(587pt)=(1.000000,0.101333,0.000000);
		rgb(588pt)=(1.000000,0.098667,0.000000);
		rgb(589pt)=(1.000000,0.096000,0.000000);
		rgb(590pt)=(1.000000,0.093333,0.000000);
		rgb(591pt)=(1.000000,0.090667,0.000000);
		rgb(592pt)=(1.000000,0.088000,0.000000);
		rgb(593pt)=(1.000000,0.085333,0.000000);
		rgb(594pt)=(1.000000,0.082667,0.000000);
		rgb(595pt)=(1.000000,0.080000,0.000000);
		rgb(596pt)=(1.000000,0.077333,0.000000);
		rgb(597pt)=(1.000000,0.074667,0.000000);
		rgb(598pt)=(1.000000,0.072000,0.000000);
		rgb(599pt)=(1.000000,0.069333,0.000000);
		rgb(600pt)=(1.000000,0.066667,0.000000);
		rgb(601pt)=(1.000000,0.064000,0.000000);
		rgb(602pt)=(1.000000,0.061333,0.000000);
		rgb(603pt)=(1.000000,0.058667,0.000000);
		rgb(604pt)=(1.000000,0.056000,0.000000);
		rgb(605pt)=(1.000000,0.053333,0.000000);
		rgb(606pt)=(1.000000,0.050667,0.000000);
		rgb(607pt)=(1.000000,0.048000,0.000000);
		rgb(608pt)=(1.000000,0.045333,0.000000);
		rgb(609pt)=(1.000000,0.042667,0.000000);
		rgb(610pt)=(1.000000,0.040000,0.000000);
		rgb(611pt)=(1.000000,0.037333,0.000000);
		rgb(612pt)=(1.000000,0.034667,0.000000);
		rgb(613pt)=(1.000000,0.032000,0.000000);
		rgb(614pt)=(1.000000,0.029333,0.000000);
		rgb(615pt)=(1.000000,0.026667,0.000000);
		rgb(616pt)=(1.000000,0.024000,0.000000);
		rgb(617pt)=(1.000000,0.021333,0.000000);
		rgb(618pt)=(1.000000,0.018667,0.000000);
		rgb(619pt)=(1.000000,0.016000,0.000000);
		rgb(620pt)=(1.000000,0.013333,0.000000);
		rgb(621pt)=(1.000000,0.010667,0.000000);
		rgb(622pt)=(1.000000,0.008000,0.000000);
		rgb(623pt)=(1.000000,0.005333,0.000000);
		rgb(624pt)=(1.000000,0.002667,0.000000);
		rgb(625pt)=(1.000000,0.000000,0.000000);
		rgb(626pt)=(0.997333,0.000000,0.000000);
		rgb(627pt)=(0.994667,0.000000,0.000000);
		rgb(628pt)=(0.992000,0.000000,0.000000);
		rgb(629pt)=(0.989333,0.000000,0.000000);
		rgb(630pt)=(0.986667,0.000000,0.000000);
		rgb(631pt)=(0.984000,0.000000,0.000000);
		rgb(632pt)=(0.981333,0.000000,0.000000);
		rgb(633pt)=(0.978667,0.000000,0.000000);
		rgb(634pt)=(0.976000,0.000000,0.000000);
		rgb(635pt)=(0.973333,0.000000,0.000000);
		rgb(636pt)=(0.970667,0.000000,0.000000);
		rgb(637pt)=(0.968000,0.000000,0.000000);
		rgb(638pt)=(0.965333,0.000000,0.000000);
		rgb(639pt)=(0.962667,0.000000,0.000000);
		rgb(640pt)=(0.960000,0.000000,0.000000);
		rgb(641pt)=(0.957333,0.000000,0.000000);
		rgb(642pt)=(0.954667,0.000000,0.000000);
		rgb(643pt)=(0.952000,0.000000,0.000000);
		rgb(644pt)=(0.949333,0.000000,0.000000);
		rgb(645pt)=(0.946667,0.000000,0.000000);
		rgb(646pt)=(0.944000,0.000000,0.000000);
		rgb(647pt)=(0.941333,0.000000,0.000000);
		rgb(648pt)=(0.938667,0.000000,0.000000);
		rgb(649pt)=(0.936000,0.000000,0.000000);
		rgb(650pt)=(0.933333,0.000000,0.000000);
		rgb(651pt)=(0.930667,0.000000,0.000000);
		rgb(652pt)=(0.928000,0.000000,0.000000);
		rgb(653pt)=(0.925333,0.000000,0.000000);
		rgb(654pt)=(0.922667,0.000000,0.000000);
		rgb(655pt)=(0.920000,0.000000,0.000000);
		rgb(656pt)=(0.917333,0.000000,0.000000);
		rgb(657pt)=(0.914667,0.000000,0.000000);
		rgb(658pt)=(0.912000,0.000000,0.000000);
		rgb(659pt)=(0.909333,0.000000,0.000000);
		rgb(660pt)=(0.906667,0.000000,0.000000);
		rgb(661pt)=(0.904000,0.000000,0.000000);
		rgb(662pt)=(0.901333,0.000000,0.000000);
		rgb(663pt)=(0.898667,0.000000,0.000000);
		rgb(664pt)=(0.896000,0.000000,0.000000);
		rgb(665pt)=(0.893333,0.000000,0.000000);
		rgb(666pt)=(0.890667,0.000000,0.000000);
		rgb(667pt)=(0.888000,0.000000,0.000000);
		rgb(668pt)=(0.885333,0.000000,0.000000);
		rgb(669pt)=(0.882667,0.000000,0.000000);
		rgb(670pt)=(0.880000,0.000000,0.000000);
		rgb(671pt)=(0.877333,0.000000,0.000000);
		rgb(672pt)=(0.874667,0.000000,0.000000);
		rgb(673pt)=(0.872000,0.000000,0.000000);
		rgb(674pt)=(0.869333,0.000000,0.000000);
		rgb(675pt)=(0.866667,0.000000,0.000000);
		rgb(676pt)=(0.864000,0.000000,0.000000);
		rgb(677pt)=(0.861333,0.000000,0.000000);
		rgb(678pt)=(0.858667,0.000000,0.000000);
		rgb(679pt)=(0.856000,0.000000,0.000000);
		rgb(680pt)=(0.853333,0.000000,0.000000);
		rgb(681pt)=(0.850667,0.000000,0.000000);
		rgb(682pt)=(0.848000,0.000000,0.000000);
		rgb(683pt)=(0.845333,0.000000,0.000000);
		rgb(684pt)=(0.842667,0.000000,0.000000);
		rgb(685pt)=(0.840000,0.000000,0.000000);
		rgb(686pt)=(0.837333,0.000000,0.000000);
		rgb(687pt)=(0.834667,0.000000,0.000000);
		rgb(688pt)=(0.832000,0.000000,0.000000);
		rgb(689pt)=(0.829333,0.000000,0.000000);
		rgb(690pt)=(0.826667,0.000000,0.000000);
		rgb(691pt)=(0.824000,0.000000,0.000000);
		rgb(692pt)=(0.821333,0.000000,0.000000);
		rgb(693pt)=(0.818667,0.000000,0.000000);
		rgb(694pt)=(0.816000,0.000000,0.000000);
		rgb(695pt)=(0.813333,0.000000,0.000000);
		rgb(696pt)=(0.810667,0.000000,0.000000);
		rgb(697pt)=(0.808000,0.000000,0.000000);
		rgb(698pt)=(0.805333,0.000000,0.000000);
		rgb(699pt)=(0.802667,0.000000,0.000000);
		rgb(700pt)=(0.800000,0.000000,0.000000);
		rgb(701pt)=(0.797333,0.000000,0.000000);
		rgb(702pt)=(0.794667,0.000000,0.000000);
		rgb(703pt)=(0.792000,0.000000,0.000000);
		rgb(704pt)=(0.789333,0.000000,0.000000);
		rgb(705pt)=(0.786667,0.000000,0.000000);
		rgb(706pt)=(0.784000,0.000000,0.000000);
		rgb(707pt)=(0.781333,0.000000,0.000000);
		rgb(708pt)=(0.778667,0.000000,0.000000);
		rgb(709pt)=(0.776000,0.000000,0.000000);
		rgb(710pt)=(0.773333,0.000000,0.000000);
		rgb(711pt)=(0.770667,0.000000,0.000000);
		rgb(712pt)=(0.768000,0.000000,0.000000);
		rgb(713pt)=(0.765333,0.000000,0.000000);
		rgb(714pt)=(0.762667,0.000000,0.000000);
		rgb(715pt)=(0.760000,0.000000,0.000000);
		rgb(716pt)=(0.757333,0.000000,0.000000);
		rgb(717pt)=(0.754667,0.000000,0.000000);
		rgb(718pt)=(0.752000,0.000000,0.000000);
		rgb(719pt)=(0.749333,0.000000,0.000000);
		rgb(720pt)=(0.746667,0.000000,0.000000);
		rgb(721pt)=(0.744000,0.000000,0.000000);
		rgb(722pt)=(0.741333,0.000000,0.000000);
		rgb(723pt)=(0.738667,0.000000,0.000000);
		rgb(724pt)=(0.736000,0.000000,0.000000);
		rgb(725pt)=(0.733333,0.000000,0.000000);
		rgb(726pt)=(0.730667,0.000000,0.000000);
		rgb(727pt)=(0.728000,0.000000,0.000000);
		rgb(728pt)=(0.725333,0.000000,0.000000);
		rgb(729pt)=(0.722667,0.000000,0.000000);
		rgb(730pt)=(0.720000,0.000000,0.000000);
		rgb(731pt)=(0.717333,0.000000,0.000000);
		rgb(732pt)=(0.714667,0.000000,0.000000);
		rgb(733pt)=(0.712000,0.000000,0.000000);
		rgb(734pt)=(0.709333,0.000000,0.000000);
		rgb(735pt)=(0.706667,0.000000,0.000000);
		rgb(736pt)=(0.704000,0.000000,0.000000);
		rgb(737pt)=(0.701333,0.000000,0.000000);
		rgb(738pt)=(0.698667,0.000000,0.000000);
		rgb(739pt)=(0.696000,0.000000,0.000000);
		rgb(740pt)=(0.693333,0.000000,0.000000);
		rgb(741pt)=(0.690667,0.000000,0.000000);
		rgb(742pt)=(0.688000,0.000000,0.000000);
		rgb(743pt)=(0.685333,0.000000,0.000000);
		rgb(744pt)=(0.682667,0.000000,0.000000);
		rgb(745pt)=(0.680000,0.000000,0.000000);
		rgb(746pt)=(0.677333,0.000000,0.000000);
		rgb(747pt)=(0.674667,0.000000,0.000000);
		rgb(748pt)=(0.672000,0.000000,0.000000);
		rgb(749pt)=(0.669333,0.000000,0.000000);
		rgb(750pt)=(0.666667,0.000000,0.000000);
		rgb(751pt)=(0.664000,0.000000,0.000000);
		rgb(752pt)=(0.661333,0.000000,0.000000);
		rgb(753pt)=(0.658667,0.000000,0.000000);
		rgb(754pt)=(0.656000,0.000000,0.000000);
		rgb(755pt)=(0.653333,0.000000,0.000000);
		rgb(756pt)=(0.650667,0.000000,0.000000);
		rgb(757pt)=(0.648000,0.000000,0.000000);
		rgb(758pt)=(0.645333,0.000000,0.000000);
		rgb(759pt)=(0.642667,0.000000,0.000000);
		rgb(760pt)=(0.640000,0.000000,0.000000);
		rgb(761pt)=(0.637333,0.000000,0.000000);
		rgb(762pt)=(0.634667,0.000000,0.000000);
		rgb(763pt)=(0.632000,0.000000,0.000000);
		rgb(764pt)=(0.629333,0.000000,0.000000);
		rgb(765pt)=(0.626667,0.000000,0.000000);
		rgb(766pt)=(0.624000,0.000000,0.000000);
		rgb(767pt)=(0.621333,0.000000,0.000000);
		rgb(768pt)=(0.618667,0.000000,0.000000);
		rgb(769pt)=(0.616000,0.000000,0.000000);
		rgb(770pt)=(0.613333,0.000000,0.000000);
		rgb(771pt)=(0.610667,0.000000,0.000000);
		rgb(772pt)=(0.608000,0.000000,0.000000);
		rgb(773pt)=(0.605333,0.000000,0.000000);
		rgb(774pt)=(0.602667,0.000000,0.000000);
		rgb(775pt)=(0.600000,0.000000,0.000000);
		rgb(776pt)=(0.597333,0.000000,0.000000);
		rgb(777pt)=(0.594667,0.000000,0.000000);
		rgb(778pt)=(0.592000,0.000000,0.000000);
		rgb(779pt)=(0.589333,0.000000,0.000000);
		rgb(780pt)=(0.586667,0.000000,0.000000);
		rgb(781pt)=(0.584000,0.000000,0.000000);
		rgb(782pt)=(0.581333,0.000000,0.000000);
		rgb(783pt)=(0.578667,0.000000,0.000000);
		rgb(784pt)=(0.576000,0.000000,0.000000);
		rgb(785pt)=(0.573333,0.000000,0.000000);
		rgb(786pt)=(0.570667,0.000000,0.000000);
		rgb(787pt)=(0.568000,0.000000,0.000000);
		rgb(788pt)=(0.565333,0.000000,0.000000);
		rgb(789pt)=(0.562667,0.000000,0.000000);
		rgb(790pt)=(0.560000,0.000000,0.000000);
		rgb(791pt)=(0.557333,0.000000,0.000000);
		rgb(792pt)=(0.554667,0.000000,0.000000);
		rgb(793pt)=(0.552000,0.000000,0.000000);
		rgb(794pt)=(0.549333,0.000000,0.000000);
		rgb(795pt)=(0.546667,0.000000,0.000000);
		rgb(796pt)=(0.544000,0.000000,0.000000);
		rgb(797pt)=(0.541333,0.000000,0.000000);
		rgb(798pt)=(0.538667,0.000000,0.000000);
		rgb(799pt)=(0.536000,0.000000,0.000000);
		rgb(800pt)=(0.533333,0.000000,0.000000);
		rgb(801pt)=(0.530667,0.000000,0.000000);
		rgb(802pt)=(0.528000,0.000000,0.000000);
		rgb(803pt)=(0.525333,0.000000,0.000000);
		rgb(804pt)=(0.522667,0.000000,0.000000);
		rgb(805pt)=(0.520000,0.000000,0.000000);
		rgb(806pt)=(0.517333,0.000000,0.000000);
		rgb(807pt)=(0.514667,0.000000,0.000000);
		rgb(808pt)=(0.512000,0.000000,0.000000);
		rgb(809pt)=(0.509333,0.000000,0.000000);
		rgb(810pt)=(0.506667,0.000000,0.000000);
		rgb(811pt)=(0.504000,0.000000,0.000000);
		rgb(812pt)=(0.501333,0.000000,0.000000);
		rgb(813pt)=(0.498667,0.000000,0.000000);
		rgb(814pt)=(0.496000,0.000000,0.000000);
		rgb(815pt)=(0.493333,0.000000,0.000000);
		rgb(816pt)=(0.490667,0.000000,0.000000);
		rgb(817pt)=(0.488000,0.000000,0.000000);
		rgb(818pt)=(0.485333,0.000000,0.000000);
		rgb(819pt)=(0.482667,0.000000,0.000000);
		rgb(820pt)=(0.480000,0.000000,0.000000);
		rgb(821pt)=(0.477333,0.000000,0.000000);
		rgb(822pt)=(0.474667,0.000000,0.000000);
		rgb(823pt)=(0.472000,0.000000,0.000000);
		rgb(824pt)=(0.469333,0.000000,0.000000);
		rgb(825pt)=(0.466667,0.000000,0.000000);
		rgb(826pt)=(0.464000,0.000000,0.000000);
		rgb(827pt)=(0.461333,0.000000,0.000000);
		rgb(828pt)=(0.458667,0.000000,0.000000);
		rgb(829pt)=(0.456000,0.000000,0.000000);
		rgb(830pt)=(0.453333,0.000000,0.000000);
		rgb(831pt)=(0.450667,0.000000,0.000000);
		rgb(832pt)=(0.448000,0.000000,0.000000);
		rgb(833pt)=(0.445333,0.000000,0.000000);
		rgb(834pt)=(0.442667,0.000000,0.000000);
		rgb(835pt)=(0.440000,0.000000,0.000000);
		rgb(836pt)=(0.437333,0.000000,0.000000);
		rgb(837pt)=(0.434667,0.000000,0.000000);
		rgb(838pt)=(0.432000,0.000000,0.000000);
		rgb(839pt)=(0.429333,0.000000,0.000000);
		rgb(840pt)=(0.426667,0.000000,0.000000);
		rgb(841pt)=(0.424000,0.000000,0.000000);
		rgb(842pt)=(0.421333,0.000000,0.000000);
		rgb(843pt)=(0.418667,0.000000,0.000000);
		rgb(844pt)=(0.416000,0.000000,0.000000);
		rgb(845pt)=(0.413333,0.000000,0.000000);
		rgb(846pt)=(0.410667,0.000000,0.000000);
		rgb(847pt)=(0.408000,0.000000,0.000000);
		rgb(848pt)=(0.405333,0.000000,0.000000);
		rgb(849pt)=(0.402667,0.000000,0.000000);
		rgb(850pt)=(0.400000,0.000000,0.000000);
		rgb(851pt)=(0.397333,0.000000,0.000000);
		rgb(852pt)=(0.394667,0.000000,0.000000);
		rgb(853pt)=(0.392000,0.000000,0.000000);
		rgb(854pt)=(0.389333,0.000000,0.000000);
		rgb(855pt)=(0.386667,0.000000,0.000000);
		rgb(856pt)=(0.384000,0.000000,0.000000);
		rgb(857pt)=(0.381333,0.000000,0.000000);
		rgb(858pt)=(0.378667,0.000000,0.000000);
		rgb(859pt)=(0.376000,0.000000,0.000000);
		rgb(860pt)=(0.373333,0.000000,0.000000);
		rgb(861pt)=(0.370667,0.000000,0.000000);
		rgb(862pt)=(0.368000,0.000000,0.000000);
		rgb(863pt)=(0.365333,0.000000,0.000000);
		rgb(864pt)=(0.362667,0.000000,0.000000);
		rgb(865pt)=(0.360000,0.000000,0.000000);
		rgb(866pt)=(0.357333,0.000000,0.000000);
		rgb(867pt)=(0.354667,0.000000,0.000000);
		rgb(868pt)=(0.352000,0.000000,0.000000);
		rgb(869pt)=(0.349333,0.000000,0.000000);
		rgb(870pt)=(0.346667,0.000000,0.000000);
		rgb(871pt)=(0.344000,0.000000,0.000000);
		rgb(872pt)=(0.341333,0.000000,0.000000);
		rgb(873pt)=(0.338667,0.000000,0.000000);
		rgb(874pt)=(0.336000,0.000000,0.000000);
		rgb(875pt)=(0.333333,0.000000,0.000000);
		rgb(876pt)=(0.330667,0.000000,0.000000);
		rgb(877pt)=(0.328000,0.000000,0.000000);
		rgb(878pt)=(0.325333,0.000000,0.000000);
		rgb(879pt)=(0.322667,0.000000,0.000000);
		rgb(880pt)=(0.320000,0.000000,0.000000);
		rgb(881pt)=(0.317333,0.000000,0.000000);
		rgb(882pt)=(0.314667,0.000000,0.000000);
		rgb(883pt)=(0.312000,0.000000,0.000000);
		rgb(884pt)=(0.309333,0.000000,0.000000);
		rgb(885pt)=(0.306667,0.000000,0.000000);
		rgb(886pt)=(0.304000,0.000000,0.000000);
		rgb(887pt)=(0.301333,0.000000,0.000000);
		rgb(888pt)=(0.298667,0.000000,0.000000);
		rgb(889pt)=(0.296000,0.000000,0.000000);
		rgb(890pt)=(0.293333,0.000000,0.000000);
		rgb(891pt)=(0.290667,0.000000,0.000000);
		rgb(892pt)=(0.288000,0.000000,0.000000);
		rgb(893pt)=(0.285333,0.000000,0.000000);
		rgb(894pt)=(0.282667,0.000000,0.000000);
		rgb(895pt)=(0.280000,0.000000,0.000000);
		rgb(896pt)=(0.277333,0.000000,0.000000);
		rgb(897pt)=(0.274667,0.000000,0.000000);
		rgb(898pt)=(0.272000,0.000000,0.000000);
		rgb(899pt)=(0.269333,0.000000,0.000000);
		rgb(900pt)=(0.266667,0.000000,0.000000);
		rgb(901pt)=(0.264000,0.000000,0.000000);
		rgb(902pt)=(0.261333,0.000000,0.000000);
		rgb(903pt)=(0.258667,0.000000,0.000000);
		rgb(904pt)=(0.256000,0.000000,0.000000);
		rgb(905pt)=(0.253333,0.000000,0.000000);
		rgb(906pt)=(0.250667,0.000000,0.000000);
		rgb(907pt)=(0.248000,0.000000,0.000000);
		rgb(908pt)=(0.245333,0.000000,0.000000);
		rgb(909pt)=(0.242667,0.000000,0.000000);
		rgb(910pt)=(0.240000,0.000000,0.000000);
		rgb(911pt)=(0.237333,0.000000,0.000000);
		rgb(912pt)=(0.234667,0.000000,0.000000);
		rgb(913pt)=(0.232000,0.000000,0.000000);
		rgb(914pt)=(0.229333,0.000000,0.000000);
		rgb(915pt)=(0.226667,0.000000,0.000000);
		rgb(916pt)=(0.224000,0.000000,0.000000);
		rgb(917pt)=(0.221333,0.000000,0.000000);
		rgb(918pt)=(0.218667,0.000000,0.000000);
		rgb(919pt)=(0.216000,0.000000,0.000000);
		rgb(920pt)=(0.213333,0.000000,0.000000);
		rgb(921pt)=(0.210667,0.000000,0.000000);
		rgb(922pt)=(0.208000,0.000000,0.000000);
		rgb(923pt)=(0.205333,0.000000,0.000000);
		rgb(924pt)=(0.202667,0.000000,0.000000);
		rgb(925pt)=(0.200000,0.000000,0.000000);
		rgb(926pt)=(0.197333,0.000000,0.000000);
		rgb(927pt)=(0.194667,0.000000,0.000000);
		rgb(928pt)=(0.192000,0.000000,0.000000);
		rgb(929pt)=(0.189333,0.000000,0.000000);
		rgb(930pt)=(0.186667,0.000000,0.000000);
		rgb(931pt)=(0.184000,0.000000,0.000000);
		rgb(932pt)=(0.181333,0.000000,0.000000);
		rgb(933pt)=(0.178667,0.000000,0.000000);
		rgb(934pt)=(0.176000,0.000000,0.000000);
		rgb(935pt)=(0.173333,0.000000,0.000000);
		rgb(936pt)=(0.170667,0.000000,0.000000);
		rgb(937pt)=(0.168000,0.000000,0.000000);
		rgb(938pt)=(0.165333,0.000000,0.000000);
		rgb(939pt)=(0.162667,0.000000,0.000000);
		rgb(940pt)=(0.160000,0.000000,0.000000);
		rgb(941pt)=(0.157333,0.000000,0.000000);
		rgb(942pt)=(0.154667,0.000000,0.000000);
		rgb(943pt)=(0.152000,0.000000,0.000000);
		rgb(944pt)=(0.149333,0.000000,0.000000);
		rgb(945pt)=(0.146667,0.000000,0.000000);
		rgb(946pt)=(0.144000,0.000000,0.000000);
		rgb(947pt)=(0.141333,0.000000,0.000000);
		rgb(948pt)=(0.138667,0.000000,0.000000);
		rgb(949pt)=(0.136000,0.000000,0.000000);
		rgb(950pt)=(0.133333,0.000000,0.000000);
		rgb(951pt)=(0.130667,0.000000,0.000000);
		rgb(952pt)=(0.128000,0.000000,0.000000);
		rgb(953pt)=(0.125333,0.000000,0.000000);
		rgb(954pt)=(0.122667,0.000000,0.000000);
		rgb(955pt)=(0.120000,0.000000,0.000000);
		rgb(956pt)=(0.117333,0.000000,0.000000);
		rgb(957pt)=(0.114667,0.000000,0.000000);
		rgb(958pt)=(0.112000,0.000000,0.000000);
		rgb(959pt)=(0.109333,0.000000,0.000000);
		rgb(960pt)=(0.106667,0.000000,0.000000);
		rgb(961pt)=(0.104000,0.000000,0.000000);
		rgb(962pt)=(0.101333,0.000000,0.000000);
		rgb(963pt)=(0.098667,0.000000,0.000000);
		rgb(964pt)=(0.096000,0.000000,0.000000);
		rgb(965pt)=(0.093333,0.000000,0.000000);
		rgb(966pt)=(0.090667,0.000000,0.000000);
		rgb(967pt)=(0.088000,0.000000,0.000000);
		rgb(968pt)=(0.085333,0.000000,0.000000);
		rgb(969pt)=(0.082667,0.000000,0.000000);
		rgb(970pt)=(0.080000,0.000000,0.000000);
		rgb(971pt)=(0.077333,0.000000,0.000000);
		rgb(972pt)=(0.074667,0.000000,0.000000);
		rgb(973pt)=(0.072000,0.000000,0.000000);
		rgb(974pt)=(0.069333,0.000000,0.000000);
		rgb(975pt)=(0.066667,0.000000,0.000000);
		rgb(976pt)=(0.064000,0.000000,0.000000);
		rgb(977pt)=(0.061333,0.000000,0.000000);
		rgb(978pt)=(0.058667,0.000000,0.000000);
		rgb(979pt)=(0.056000,0.000000,0.000000);
		rgb(980pt)=(0.053333,0.000000,0.000000);
		rgb(981pt)=(0.050667,0.000000,0.000000);
		rgb(982pt)=(0.048000,0.000000,0.000000);
		rgb(983pt)=(0.045333,0.000000,0.000000);
		rgb(984pt)=(0.042667,0.000000,0.000000);
		rgb(985pt)=(0.040000,0.000000,0.000000);
		rgb(986pt)=(0.037333,0.000000,0.000000);
		rgb(987pt)=(0.034667,0.000000,0.000000);
		rgb(988pt)=(0.032000,0.000000,0.000000);
		rgb(989pt)=(0.029333,0.000000,0.000000);
		rgb(990pt)=(0.026667,0.000000,0.000000);
		rgb(991pt)=(0.024000,0.000000,0.000000);
		rgb(992pt)=(0.021333,0.000000,0.000000);
		rgb(993pt)=(0.018667,0.000000,0.000000);
		rgb(994pt)=(0.016000,0.000000,0.000000);
		rgb(995pt)=(0.013333,0.000000,0.000000);
		rgb(996pt)=(0.010667,0.000000,0.000000);
		rgb(997pt)=(0.008000,0.000000,0.000000);
		rgb(998pt)=(0.005333,0.000000,0.000000);
		rgb(999pt)=(0.002667,0.000000,0.000000);
}}

\newlength{\figureheight}
\newlength{\figurewidth}
\pgfplotsset{
	colormap={parula}{
		rgb(0pt)=(0.2081,0.1663,0.5292);
		rgb(1pt)=(0.208355,0.16778,0.532238);
		rgb(2pt)=(0.208611,0.169261,0.535275);
		rgb(3pt)=(0.208866,0.170741,0.538313);
		rgb(4pt)=(0.209121,0.172222,0.54135);
		rgb(5pt)=(0.209376,0.173702,0.544388);
		rgb(6pt)=(0.209632,0.175183,0.547425);
		rgb(7pt)=(0.209887,0.176663,0.550463);
		rgb(8pt)=(0.210134,0.178144,0.553505);
		rgb(9pt)=(0.210338,0.179624,0.556568);
		rgb(10pt)=(0.210542,0.181105,0.559631);
		rgb(11pt)=(0.210746,0.182585,0.562694);
		rgb(12pt)=(0.210944,0.184066,0.565763);
		rgb(13pt)=(0.211123,0.185546,0.568852);
		rgb(14pt)=(0.211302,0.187027,0.57194);
		rgb(15pt)=(0.21148,0.188507,0.575029);
		rgb(16pt)=(0.211642,0.189996,0.578117);
		rgb(17pt)=(0.21177,0.191502,0.581206);
		rgb(18pt)=(0.211897,0.193008,0.584295);
		rgb(19pt)=(0.212025,0.194514,0.587383);
		rgb(20pt)=(0.212132,0.19602,0.590472);
		rgb(21pt)=(0.212208,0.197526,0.59356);
		rgb(22pt)=(0.212285,0.199032,0.596649);
		rgb(23pt)=(0.212361,0.200538,0.599738);
		rgb(24pt)=(0.212413,0.202044,0.602839);
		rgb(25pt)=(0.212438,0.20355,0.605953);
		rgb(26pt)=(0.212464,0.205056,0.609067);
		rgb(27pt)=(0.212489,0.206562,0.612181);
		rgb(28pt)=(0.212471,0.208083,0.61531);
		rgb(29pt)=(0.21242,0.209614,0.61845);
		rgb(30pt)=(0.212368,0.211146,0.621589);
		rgb(31pt)=(0.212317,0.212677,0.624729);
		rgb(32pt)=(0.212216,0.214209,0.627868);
		rgb(33pt)=(0.212088,0.215741,0.631008);
		rgb(34pt)=(0.211961,0.217272,0.634148);
		rgb(35pt)=(0.211833,0.218804,0.637287);
		rgb(36pt)=(0.211668,0.220354,0.640446);
		rgb(37pt)=(0.211489,0.221911,0.643611);
		rgb(38pt)=(0.21131,0.223468,0.646776);
		rgb(39pt)=(0.211132,0.225025,0.649941);
		rgb(40pt)=(0.210848,0.226603,0.653107);
		rgb(41pt)=(0.210541,0.228186,0.656272);
		rgb(42pt)=(0.210235,0.229768,0.659437);
		rgb(43pt)=(0.209929,0.231351,0.662602);
		rgb(44pt)=(0.209553,0.232934,0.665767);
		rgb(45pt)=(0.20917,0.234516,0.668932);
		rgb(46pt)=(0.208787,0.236099,0.672098);
		rgb(47pt)=(0.208405,0.237681,0.675263);
		rgb(48pt)=(0.20787,0.239289,0.678453);
		rgb(49pt)=(0.207334,0.240897,0.681644);
		rgb(50pt)=(0.206798,0.242505,0.684835);
		rgb(51pt)=(0.206255,0.244114,0.688025);
		rgb(52pt)=(0.205617,0.245722,0.691216);
		rgb(53pt)=(0.204979,0.24733,0.694407);
		rgb(54pt)=(0.204341,0.248938,0.697597);
		rgb(55pt)=(0.203675,0.250554,0.700792);
		rgb(56pt)=(0.202858,0.252213,0.704008);
		rgb(57pt)=(0.202041,0.253872,0.707224);
		rgb(58pt)=(0.201225,0.255531,0.710441);
		rgb(59pt)=(0.200372,0.257184,0.713657);
		rgb(60pt)=(0.199402,0.258818,0.716873);
		rgb(61pt)=(0.198432,0.260452,0.720089);
		rgb(62pt)=(0.197462,0.262085,0.723305);
		rgb(63pt)=(0.196419,0.263735,0.726522);
		rgb(64pt)=(0.195219,0.26542,0.729738);
		rgb(65pt)=(0.19402,0.267105,0.732954);
		rgb(66pt)=(0.19282,0.268789,0.73617);
		rgb(67pt)=(0.191549,0.270474,0.739386);
		rgb(68pt)=(0.19017,0.272159,0.742603);
		rgb(69pt)=(0.188792,0.273843,0.745819);
		rgb(70pt)=(0.187414,0.275528,0.749035);
		rgb(71pt)=(0.1859,0.277237,0.752264);
		rgb(72pt)=(0.184241,0.278973,0.755505);
		rgb(73pt)=(0.182581,0.280709,0.758747);
		rgb(74pt)=(0.180922,0.282444,0.761989);
		rgb(75pt)=(0.179133,0.284209,0.765245);
		rgb(76pt)=(0.177244,0.285996,0.768512);
		rgb(77pt)=(0.175356,0.287783,0.77178);
		rgb(78pt)=(0.173467,0.289569,0.775047);
		rgb(79pt)=(0.171363,0.291406,0.778314);
		rgb(80pt)=(0.169142,0.293269,0.781581);
		rgb(81pt)=(0.166922,0.295132,0.784849);
		rgb(82pt)=(0.164701,0.296996,0.788116);
		rgb(83pt)=(0.162238,0.298934,0.791365);
		rgb(84pt)=(0.159686,0.300899,0.794606);
		rgb(85pt)=(0.157133,0.302865,0.797848);
		rgb(86pt)=(0.15458,0.30483,0.80109);
		rgb(87pt)=(0.151738,0.306858,0.804352);
		rgb(88pt)=(0.148828,0.3089,0.80762);
		rgb(89pt)=(0.145918,0.310942,0.810887);
		rgb(90pt)=(0.143008,0.312984,0.814154);
		rgb(91pt)=(0.139687,0.31514,0.81733);
		rgb(92pt)=(0.136318,0.31731,0.820495);
		rgb(93pt)=(0.132949,0.319479,0.82366);
		rgb(94pt)=(0.129579,0.321649,0.826826);
		rgb(95pt)=(0.125811,0.323918,0.829841);
		rgb(96pt)=(0.122033,0.32619,0.832853);
		rgb(97pt)=(0.118256,0.328462,0.835865);
		rgb(98pt)=(0.114458,0.330737,0.838862);
		rgb(99pt)=(0.110349,0.333059,0.841619);
		rgb(100pt)=(0.106239,0.335382,0.844376);
		rgb(101pt)=(0.102129,0.337705,0.847132);
		rgb(102pt)=(0.0979874,0.340021,0.849835);
		rgb(103pt)=(0.093648,0.342292,0.852209);
		rgb(104pt)=(0.0893087,0.344564,0.854583);
		rgb(105pt)=(0.0849694,0.346836,0.856957);
		rgb(106pt)=(0.08063,0.349091,0.859234);
		rgb(107pt)=(0.0762907,0.351286,0.861174);
		rgb(108pt)=(0.0719514,0.353481,0.863114);
		rgb(109pt)=(0.067612,0.355676,0.865053);
		rgb(110pt)=(0.0633195,0.357817,0.866853);
		rgb(111pt)=(0.0591333,0.359833,0.868333);
		rgb(112pt)=(0.0549471,0.36185,0.869814);
		rgb(113pt)=(0.050761,0.363866,0.871294);
		rgb(114pt)=(0.0466838,0.365823,0.872626);
		rgb(115pt)=(0.0427784,0.367687,0.873724);
		rgb(116pt)=(0.038873,0.36955,0.874821);
		rgb(117pt)=(0.0349676,0.371414,0.875919);
		rgb(118pt)=(0.0315066,0.373217,0.876872);
		rgb(119pt)=(0.0285456,0.374953,0.877664);
		rgb(120pt)=(0.0255847,0.376688,0.878455);
		rgb(121pt)=(0.0226237,0.378424,0.879246);
		rgb(122pt)=(0.0202132,0.380061,0.879868);
		rgb(123pt)=(0.0182477,0.381618,0.880353);
		rgb(124pt)=(0.0162823,0.383175,0.880838);
		rgb(125pt)=(0.0143168,0.384732,0.881323);
		rgb(126pt)=(0.0127892,0.386241,0.881695);
		rgb(127pt)=(0.0115129,0.387721,0.882001);
		rgb(128pt)=(0.0102366,0.389202,0.882307);
		rgb(129pt)=(0.00896036,0.390682,0.882614);
		rgb(130pt)=(0.00812372,0.392089,0.88281);
		rgb(131pt)=(0.00746006,0.393468,0.882963);
		rgb(132pt)=(0.0067964,0.394846,0.883116);
		rgb(133pt)=(0.00613273,0.396224,0.883269);
		rgb(134pt)=(0.00581622,0.397562,0.88332);
		rgb(135pt)=(0.00558649,0.398889,0.883346);
		rgb(136pt)=(0.00535676,0.400217,0.883371);
		rgb(137pt)=(0.00512703,0.401544,0.883397);
		rgb(138pt)=(0.00516757,0.402804,0.883332);
		rgb(139pt)=(0.00524414,0.404054,0.883256);
		rgb(140pt)=(0.00532072,0.405305,0.883179);
		rgb(141pt)=(0.0053973,0.406556,0.883103);
		rgb(142pt)=(0.00572012,0.407757,0.882952);
		rgb(143pt)=(0.00605195,0.408957,0.882799);
		rgb(144pt)=(0.00638378,0.410157,0.882646);
		rgb(145pt)=(0.00672643,0.411355,0.882489);
		rgb(146pt)=(0.00728799,0.412529,0.882259);
		rgb(147pt)=(0.00784955,0.413704,0.88203);
		rgb(148pt)=(0.00841111,0.414878,0.8818);
		rgb(149pt)=(0.00898919,0.416045,0.881564);
		rgb(150pt)=(0.00967838,0.417168,0.881283);
		rgb(151pt)=(0.0103676,0.418292,0.881002);
		rgb(152pt)=(0.0110568,0.419415,0.880721);
		rgb(153pt)=(0.011773,0.420532,0.880435);
		rgb(154pt)=(0.0125898,0.42163,0.880129);
		rgb(155pt)=(0.0134066,0.422728,0.879823);
		rgb(156pt)=(0.0142234,0.423825,0.879516);
		rgb(157pt)=(0.0150703,0.424915,0.879195);
		rgb(158pt)=(0.0159892,0.425987,0.878838);
		rgb(159pt)=(0.0169081,0.427059,0.87848);
		rgb(160pt)=(0.017827,0.428132,0.878123);
		rgb(161pt)=(0.0187748,0.429194,0.877746);
		rgb(162pt)=(0.0197703,0.430241,0.877338);
		rgb(163pt)=(0.0207658,0.431287,0.876929);
		rgb(164pt)=(0.0217613,0.432334,0.876521);
		rgb(165pt)=(0.0227802,0.43338,0.876113);
		rgb(166pt)=(0.0238267,0.434427,0.875704);
		rgb(167pt)=(0.0248733,0.435473,0.875296);
		rgb(168pt)=(0.0259198,0.43652,0.874887);
		rgb(169pt)=(0.0269802,0.437553,0.874451);
		rgb(170pt)=(0.0280523,0.438574,0.873992);
		rgb(171pt)=(0.0291243,0.439595,0.873532);
		rgb(172pt)=(0.0301964,0.440616,0.873073);
		rgb(173pt)=(0.0312844,0.441621,0.872614);
		rgb(174pt)=(0.032382,0.442616,0.872154);
		rgb(175pt)=(0.0334796,0.443612,0.871695);
		rgb(176pt)=(0.0345772,0.444607,0.871235);
		rgb(177pt)=(0.0357108,0.445603,0.870758);
		rgb(178pt)=(0.0368595,0.446598,0.870273);
		rgb(179pt)=(0.0380081,0.447594,0.869788);
		rgb(180pt)=(0.0391568,0.448589,0.869303);
		rgb(181pt)=(0.0402652,0.449565,0.868798);
		rgb(182pt)=(0.0413628,0.450535,0.868287);
		rgb(183pt)=(0.0424604,0.451505,0.867777);
		rgb(184pt)=(0.043558,0.452474,0.867266);
		rgb(185pt)=(0.0445889,0.453444,0.866756);
		rgb(186pt)=(0.0456099,0.454414,0.866245);
		rgb(187pt)=(0.0466309,0.455384,0.865735);
		rgb(188pt)=(0.047652,0.456354,0.865224);
		rgb(189pt)=(0.0486,0.457324,0.864714);
		rgb(190pt)=(0.0495444,0.458294,0.864203);
		rgb(191pt)=(0.0504889,0.459264,0.863692);
		rgb(192pt)=(0.0514315,0.460234,0.863181);
		rgb(193pt)=(0.0523249,0.461204,0.862645);
		rgb(194pt)=(0.0532183,0.462174,0.862109);
		rgb(195pt)=(0.0541117,0.463144,0.861573);
		rgb(196pt)=(0.0549991,0.464111,0.861034);
		rgb(197pt)=(0.0558414,0.465056,0.860472);
		rgb(198pt)=(0.0566838,0.466,0.859911);
		rgb(199pt)=(0.0575261,0.466944,0.859349);
		rgb(200pt)=(0.0583532,0.467889,0.858793);
		rgb(201pt)=(0.0591189,0.468833,0.858257);
		rgb(202pt)=(0.0598847,0.469778,0.857721);
		rgb(203pt)=(0.0606505,0.470722,0.857185);
		rgb(204pt)=(0.0614018,0.471667,0.856641);
		rgb(205pt)=(0.0621165,0.472611,0.85608);
		rgb(206pt)=(0.0628312,0.473556,0.855518);
		rgb(207pt)=(0.0635459,0.4745,0.854957);
		rgb(208pt)=(0.064242,0.475444,0.854405);
		rgb(209pt)=(0.0649057,0.476389,0.853868);
		rgb(210pt)=(0.0655694,0.477333,0.853332);
		rgb(211pt)=(0.066233,0.478278,0.852796);
		rgb(212pt)=(0.0668625,0.479222,0.852249);
		rgb(213pt)=(0.0674495,0.480167,0.851687);
		rgb(214pt)=(0.0680366,0.481111,0.851126);
		rgb(215pt)=(0.0686237,0.482056,0.850564);
		rgb(216pt)=(0.0691838,0.483,0.850003);
		rgb(217pt)=(0.0697198,0.483944,0.849441);
		rgb(218pt)=(0.0702559,0.484889,0.84888);
		rgb(219pt)=(0.0707919,0.485833,0.848318);
		rgb(220pt)=(0.0712967,0.486778,0.847772);
		rgb(221pt)=(0.0717817,0.487722,0.847236);
		rgb(222pt)=(0.0722667,0.488667,0.8467);
		rgb(223pt)=(0.0727517,0.489611,0.846164);
		rgb(224pt)=(0.0732012,0.490573,0.845628);
		rgb(225pt)=(0.0736351,0.491543,0.845092);
		rgb(226pt)=(0.0740691,0.492513,0.844556);
		rgb(227pt)=(0.074503,0.493483,0.84402);
		rgb(228pt)=(0.0748973,0.494433,0.843484);
		rgb(229pt)=(0.0752802,0.495378,0.842948);
		rgb(230pt)=(0.0756631,0.496322,0.842412);
		rgb(231pt)=(0.0760459,0.497267,0.841876);
		rgb(232pt)=(0.0763631,0.498233,0.841362);
		rgb(233pt)=(0.0766694,0.499203,0.840851);
		rgb(234pt)=(0.0769757,0.500173,0.840341);
		rgb(235pt)=(0.077282,0.501143,0.83983);
		rgb(236pt)=(0.0775162,0.502137,0.83932);
		rgb(237pt)=(0.0777459,0.503132,0.838809);
		rgb(238pt)=(0.0779757,0.504128,0.838298);
		rgb(239pt)=(0.0782042,0.505123,0.837789);
		rgb(240pt)=(0.0783829,0.506093,0.837304);
		rgb(241pt)=(0.0785616,0.507063,0.836819);
		rgb(242pt)=(0.0787402,0.508033,0.836334);
		rgb(243pt)=(0.0789135,0.509008,0.835851);
		rgb(244pt)=(0.0790411,0.510029,0.835392);
		rgb(245pt)=(0.0791688,0.51105,0.834932);
		rgb(246pt)=(0.0792964,0.512071,0.834473);
		rgb(247pt)=(0.0794048,0.513092,0.834018);
		rgb(248pt)=(0.0794303,0.514113,0.833584);
		rgb(249pt)=(0.0794559,0.515134,0.83315);
		rgb(250pt)=(0.0794814,0.516155,0.832717);
		rgb(251pt)=(0.0794862,0.517183,0.832289);
		rgb(252pt)=(0.0794351,0.51823,0.831881);
		rgb(253pt)=(0.0793841,0.519276,0.831473);
		rgb(254pt)=(0.079333,0.520323,0.831064);
		rgb(255pt)=(0.079255,0.521369,0.830665);
		rgb(256pt)=(0.0791273,0.522416,0.830282);
		rgb(257pt)=(0.0789997,0.523462,0.829899);
		rgb(258pt)=(0.0788721,0.524509,0.829516);
		rgb(259pt)=(0.0786889,0.525589,0.829156);
		rgb(260pt)=(0.0784336,0.526712,0.828824);
		rgb(261pt)=(0.0781784,0.527835,0.828492);
		rgb(262pt)=(0.0779231,0.528958,0.82816);
		rgb(263pt)=(0.077615,0.530081,0.827868);
		rgb(264pt)=(0.0772577,0.531205,0.827613);
		rgb(265pt)=(0.0769003,0.532328,0.827357);
		rgb(266pt)=(0.0765429,0.533451,0.827102);
		rgb(267pt)=(0.0761243,0.534589,0.826862);
		rgb(268pt)=(0.0756649,0.535738,0.826632);
		rgb(269pt)=(0.0752054,0.536886,0.826403);
		rgb(270pt)=(0.0747459,0.538035,0.826173);
		rgb(271pt)=(0.0742168,0.539219,0.825961);
		rgb(272pt)=(0.0736553,0.540418,0.825756);
		rgb(273pt)=(0.0730937,0.541618,0.825552);
		rgb(274pt)=(0.0725321,0.542818,0.825348);
		rgb(275pt)=(0.0718925,0.544037,0.825183);
		rgb(276pt)=(0.0712288,0.545262,0.82503);
		rgb(277pt)=(0.0705652,0.546487,0.824877);
		rgb(278pt)=(0.0699015,0.547713,0.824723);
		rgb(279pt)=(0.0691514,0.548938,0.824614);
		rgb(280pt)=(0.0683856,0.550163,0.824511);
		rgb(281pt)=(0.0676198,0.551388,0.824409);
		rgb(282pt)=(0.0668541,0.552614,0.824307);
		rgb(283pt)=(0.0660408,0.553886,0.824205);
		rgb(284pt)=(0.065224,0.555162,0.824103);
		rgb(285pt)=(0.0644072,0.556439,0.824001);
		rgb(286pt)=(0.0635892,0.557715,0.823899);
		rgb(287pt)=(0.0626703,0.558991,0.823848);
		rgb(288pt)=(0.0617514,0.560268,0.823797);
		rgb(289pt)=(0.0608324,0.561544,0.823746);
		rgb(290pt)=(0.0599087,0.56282,0.823693);
		rgb(291pt)=(0.0589387,0.564096,0.823616);
		rgb(292pt)=(0.0579688,0.565373,0.82354);
		rgb(293pt)=(0.0569988,0.566649,0.823463);
		rgb(294pt)=(0.0560243,0.567925,0.823386);
		rgb(295pt)=(0.0550288,0.569202,0.82331);
		rgb(296pt)=(0.0540333,0.570478,0.823233);
		rgb(297pt)=(0.0530378,0.571754,0.823157);
		rgb(298pt)=(0.0520423,0.57303,0.82308);
		rgb(299pt)=(0.0510468,0.574307,0.823004);
		rgb(300pt)=(0.0500514,0.575583,0.822927);
		rgb(301pt)=(0.0490559,0.576859,0.82285);
		rgb(302pt)=(0.0480604,0.578127,0.822756);
		rgb(303pt)=(0.0470649,0.579377,0.822629);
		rgb(304pt)=(0.0460694,0.580628,0.822501);
		rgb(305pt)=(0.0450739,0.581879,0.822374);
		rgb(306pt)=(0.0441,0.583119,0.822235);
		rgb(307pt)=(0.0431556,0.584344,0.822082);
		rgb(308pt)=(0.0422111,0.585569,0.821929);
		rgb(309pt)=(0.0412667,0.586795,0.821776);
		rgb(310pt)=(0.0403351,0.58802,0.821597);
		rgb(311pt)=(0.0394162,0.589245,0.821392);
		rgb(312pt)=(0.0384973,0.59047,0.821188);
		rgb(313pt)=(0.0375784,0.591695,0.820984);
		rgb(314pt)=(0.0367495,0.592891,0.820735);
		rgb(315pt)=(0.0359838,0.594065,0.820454);
		rgb(316pt)=(0.035218,0.595239,0.820173);
		rgb(317pt)=(0.0344523,0.596413,0.819892);
		rgb(318pt)=(0.0337721,0.597553,0.819595);
		rgb(319pt)=(0.0331339,0.598676,0.819288);
		rgb(320pt)=(0.0324958,0.599799,0.818982);
		rgb(321pt)=(0.0318577,0.600923,0.818676);
		rgb(322pt)=(0.0312964,0.602026,0.818312);
		rgb(323pt)=(0.0307604,0.603124,0.817929);
		rgb(324pt)=(0.0302243,0.604222,0.817546);
		rgb(325pt)=(0.0296883,0.605319,0.817163);
		rgb(326pt)=(0.0292375,0.606395,0.816738);
		rgb(327pt)=(0.0288036,0.607468,0.816304);
		rgb(328pt)=(0.0283697,0.60854,0.81587);
		rgb(329pt)=(0.0279357,0.609612,0.815436);
		rgb(330pt)=(0.0275721,0.610637,0.814955);
		rgb(331pt)=(0.0272147,0.611658,0.81447);
		rgb(332pt)=(0.0268574,0.612679,0.813985);
		rgb(333pt)=(0.0265,0.6137,0.8135);
		rgb(334pt)=(0.0262447,0.614695,0.812964);
		rgb(335pt)=(0.0259895,0.615691,0.812428);
		rgb(336pt)=(0.0257342,0.616686,0.811892);
		rgb(337pt)=(0.0254853,0.61768,0.811352);
		rgb(338pt)=(0.0253066,0.61865,0.810765);
		rgb(339pt)=(0.0251279,0.61962,0.810177);
		rgb(340pt)=(0.0249492,0.62059,0.80959);
		rgb(341pt)=(0.024779,0.621551,0.808995);
		rgb(342pt)=(0.0246514,0.62247,0.808357);
		rgb(343pt)=(0.0245237,0.623389,0.807719);
		rgb(344pt)=(0.0243961,0.624308,0.80708);
		rgb(345pt)=(0.0242748,0.625221,0.80643);
		rgb(346pt)=(0.0241727,0.626114,0.805741);
		rgb(347pt)=(0.0240706,0.627008,0.805051);
		rgb(348pt)=(0.0239685,0.627901,0.804362);
		rgb(349pt)=(0.0238832,0.628786,0.803656);
		rgb(350pt)=(0.0238321,0.629654,0.802916);
		rgb(351pt)=(0.0237811,0.630522,0.802176);
		rgb(352pt)=(0.02373,0.631389,0.801435);
		rgb(353pt)=(0.023679,0.632247,0.800685);
		rgb(354pt)=(0.0236279,0.633089,0.799919);
		rgb(355pt)=(0.0235769,0.633932,0.799153);
		rgb(356pt)=(0.0235258,0.634774,0.798387);
		rgb(357pt)=(0.0234748,0.635604,0.797596);
		rgb(358pt)=(0.0234237,0.63642,0.79678);
		rgb(359pt)=(0.0233727,0.637237,0.795963);
		rgb(360pt)=(0.0233216,0.638054,0.795146);
		rgb(361pt)=(0.0232706,0.638856,0.794329);
		rgb(362pt)=(0.0232195,0.639647,0.793512);
		rgb(363pt)=(0.0231685,0.640439,0.792695);
		rgb(364pt)=(0.0231174,0.64123,0.791879);
		rgb(365pt)=(0.0230832,0.642005,0.791011);
		rgb(366pt)=(0.0230577,0.64277,0.790118);
		rgb(367pt)=(0.0230321,0.643536,0.789225);
		rgb(368pt)=(0.0230066,0.644302,0.788331);
		rgb(369pt)=(0.0229811,0.645049,0.787438);
		rgb(370pt)=(0.0229556,0.645789,0.786544);
		rgb(371pt)=(0.02293,0.646529,0.785651);
		rgb(372pt)=(0.0229045,0.647269,0.784758);
		rgb(373pt)=(0.022858,0.64801,0.783843);
		rgb(374pt)=(0.0228069,0.64875,0.782924);
		rgb(375pt)=(0.0227559,0.64949,0.782005);
		rgb(376pt)=(0.0227048,0.65023,0.781086);
		rgb(377pt)=(0.0227,0.650947,0.780144);
		rgb(378pt)=(0.0227,0.651662,0.7792);
		rgb(379pt)=(0.0227,0.652377,0.778256);
		rgb(380pt)=(0.0227,0.653092,0.777311);
		rgb(381pt)=(0.0228261,0.653781,0.776341);
		rgb(382pt)=(0.0229538,0.65447,0.775371);
		rgb(383pt)=(0.0230814,0.655159,0.774402);
		rgb(384pt)=(0.0232108,0.655849,0.77343);
		rgb(385pt)=(0.023364,0.656538,0.772434);
		rgb(386pt)=(0.0235171,0.657227,0.771439);
		rgb(387pt)=(0.0236703,0.657916,0.770443);
		rgb(388pt)=(0.0238312,0.658602,0.769444);
		rgb(389pt)=(0.0240354,0.659265,0.768423);
		rgb(390pt)=(0.0242396,0.659929,0.767402);
		rgb(391pt)=(0.0244438,0.660592,0.766381);
		rgb(392pt)=(0.0247021,0.661256,0.765354);
		rgb(393pt)=(0.025136,0.66192,0.764307);
		rgb(394pt)=(0.02557,0.662583,0.763261);
		rgb(395pt)=(0.0260039,0.663247,0.762214);
		rgb(396pt)=(0.0264541,0.663911,0.761168);
		rgb(397pt)=(0.026939,0.664574,0.760121);
		rgb(398pt)=(0.027424,0.665238,0.759074);
		rgb(399pt)=(0.027909,0.665902,0.758028);
		rgb(400pt)=(0.028445,0.666555,0.756971);
		rgb(401pt)=(0.0290577,0.667193,0.755899);
		rgb(402pt)=(0.0296703,0.667832,0.754827);
		rgb(403pt)=(0.0302829,0.66847,0.753755);
		rgb(404pt)=(0.030994,0.669095,0.752683);
		rgb(405pt)=(0.0318108,0.669708,0.751611);
		rgb(406pt)=(0.0326276,0.670321,0.750539);
		rgb(407pt)=(0.0334444,0.670933,0.749467);
		rgb(408pt)=(0.0343045,0.67156,0.748366);
		rgb(409pt)=(0.0351979,0.672198,0.747243);
		rgb(410pt)=(0.0360913,0.672837,0.74612);
		rgb(411pt)=(0.0369847,0.673475,0.744996);
		rgb(412pt)=(0.0380432,0.674096,0.743873);
		rgb(413pt)=(0.0391919,0.674709,0.74275);
		rgb(414pt)=(0.0403405,0.675322,0.741627);
		rgb(415pt)=(0.0414892,0.675934,0.740504);
		rgb(416pt)=(0.0427123,0.676528,0.739381);
		rgb(417pt)=(0.0439631,0.677115,0.738258);
		rgb(418pt)=(0.0452138,0.677702,0.737135);
		rgb(419pt)=(0.0464646,0.678289,0.736011);
		rgb(420pt)=(0.0477153,0.678897,0.734868);
		rgb(421pt)=(0.0489661,0.67951,0.733719);
		rgb(422pt)=(0.0502168,0.680123,0.73257);
		rgb(423pt)=(0.0514676,0.680735,0.731422);
		rgb(424pt)=(0.0529237,0.681325,0.73025);
		rgb(425pt)=(0.0544042,0.681912,0.729076);
		rgb(426pt)=(0.0558847,0.682499,0.727902);
		rgb(427pt)=(0.0573652,0.683086,0.726728);
		rgb(428pt)=(0.0587709,0.683673,0.725553);
		rgb(429pt)=(0.0601748,0.68426,0.724379);
		rgb(430pt)=(0.0615787,0.684847,0.723205);
		rgb(431pt)=(0.0629946,0.685435,0.722028);
		rgb(432pt)=(0.0646027,0.686022,0.720803);
		rgb(433pt)=(0.0662108,0.686609,0.719577);
		rgb(434pt)=(0.0678189,0.687196,0.718352);
		rgb(435pt)=(0.069427,0.687779,0.717131);
		rgb(436pt)=(0.0710351,0.688341,0.715931);
		rgb(437pt)=(0.0726432,0.688902,0.714731);
		rgb(438pt)=(0.0742514,0.689464,0.713532);
		rgb(439pt)=(0.0758709,0.690026,0.712326);
		rgb(440pt)=(0.07753,0.690587,0.711101);
		rgb(441pt)=(0.0791892,0.691149,0.709876);
		rgb(442pt)=(0.0808483,0.69171,0.70865);
		rgb(443pt)=(0.0825387,0.692272,0.707417);
		rgb(444pt)=(0.0843,0.692833,0.706167);
		rgb(445pt)=(0.0860613,0.693395,0.704916);
		rgb(446pt)=(0.0878225,0.693956,0.703665);
		rgb(447pt)=(0.089564,0.694518,0.702405);
		rgb(448pt)=(0.0912742,0.69508,0.701128);
		rgb(449pt)=(0.0929844,0.695641,0.699852);
		rgb(450pt)=(0.0946946,0.696203,0.698576);
		rgb(451pt)=(0.0965009,0.696752,0.697299);
		rgb(452pt)=(0.0984153,0.697288,0.696023);
		rgb(453pt)=(0.10033,0.697824,0.694747);
		rgb(454pt)=(0.102244,0.69836,0.693471);
		rgb(455pt)=(0.10413,0.698896,0.69218);
		rgb(456pt)=(0.105994,0.699432,0.690878);
		rgb(457pt)=(0.107857,0.699968,0.689577);
		rgb(458pt)=(0.10972,0.700505,0.688275);
		rgb(459pt)=(0.111632,0.701041,0.686973);
		rgb(460pt)=(0.113572,0.701577,0.685671);
		rgb(461pt)=(0.115512,0.702113,0.684369);
		rgb(462pt)=(0.117452,0.702649,0.683068);
		rgb(463pt)=(0.119429,0.703185,0.681747);
		rgb(464pt)=(0.12142,0.703721,0.68042);
		rgb(465pt)=(0.123411,0.704257,0.679093);
		rgb(466pt)=(0.125402,0.704793,0.677765);
		rgb(467pt)=(0.127372,0.705308,0.676438);
		rgb(468pt)=(0.129338,0.705819,0.675111);
		rgb(469pt)=(0.131303,0.706329,0.673783);
		rgb(470pt)=(0.133269,0.70684,0.672456);
		rgb(471pt)=(0.135369,0.70735,0.671084);
		rgb(472pt)=(0.137488,0.707861,0.669705);
		rgb(473pt)=(0.139607,0.708371,0.668327);
		rgb(474pt)=(0.141725,0.708882,0.666949);
		rgb(475pt)=(0.143795,0.709392,0.665595);
		rgb(476pt)=(0.145862,0.709903,0.664242);
		rgb(477pt)=(0.14793,0.710414,0.662889);
		rgb(478pt)=(0.150003,0.710924,0.661534);
		rgb(479pt)=(0.152198,0.711435,0.66013);
		rgb(480pt)=(0.154394,0.711945,0.658726);
		rgb(481pt)=(0.156589,0.712456,0.657322);
		rgb(482pt)=(0.158784,0.712963,0.655922);
		rgb(483pt)=(0.160979,0.713448,0.654543);
		rgb(484pt)=(0.163174,0.713933,0.653165);
		rgb(485pt)=(0.16537,0.714418,0.651786);
		rgb(486pt)=(0.16757,0.714908,0.650397);
		rgb(487pt)=(0.169791,0.715419,0.648968);
		rgb(488pt)=(0.172012,0.715929,0.647538);
		rgb(489pt)=(0.174232,0.71644,0.646109);
		rgb(490pt)=(0.176483,0.716935,0.64468);
		rgb(491pt)=(0.178806,0.717395,0.64325);
		rgb(492pt)=(0.181129,0.717854,0.641821);
		rgb(493pt)=(0.183452,0.718314,0.640391);
		rgb(494pt)=(0.185755,0.718783,0.638952);
		rgb(495pt)=(0.188027,0.719268,0.637497);
		rgb(496pt)=(0.190299,0.719753,0.636042);
		rgb(497pt)=(0.192571,0.720238,0.634587);
		rgb(498pt)=(0.194913,0.720711,0.633132);
		rgb(499pt)=(0.197338,0.72117,0.631677);
		rgb(500pt)=(0.199762,0.72163,0.630223);
		rgb(501pt)=(0.202187,0.722089,0.628768);
		rgb(502pt)=(0.204612,0.722549,0.627299);
		rgb(503pt)=(0.207037,0.723008,0.625818);
		rgb(504pt)=(0.209462,0.723468,0.624338);
		rgb(505pt)=(0.211887,0.723927,0.622857);
		rgb(506pt)=(0.214328,0.724386,0.621377);
		rgb(507pt)=(0.216778,0.724846,0.619896);
		rgb(508pt)=(0.219229,0.725305,0.618416);
		rgb(509pt)=(0.221679,0.725765,0.616935);
		rgb(510pt)=(0.224202,0.726188,0.615455);
		rgb(511pt)=(0.226754,0.726597,0.613974);
		rgb(512pt)=(0.229307,0.727005,0.612494);
		rgb(513pt)=(0.231859,0.727414,0.611014);
		rgb(514pt)=(0.234392,0.727842,0.609513);
		rgb(515pt)=(0.236919,0.728276,0.608007);
		rgb(516pt)=(0.239446,0.72871,0.606501);
		rgb(517pt)=(0.241973,0.729144,0.604995);
		rgb(518pt)=(0.244611,0.729556,0.603467);
		rgb(519pt)=(0.247266,0.729964,0.601935);
		rgb(520pt)=(0.24992,0.730372,0.600404);
		rgb(521pt)=(0.252575,0.730781,0.598872);
		rgb(522pt)=(0.25523,0.731189,0.597365);
		rgb(523pt)=(0.257884,0.731598,0.595859);
		rgb(524pt)=(0.260539,0.732006,0.594353);
		rgb(525pt)=(0.263194,0.732414,0.592846);
		rgb(526pt)=(0.265848,0.732796,0.591314);
		rgb(527pt)=(0.268503,0.733179,0.589783);
		rgb(528pt)=(0.271158,0.733562,0.588251);
		rgb(529pt)=(0.27383,0.733945,0.58672);
		rgb(530pt)=(0.276638,0.734328,0.585188);
		rgb(531pt)=(0.279446,0.734711,0.583657);
		rgb(532pt)=(0.282254,0.735094,0.582125);
		rgb(533pt)=(0.285051,0.735471,0.580594);
		rgb(534pt)=(0.287808,0.735829,0.579062);
		rgb(535pt)=(0.290565,0.736186,0.577531);
		rgb(536pt)=(0.293322,0.736544,0.575999);
		rgb(537pt)=(0.2961,0.736894,0.574468);
		rgb(538pt)=(0.298933,0.737226,0.572936);
		rgb(539pt)=(0.301767,0.737557,0.571405);
		rgb(540pt)=(0.3046,0.737889,0.569873);
		rgb(541pt)=(0.307452,0.738221,0.568351);
		rgb(542pt)=(0.310336,0.738553,0.566845);
		rgb(543pt)=(0.313221,0.738885,0.565339);
		rgb(544pt)=(0.316105,0.739217,0.563833);
		rgb(545pt)=(0.318978,0.739537,0.562315);
		rgb(546pt)=(0.321837,0.739843,0.560784);
		rgb(547pt)=(0.324696,0.74015,0.559252);
		rgb(548pt)=(0.327555,0.740456,0.557721);
		rgb(549pt)=(0.330468,0.740749,0.556216);
		rgb(550pt)=(0.333429,0.741029,0.554736);
		rgb(551pt)=(0.336389,0.74131,0.553255);
		rgb(552pt)=(0.33935,0.741591,0.551775);
		rgb(553pt)=(0.342296,0.741872,0.550279);
		rgb(554pt)=(0.345231,0.742153,0.548773);
		rgb(555pt)=(0.348167,0.742433,0.547267);
		rgb(556pt)=(0.351102,0.742714,0.545761);
		rgb(557pt)=(0.354038,0.742977,0.54429);
		rgb(558pt)=(0.356973,0.743232,0.542835);
		rgb(559pt)=(0.359908,0.743488,0.54138);
		rgb(560pt)=(0.362844,0.743743,0.539925);
		rgb(561pt)=(0.365839,0.743959,0.53847);
		rgb(562pt)=(0.368851,0.744163,0.537015);
		rgb(563pt)=(0.371863,0.744367,0.53556);
		rgb(564pt)=(0.374875,0.744571,0.534105);
		rgb(565pt)=(0.377843,0.744775,0.532672);
		rgb(566pt)=(0.380804,0.74498,0.531243);
		rgb(567pt)=(0.383765,0.745184,0.529814);
		rgb(568pt)=(0.386726,0.745388,0.528384);
		rgb(569pt)=(0.389711,0.745568,0.527003);
		rgb(570pt)=(0.392697,0.745747,0.525624);
		rgb(571pt)=(0.395684,0.745926,0.524246);
		rgb(572pt)=(0.39867,0.746104,0.522868);
		rgb(573pt)=(0.401657,0.746257,0.521489);
		rgb(574pt)=(0.404643,0.74641,0.520111);
		rgb(575pt)=(0.40763,0.746563,0.518732);
		rgb(576pt)=(0.410611,0.746716,0.517359);
		rgb(577pt)=(0.413546,0.746869,0.516032);
		rgb(578pt)=(0.416482,0.747023,0.514705);
		rgb(579pt)=(0.419417,0.747176,0.513377);
		rgb(580pt)=(0.422357,0.747319,0.512055);
		rgb(581pt)=(0.425318,0.747421,0.510753);
		rgb(582pt)=(0.428279,0.747523,0.509451);
		rgb(583pt)=(0.43124,0.747626,0.50815);
		rgb(584pt)=(0.43418,0.747735,0.506848);
		rgb(585pt)=(0.437065,0.747862,0.505546);
		rgb(586pt)=(0.439949,0.74799,0.504244);
		rgb(587pt)=(0.442834,0.748117,0.502942);
		rgb(588pt)=(0.445727,0.748227,0.501659);
		rgb(589pt)=(0.448637,0.748304,0.500408);
		rgb(590pt)=(0.451547,0.74838,0.499157);
		rgb(591pt)=(0.454457,0.748457,0.497906);
		rgb(592pt)=(0.457333,0.748522,0.496667);
		rgb(593pt)=(0.460167,0.748573,0.495441);
		rgb(594pt)=(0.463,0.748624,0.494216);
		rgb(595pt)=(0.465833,0.748675,0.492991);
		rgb(596pt)=(0.468667,0.748726,0.491779);
		rgb(597pt)=(0.4715,0.748777,0.490579);
		rgb(598pt)=(0.474333,0.748829,0.48938);
		rgb(599pt)=(0.477167,0.74888,0.48818);
		rgb(600pt)=(0.479969,0.748931,0.486995);
		rgb(601pt)=(0.482752,0.748982,0.485821);
		rgb(602pt)=(0.485534,0.749033,0.484647);
		rgb(603pt)=(0.488316,0.749084,0.483473);
		rgb(604pt)=(0.491081,0.7491,0.482316);
		rgb(605pt)=(0.493838,0.7491,0.481168);
		rgb(606pt)=(0.496595,0.7491,0.480019);
		rgb(607pt)=(0.499351,0.7491,0.47887);
		rgb(608pt)=(0.502069,0.74912,0.477722);
		rgb(609pt)=(0.504775,0.749145,0.476573);
		rgb(610pt)=(0.50748,0.749171,0.475424);
		rgb(611pt)=(0.510186,0.749196,0.474276);
		rgb(612pt)=(0.512892,0.7492,0.47317);
		rgb(613pt)=(0.515598,0.7492,0.472073);
		rgb(614pt)=(0.518303,0.7492,0.470975);
		rgb(615pt)=(0.521009,0.7492,0.469877);
		rgb(616pt)=(0.523644,0.749176,0.46878);
		rgb(617pt)=(0.526273,0.749151,0.467682);
		rgb(618pt)=(0.528902,0.749125,0.466585);
		rgb(619pt)=(0.531531,0.7491,0.465487);
		rgb(620pt)=(0.53416,0.749074,0.464415);
		rgb(621pt)=(0.536789,0.749049,0.463343);
		rgb(622pt)=(0.539418,0.749023,0.462271);
		rgb(623pt)=(0.542043,0.748998,0.461199);
		rgb(624pt)=(0.544621,0.748972,0.460127);
		rgb(625pt)=(0.547199,0.748947,0.459055);
		rgb(626pt)=(0.549777,0.748921,0.457983);
		rgb(627pt)=(0.55235,0.748891,0.45692);
		rgb(628pt)=(0.554903,0.74884,0.455899);
		rgb(629pt)=(0.557456,0.748789,0.454878);
		rgb(630pt)=(0.560008,0.748738,0.453857);
		rgb(631pt)=(0.562554,0.748687,0.452829);
		rgb(632pt)=(0.565081,0.748636,0.451783);
		rgb(633pt)=(0.567608,0.748585,0.450736);
		rgb(634pt)=(0.570135,0.748534,0.449689);
		rgb(635pt)=(0.572653,0.748474,0.44866);
		rgb(636pt)=(0.575155,0.748397,0.447665);
		rgb(637pt)=(0.577656,0.748321,0.446669);
		rgb(638pt)=(0.580158,0.748244,0.445674);
		rgb(639pt)=(0.582649,0.748168,0.444678);
		rgb(640pt)=(0.585125,0.748091,0.443683);
		rgb(641pt)=(0.587601,0.748014,0.442687);
		rgb(642pt)=(0.590077,0.747938,0.441692);
		rgb(643pt)=(0.59254,0.747861,0.440709);
		rgb(644pt)=(0.59499,0.747785,0.439739);
		rgb(645pt)=(0.597441,0.747708,0.438769);
		rgb(646pt)=(0.599891,0.747632,0.437799);
		rgb(647pt)=(0.602311,0.747555,0.436814);
		rgb(648pt)=(0.604711,0.747478,0.435819);
		rgb(649pt)=(0.60711,0.747402,0.434823);
		rgb(650pt)=(0.60951,0.747325,0.433828);
		rgb(651pt)=(0.611909,0.747232,0.432867);
		rgb(652pt)=(0.614308,0.747129,0.431922);
		rgb(653pt)=(0.616708,0.747027,0.430978);
		rgb(654pt)=(0.619107,0.746925,0.430033);
		rgb(655pt)=(0.621487,0.746823,0.429089);
		rgb(656pt)=(0.623861,0.746721,0.428144);
		rgb(657pt)=(0.626235,0.746619,0.4272);
		rgb(658pt)=(0.628609,0.746517,0.426256);
		rgb(659pt)=(0.630962,0.746393,0.425311);
		rgb(660pt)=(0.63331,0.746266,0.424367);
		rgb(661pt)=(0.635658,0.746138,0.423422);
		rgb(662pt)=(0.638007,0.746011,0.422478);
		rgb(663pt)=(0.640332,0.745906,0.421557);
		rgb(664pt)=(0.642654,0.745804,0.420638);
		rgb(665pt)=(0.644977,0.745702,0.419719);
		rgb(666pt)=(0.6473,0.7456,0.4188);
		rgb(667pt)=(0.649623,0.745472,0.417881);
		rgb(668pt)=(0.651946,0.745345,0.416962);
		rgb(669pt)=(0.654268,0.745217,0.416043);
		rgb(670pt)=(0.656587,0.745089,0.415124);
		rgb(671pt)=(0.658859,0.744962,0.414205);
		rgb(672pt)=(0.661131,0.744834,0.413286);
		rgb(673pt)=(0.663402,0.744707,0.412368);
		rgb(674pt)=(0.665674,0.744579,0.411453);
		rgb(675pt)=(0.667946,0.744451,0.410559);
		rgb(676pt)=(0.670218,0.744324,0.409666);
		rgb(677pt)=(0.672489,0.744196,0.408773);
		rgb(678pt)=(0.674755,0.744062,0.407879);
		rgb(679pt)=(0.677001,0.743909,0.406986);
		rgb(680pt)=(0.679247,0.743756,0.406092);
		rgb(681pt)=(0.681494,0.743603,0.405199);
		rgb(682pt)=(0.68374,0.743458,0.404306);
		rgb(683pt)=(0.685986,0.74333,0.403412);
		rgb(684pt)=(0.688232,0.743203,0.402519);
		rgb(685pt)=(0.690479,0.743075,0.401626);
		rgb(686pt)=(0.692704,0.742937,0.400732);
		rgb(687pt)=(0.694899,0.742784,0.399839);
		rgb(688pt)=(0.697094,0.742631,0.398945);
		rgb(689pt)=(0.699289,0.742477,0.398052);
		rgb(690pt)=(0.701497,0.742324,0.397171);
		rgb(691pt)=(0.703718,0.742171,0.396303);
		rgb(692pt)=(0.705939,0.742018,0.395435);
		rgb(693pt)=(0.708159,0.741865,0.394568);
		rgb(694pt)=(0.710351,0.741712,0.3937);
		rgb(695pt)=(0.71252,0.741559,0.392832);
		rgb(696pt)=(0.71469,0.741405,0.391964);
		rgb(697pt)=(0.71686,0.741252,0.391096);
		rgb(698pt)=(0.719029,0.741082,0.390228);
		rgb(699pt)=(0.721199,0.740904,0.38936);
		rgb(700pt)=(0.723369,0.740725,0.388492);
		rgb(701pt)=(0.725538,0.740546,0.387625);
		rgb(702pt)=(0.727708,0.740386,0.386757);
		rgb(703pt)=(0.729878,0.740233,0.385889);
		rgb(704pt)=(0.732047,0.74008,0.385021);
		rgb(705pt)=(0.734217,0.739927,0.384153);
		rgb(706pt)=(0.736366,0.739753,0.383285);
		rgb(707pt)=(0.73851,0.739574,0.382417);
		rgb(708pt)=(0.740654,0.739395,0.38155);
		rgb(709pt)=(0.742798,0.739217,0.380682);
		rgb(710pt)=(0.744919,0.739038,0.379837);
		rgb(711pt)=(0.747038,0.738859,0.378995);
		rgb(712pt)=(0.749156,0.738681,0.378152);
		rgb(713pt)=(0.751275,0.738502,0.37731);
		rgb(714pt)=(0.753394,0.738323,0.376442);
		rgb(715pt)=(0.755512,0.738145,0.375574);
		rgb(716pt)=(0.757631,0.737966,0.374707);
		rgb(717pt)=(0.75975,0.737789,0.373841);
		rgb(718pt)=(0.761868,0.737636,0.372998);
		rgb(719pt)=(0.763987,0.737483,0.372156);
		rgb(720pt)=(0.766105,0.73733,0.371314);
		rgb(721pt)=(0.76822,0.737169,0.370471);
		rgb(722pt)=(0.770313,0.736965,0.369629);
		rgb(723pt)=(0.772406,0.73676,0.368786);
		rgb(724pt)=(0.774499,0.736556,0.367944);
		rgb(725pt)=(0.776592,0.736358,0.367096);
		rgb(726pt)=(0.778686,0.736179,0.366228);
		rgb(727pt)=(0.780779,0.736001,0.36536);
		rgb(728pt)=(0.782872,0.735822,0.364492);
		rgb(729pt)=(0.784957,0.735643,0.363632);
		rgb(730pt)=(0.787024,0.735465,0.36279);
		rgb(731pt)=(0.789092,0.735286,0.361948);
		rgb(732pt)=(0.791159,0.735107,0.361105);
		rgb(733pt)=(0.793227,0.734929,0.360263);
		rgb(734pt)=(0.795295,0.73475,0.359421);
		rgb(735pt)=(0.797362,0.734571,0.358578);
		rgb(736pt)=(0.79943,0.734392,0.357736);
		rgb(737pt)=(0.801485,0.734214,0.356881);
		rgb(738pt)=(0.803527,0.734035,0.356014);
		rgb(739pt)=(0.805569,0.733856,0.355146);
		rgb(740pt)=(0.807611,0.733678,0.354278);
		rgb(741pt)=(0.809668,0.733499,0.353424);
		rgb(742pt)=(0.811735,0.73332,0.352582);
		rgb(743pt)=(0.813803,0.733142,0.35174);
		rgb(744pt)=(0.81587,0.732963,0.350897);
		rgb(745pt)=(0.817921,0.732784,0.350038);
		rgb(746pt)=(0.819963,0.732606,0.349171);
		rgb(747pt)=(0.822005,0.732427,0.348303);
		rgb(748pt)=(0.824047,0.732248,0.347435);
		rgb(749pt)=(0.826071,0.73207,0.346567);
		rgb(750pt)=(0.828087,0.731891,0.345699);
		rgb(751pt)=(0.830104,0.731712,0.344831);
		rgb(752pt)=(0.83212,0.731534,0.343963);
		rgb(753pt)=(0.834158,0.731355,0.343095);
		rgb(754pt)=(0.8362,0.731176,0.342228);
		rgb(755pt)=(0.838242,0.730998,0.34136);
		rgb(756pt)=(0.840284,0.730819,0.340492);
		rgb(757pt)=(0.842303,0.73064,0.339624);
		rgb(758pt)=(0.84432,0.730462,0.338756);
		rgb(759pt)=(0.846336,0.730283,0.337888);
		rgb(760pt)=(0.848353,0.730104,0.33702);
		rgb(761pt)=(0.850369,0.729926,0.336153);
		rgb(762pt)=(0.852386,0.729747,0.335285);
		rgb(763pt)=(0.854402,0.729568,0.334417);
		rgb(764pt)=(0.856419,0.729391,0.333546);
		rgb(765pt)=(0.858435,0.729238,0.332627);
		rgb(766pt)=(0.860452,0.729085,0.331708);
		rgb(767pt)=(0.862468,0.728932,0.330789);
		rgb(768pt)=(0.864481,0.728778,0.329874);
		rgb(769pt)=(0.866472,0.728625,0.32898);
		rgb(770pt)=(0.868463,0.728472,0.328087);
		rgb(771pt)=(0.870454,0.728319,0.327194);
		rgb(772pt)=(0.872445,0.728166,0.326295);
		rgb(773pt)=(0.874436,0.728013,0.325376);
		rgb(774pt)=(0.876427,0.727859,0.324457);
		rgb(775pt)=(0.878418,0.727706,0.323538);
		rgb(776pt)=(0.880417,0.727561,0.322619);
		rgb(777pt)=(0.882433,0.727433,0.3217);
		rgb(778pt)=(0.88445,0.727306,0.320781);
		rgb(779pt)=(0.886466,0.727178,0.319862);
		rgb(780pt)=(0.888463,0.72705,0.318933);
		rgb(781pt)=(0.890429,0.726923,0.317989);
		rgb(782pt)=(0.892394,0.726795,0.317044);
		rgb(783pt)=(0.894359,0.726668,0.3161);
		rgb(784pt)=(0.896337,0.726552,0.315132);
		rgb(785pt)=(0.898328,0.72645,0.314136);
		rgb(786pt)=(0.900319,0.726348,0.313141);
		rgb(787pt)=(0.90231,0.726246,0.312145);
		rgb(788pt)=(0.904301,0.726158,0.31115);
		rgb(789pt)=(0.906292,0.726081,0.310154);
		rgb(790pt)=(0.908283,0.726005,0.309159);
		rgb(791pt)=(0.910274,0.725928,0.308163);
		rgb(792pt)=(0.912249,0.725851,0.307151);
		rgb(793pt)=(0.914214,0.725775,0.30613);
		rgb(794pt)=(0.91618,0.725698,0.305109);
		rgb(795pt)=(0.918145,0.725622,0.304088);
		rgb(796pt)=(0.920111,0.7256,0.303031);
		rgb(797pt)=(0.922076,0.7256,0.301959);
		rgb(798pt)=(0.924041,0.7256,0.300886);
		rgb(799pt)=(0.926007,0.7256,0.299814);
		rgb(800pt)=(0.927972,0.7256,0.298722);
		rgb(801pt)=(0.929938,0.7256,0.297624);
		rgb(802pt)=(0.931903,0.7256,0.296527);
		rgb(803pt)=(0.933869,0.7256,0.295429);
		rgb(804pt)=(0.935812,0.725668,0.294264);
		rgb(805pt)=(0.937752,0.725744,0.29309);
		rgb(806pt)=(0.939692,0.725821,0.291916);
		rgb(807pt)=(0.941632,0.725897,0.290741);
		rgb(808pt)=(0.943571,0.726023,0.289518);
		rgb(809pt)=(0.945511,0.726151,0.288293);
		rgb(810pt)=(0.947451,0.726278,0.287068);
		rgb(811pt)=(0.949389,0.726411,0.285839);
		rgb(812pt)=(0.951278,0.726641,0.284537);
		rgb(813pt)=(0.953167,0.72687,0.283235);
		rgb(814pt)=(0.955056,0.7271,0.281933);
		rgb(815pt)=(0.956938,0.72734,0.280622);
		rgb(816pt)=(0.958776,0.727646,0.279243);
		rgb(817pt)=(0.960614,0.727952,0.277865);
		rgb(818pt)=(0.962451,0.728259,0.276486);
		rgb(819pt)=(0.964273,0.728597,0.275086);
		rgb(820pt)=(0.966034,0.729057,0.273606);
		rgb(821pt)=(0.967795,0.729516,0.272126);
		rgb(822pt)=(0.969557,0.729976,0.270645);
		rgb(823pt)=(0.971288,0.730473,0.269135);
		rgb(824pt)=(0.972947,0.73106,0.267552);
		rgb(825pt)=(0.974606,0.731647,0.265969);
		rgb(826pt)=(0.976265,0.732234,0.264387);
		rgb(827pt)=(0.977857,0.732879,0.262785);
		rgb(828pt)=(0.979338,0.733619,0.261151);
		rgb(829pt)=(0.980818,0.734359,0.259518);
		rgb(830pt)=(0.982299,0.735099,0.257884);
		rgb(831pt)=(0.983697,0.73591,0.256227);
		rgb(832pt)=(0.984999,0.736803,0.254542);
		rgb(833pt)=(0.986301,0.737697,0.252858);
		rgb(834pt)=(0.987603,0.73859,0.251173);
		rgb(835pt)=(0.988753,0.739566,0.249474);
		rgb(836pt)=(0.989774,0.740613,0.247764);
		rgb(837pt)=(0.990795,0.741659,0.246054);
		rgb(838pt)=(0.991816,0.742706,0.244344);
		rgb(839pt)=(0.992677,0.743816,0.242681);
		rgb(840pt)=(0.993443,0.744965,0.241048);
		rgb(841pt)=(0.994209,0.746114,0.239414);
		rgb(842pt)=(0.994975,0.747262,0.23778);
		rgb(843pt)=(0.995578,0.748465,0.236165);
		rgb(844pt)=(0.996114,0.74969,0.234557);
		rgb(845pt)=(0.99665,0.750915,0.232949);
		rgb(846pt)=(0.997186,0.752141,0.231341);
		rgb(847pt)=(0.997562,0.753386,0.229813);
		rgb(848pt)=(0.997893,0.754637,0.228307);
		rgb(849pt)=(0.998225,0.755887,0.226801);
		rgb(850pt)=(0.998557,0.757138,0.225295);
		rgb(851pt)=(0.998711,0.758433,0.223856);
		rgb(852pt)=(0.998839,0.759735,0.222426);
		rgb(853pt)=(0.998966,0.761037,0.220997);
		rgb(854pt)=(0.999094,0.762339,0.219567);
		rgb(855pt)=(0.999076,0.763641,0.218186);
		rgb(856pt)=(0.99905,0.764942,0.216808);
		rgb(857pt)=(0.999025,0.766244,0.21543);
		rgb(858pt)=(0.998995,0.767546,0.214054);
		rgb(859pt)=(0.998868,0.768848,0.212752);
		rgb(860pt)=(0.99874,0.77015,0.21145);
		rgb(861pt)=(0.998613,0.771451,0.210149);
		rgb(862pt)=(0.998473,0.772756,0.208856);
		rgb(863pt)=(0.998243,0.774083,0.207631);
		rgb(864pt)=(0.998014,0.775411,0.206405);
		rgb(865pt)=(0.997784,0.776738,0.20518);
		rgb(866pt)=(0.997539,0.77806,0.20396);
		rgb(867pt)=(0.997232,0.779362,0.20276);
		rgb(868pt)=(0.996926,0.780664,0.201561);
		rgb(869pt)=(0.99662,0.781966,0.200361);
		rgb(870pt)=(0.996299,0.783268,0.199168);
		rgb(871pt)=(0.995942,0.784569,0.197994);
		rgb(872pt)=(0.995584,0.785871,0.19682);
		rgb(873pt)=(0.995227,0.787173,0.195646);
		rgb(874pt)=(0.994842,0.788475,0.19449);
		rgb(875pt)=(0.994408,0.789777,0.193367);
		rgb(876pt)=(0.993974,0.791078,0.192244);
		rgb(877pt)=(0.99354,0.79238,0.191121);
		rgb(878pt)=(0.993083,0.793671,0.190021);
		rgb(879pt)=(0.992598,0.794947,0.188949);
		rgb(880pt)=(0.992113,0.796223,0.187877);
		rgb(881pt)=(0.991628,0.797499,0.186805);
		rgb(882pt)=(0.99113,0.798789,0.185732);
		rgb(883pt)=(0.990619,0.800091,0.18466);
		rgb(884pt)=(0.990109,0.801393,0.183588);
		rgb(885pt)=(0.989598,0.802695,0.182516);
		rgb(886pt)=(0.989072,0.803996,0.18146);
		rgb(887pt)=(0.988536,0.805298,0.180413);
		rgb(888pt)=(0.988,0.8066,0.179367);
		rgb(889pt)=(0.987464,0.807902,0.17832);
		rgb(890pt)=(0.98691,0.809186,0.177291);
		rgb(891pt)=(0.986349,0.810462,0.17627);
		rgb(892pt)=(0.985787,0.811738,0.175249);
		rgb(893pt)=(0.985226,0.813015,0.174228);
		rgb(894pt)=(0.984644,0.814311,0.173207);
		rgb(895pt)=(0.984057,0.815613,0.172186);
		rgb(896pt)=(0.98347,0.816914,0.171165);
		rgb(897pt)=(0.982883,0.818216,0.170144);
		rgb(898pt)=(0.982296,0.819518,0.169145);
		rgb(899pt)=(0.981709,0.82082,0.16815);
		rgb(900pt)=(0.981122,0.822122,0.167154);
		rgb(901pt)=(0.980535,0.823423,0.166159);
		rgb(902pt)=(0.979947,0.824725,0.165163);
		rgb(903pt)=(0.97936,0.826027,0.164168);
		rgb(904pt)=(0.978773,0.827329,0.163172);
		rgb(905pt)=(0.978186,0.828631,0.162177);
		rgb(906pt)=(0.977599,0.829932,0.161181);
		rgb(907pt)=(0.977012,0.831234,0.160186);
		rgb(908pt)=(0.976425,0.832536,0.15919);
		rgb(909pt)=(0.975838,0.833841,0.158195);
		rgb(910pt)=(0.975251,0.835168,0.157199);
		rgb(911pt)=(0.974664,0.836495,0.156204);
		rgb(912pt)=(0.974077,0.837823,0.155208);
		rgb(913pt)=(0.973489,0.83915,0.154213);
		rgb(914pt)=(0.972902,0.840477,0.153217);
		rgb(915pt)=(0.972315,0.841805,0.152222);
		rgb(916pt)=(0.971728,0.843132,0.151226);
		rgb(917pt)=(0.971155,0.844466,0.150224);
		rgb(918pt)=(0.970619,0.845819,0.149203);
		rgb(919pt)=(0.970083,0.847172,0.148182);
		rgb(920pt)=(0.969547,0.848525,0.147161);
		rgb(921pt)=(0.96902,0.849886,0.14614);
		rgb(922pt)=(0.968509,0.851265,0.145119);
		rgb(923pt)=(0.967999,0.852643,0.144098);
		rgb(924pt)=(0.967488,0.854022,0.143077);
		rgb(925pt)=(0.967,0.855411,0.142056);
		rgb(926pt)=(0.966541,0.856815,0.141035);
		rgb(927pt)=(0.966081,0.858219,0.140014);
		rgb(928pt)=(0.965622,0.859623,0.138992);
		rgb(929pt)=(0.965189,0.86104,0.137945);
		rgb(930pt)=(0.96478,0.862469,0.136873);
		rgb(931pt)=(0.964372,0.863899,0.135801);
		rgb(932pt)=(0.963963,0.865328,0.134729);
		rgb(933pt)=(0.96357,0.866773,0.133657);
		rgb(934pt)=(0.963187,0.868228,0.132585);
		rgb(935pt)=(0.962805,0.869683,0.131513);
		rgb(936pt)=(0.962422,0.871138,0.130441);
		rgb(937pt)=(0.962091,0.87261,0.129351);
		rgb(938pt)=(0.961785,0.874091,0.128253);
		rgb(939pt)=(0.961478,0.875571,0.127156);
		rgb(940pt)=(0.961172,0.877052,0.126058);
		rgb(941pt)=(0.960885,0.878571,0.124961);
		rgb(942pt)=(0.960605,0.880103,0.123863);
		rgb(943pt)=(0.960324,0.881634,0.122765);
		rgb(944pt)=(0.960043,0.883166,0.121668);
		rgb(945pt)=(0.959849,0.884719,0.120549);
		rgb(946pt)=(0.95967,0.886276,0.119426);
		rgb(947pt)=(0.959491,0.887833,0.118302);
		rgb(948pt)=(0.959313,0.88939,0.117179);
		rgb(949pt)=(0.959181,0.890995,0.116032);
		rgb(950pt)=(0.959054,0.892603,0.114884);
		rgb(951pt)=(0.958926,0.894211,0.113735);
		rgb(952pt)=(0.958799,0.895819,0.112587);
		rgb(953pt)=(0.958748,0.897453,0.111464);
		rgb(954pt)=(0.958697,0.899086,0.110341);
		rgb(955pt)=(0.958646,0.90072,0.109217);
		rgb(956pt)=(0.958602,0.902359,0.108089);
		rgb(957pt)=(0.958628,0.904043,0.106915);
		rgb(958pt)=(0.958653,0.905728,0.105741);
		rgb(959pt)=(0.958679,0.907413,0.104567);
		rgb(960pt)=(0.958718,0.909102,0.103393);
		rgb(961pt)=(0.95882,0.910812,0.102219);
		rgb(962pt)=(0.958922,0.912522,0.101044);
		rgb(963pt)=(0.959024,0.914232,0.0998703);
		rgb(964pt)=(0.959153,0.915962,0.0986961);
		rgb(965pt)=(0.959357,0.917749,0.0975219);
		rgb(966pt)=(0.959561,0.919536,0.0963477);
		rgb(967pt)=(0.959765,0.921323,0.0951736);
		rgb(968pt)=(0.959996,0.923118,0.0939907);
		rgb(969pt)=(0.960277,0.924931,0.092791);
		rgb(970pt)=(0.960557,0.926743,0.0915913);
		rgb(971pt)=(0.960838,0.928555,0.0903916);
		rgb(972pt)=(0.961151,0.930378,0.0891919);
		rgb(973pt)=(0.961509,0.932216,0.0879922);
		rgb(974pt)=(0.961866,0.934054,0.0867925);
		rgb(975pt)=(0.962223,0.935892,0.0855928);
		rgb(976pt)=(0.96262,0.937768,0.0843802);
		rgb(977pt)=(0.963053,0.939683,0.083155);
		rgb(978pt)=(0.963487,0.941597,0.0819297);
		rgb(979pt)=(0.963921,0.943512,0.0807045);
		rgb(980pt)=(0.964415,0.945426,0.0794643);
		rgb(981pt)=(0.964951,0.947341,0.0782135);
		rgb(982pt)=(0.965487,0.949255,0.0769628);
		rgb(983pt)=(0.966023,0.951169,0.075712);
		rgb(984pt)=(0.966594,0.953118,0.0744441);
		rgb(985pt)=(0.967181,0.955083,0.0731679);
		rgb(986pt)=(0.967768,0.957049,0.0718916);
		rgb(987pt)=(0.968355,0.959014,0.0706153);
		rgb(988pt)=(0.96898,0.960999,0.0693006);
		rgb(989pt)=(0.969619,0.96299,0.0679733);
		rgb(990pt)=(0.970257,0.964981,0.0666459);
		rgb(991pt)=(0.970895,0.966972,0.0653186);
		rgb(992pt)=(0.971554,0.968984,0.0639486);
		rgb(993pt)=(0.972218,0.971001,0.0625703);
		rgb(994pt)=(0.972882,0.973017,0.0611919);
		rgb(995pt)=(0.973545,0.975034,0.0598135);
		rgb(996pt)=(0.974232,0.97705,0.058318);
		rgb(997pt)=(0.974922,0.979067,0.056812);
		rgb(998pt)=(0.975611,0.981083,0.055306);
		rgb(999pt)=(0.9763,0.9831,0.0538)
	}
}

\usetikzlibrary{arrows,shapes,positioning}
\usetikzlibrary{decorations.markings}
\tikzstyle arrowstyle=[scale=1]
\tikzstyle directed=[postaction={decorate,decoration={markings,
		mark=at position .65 with {\arrow[arrowstyle]{stealth}}}}]
\tikzstyle reverse directed=[postaction={decorate,decoration={markings,
		mark=at position .65 with {\arrowreversed[arrowstyle]{stealth};}}}]

\newcommand{\secref}[1]{Section~\ref{#1}}
\newcommand{\thmref}[1]{Theorem~\ref{#1}}

\newcommand{\defref}[1]{Definition~\ref{#1}}

\renewcommand{\corref}[1]{Corollary~\ref{#1}}
\newcommand{\figref}[1]{Figure~\ref{#1}}
\newcommand{\tabref}[1]{Table~\ref{#1}}
\newcommand{\algoref}[1]{Algorithm~\ref{#1}}
\newcommand{\assref}[1]{Assumption~\ref{#1}}

\usepackage{mathtools}
\DeclarePairedDelimiter{\ceil}{\lceil}{\rceil}

\newcommand{\N}{\mathbb{N}} 

\newcommand{\R}{\mathbb{R}}
\newcommand{\Rhoch}[2]{\ensuremath{\R^{{#1} \times {#2}}}}
\newcommand{\Domain}{\ensuremath{\mathcal{X}}} %moment vector
\newcommand{\dimension}{\ensuremath{d}} %moment vector
\newcommand{\timeint}{\ensuremath{T}} %time interval
\newcommand{\timevar}{\ensuremath{t}} %final time

\newcommand{\indicator}[1]{\ensuremath{\mathbbm{1}_{#1}}}

\newcommand{\Lp}[1]{\ensuremath{L_{#1}}}

\newcommand{\flux}{\ensuremath{\mathbf{f}}}

\newcommand{\x}{\ensuremath{x}}
\newcommand{\y}{\ensuremath{y}}

\newcommand{\dx}{\partial_{x}}
\newcommand{\dy}{\partial_{\y}}

\newcommand{\dt}{\partial_\timevar}
\newcommand{\de}{\partial}

\newcommand{\quadxpoint}[1]{\hat{\x}_{#1}}
\newcommand{\quadxweight}[1]{\hat{w}_{#1}}
\newcommand{\quadypoint}[1]{\hat{\y}_{#1}}
\newcommand{\quadyweight}[1]{\hat{w}_{#1}}
\newcommand{\quadxpointG}[1]{{\x}_{#1}}
\newcommand{\quadxweightG}[1]{{w}_{#1}}
\newcommand{\quadypointG}[1]{{\y}_{#1}}
\newcommand{\quadyweightG}[1]{{w}_{#1}}

\newcommand{\quadRpoint}[1]{\hat{\uncertainty}_{#1}}
\newcommand{\quadRweight}[1]{\hat{\omega}_{#1}}

\newcommand{\support}[1]{\text{supp}\left( #1 \right)}

\newcommand{\dtstepsize}{\ensuremath{\Delta \timevar}}

\newcommand{\ncells}{\ensuremath{N_{\Domain}}}
\newcommand{\nRcells}{\ensuremath{N_{\randomSpace}}}

\newcommand{\timeind}{\ensuremath{n}}
\newcommand{\timepoint}[1]{\ensuremath{\timevar_{#1}}}
\newcommand{\cellind}{\ensuremath{i,j}}
\newcommand{\cellindx}{\ensuremath{i}}
\newcommand{\cellindy}{\ensuremath{j}}
\newcommand{\cellindR}{\ensuremath{l}}

\newcommand{\cell}[1]{\ensuremath{C_{#1}}}
\newcommand{\cellR}[1]{\ensuremath{D_{#1}}}

\newcommand{\polybasis}[1]{\ensuremath{\varphi_{#1}}}
\newcommand{\polybasisR}[1]{\ensuremath{\phi_{#1}}}

\newcommand{\spatialorder}{\ensuremath{{K_\Domain}}}

\newcommand{\polydegree}{\ensuremath{\spatialorder}}
\newcommand{\polydegreeR}{\ensuremath{\SGtruncorder}}
\newcommand{\cellmeant}[2][\cellind]{\ensuremath{\overline{#2}_{#1}}}
\newcommand{\cellmean}[3][\cellind]{\ensuremath{\overline{#2}^{(#3)}_{#1}}}

\newcommand{\limitervariable}{\ensuremath{\theta}}
\newcommand{\limitertolerance}{\ensuremath{\varepsilon}}
\newcommand{\hyperbolLimit}[1]{\ensuremath{\Lambda \Pi^{#1}}}

\newcommand{\Testfunction}[1][ ]{\ensuremath{v_{#1}}}

\newcommand{\SpaceOfPolynomials}[1]{\ensuremath{P^{#1}}}
\newcommand{\numericalFlux}{\ensuremath{\widehat{\flux}}}

\newcommand{\density}{\ensuremath{\rho}}

\newcommand{\energy}{\ensuremath{E}}
\newcommand{\pressure}{\ensuremath{p}}
\newcommand{\eulergamma}{\ensuremath{\gamma}}

\newcommand{\momentum}{\ensuremath{m}}
\newcommand{\momentumdim}[1]{\ensuremath{m}_{#1}}

\newcommand{\densityvar}{\ensuremath{\widetilde{\rho}}}
\newcommand{\energyvar}{\ensuremath{\widetilde{E}}}
\newcommand{\momentumdimvar}[1]{\ensuremath{\widetilde{m}}_{#1}}

\newcommand{\uncertainty}{\ensuremath{\boldsymbol{\xi}}}
\newcommand{\uncertaintyD}{\ensuremath{\xi}}
\newcommand{\xiPDF}{\ensuremath{f_\Xi}}
\newcommand{\xiBasisPoly}[1]{\ensuremath{\phi_{#1}}}
\newcommand{\SGsumIndex}{\ensuremath{\kappa}}
\newcommand{\SGsumIndexvar}{\ensuremath{\tilde{\kappa}}}
\newcommand{\xsumIndex}{\ensuremath{m}}
\newcommand{\xsumIndexvar}{\ensuremath{\tilde{m}}}
\newcommand{\SGeqIndex}{\ensuremath{j}}
\newcommand{\xiPDFdxi}{\ensuremath{\xiPDF \mathrm{d}\uncertainty}}
\newcommand{\intRS}{\ensuremath{\int_{\randomSpace}}}
\newcommand{\intcell}[2]{\ensuremath{\int\limits_{\cell{#1}} #2\,\mathrm{d}(\x,\y)}}

\newcommand{\MEintcellR}[2]{\ensuremath{\int\limits_{\cellR{#1}} #2\,\MExiPDFdxi}}

\newcommand{\MEintxR}[1]{\intcell{\cellind}{\!\MEintcellR{\cellindR}{#1}}}
\newcommand{\SGtruncorder}{\ensuremath{{K_\randomSpace}}}

\newcommand{\nbxnodes}{\ensuremath{Q_\Domain}}
\newcommand{\nbRnodes}{\ensuremath{Q_\randomSpace}}
\newcommand{\nbxiQuadNodes}{\ensuremath{Q_\Gamma}}

\newcommand{\xiQuadIndex}{\ensuremath{\rho}}%\reflectbox{\rotatebox[origin=c]{30}{\ensuremath{\rho}}}}
\newcommand{\xQuadIndex}{\ensuremath{\alpha}}%\mathfrak{q}}}
\newcommand{\yQuadIndex}{\ensuremath{\beta}}%\mathfrak{q}}}
\newcommand{\xQuadIndexG}{\ensuremath{p}}%\mathfrak{q}}}
\newcommand{\yQuadIndexG}{\ensuremath{q}}%\mathfrak{q}}}

\newcommand{\SGmomentvec}{\ensuremath{\bU}}
\newcommand{\SGflux}{\ensuremath{\bF}}

\newcommand{\realizableSet}{\ensuremath{\mathcal{R}}}

\newcommand{\sampleSpace}{\ensuremath{\Omega}}
\newcommand{\sigmaAlgebra}{\ensuremath{\mathcal{F}}}
\newcommand{\probabilityMeasure}{\ensuremath{\mathbb{P}}}
\newcommand{\randomSpace}{\ensuremath{\Gamma}}

\newcommand{\solution}{\ensuremath{\bu}}

\newcommand{\solutionprojected}{\ensuremath{\bu}}
\newcommand{\localsolution}[2]{\ensuremath{\bu}_{#1,#2}}
\newcommand{\auxiliaryprojected}{\ensuremath{\bw_h}}
\newcommand{\rhs}{\ensuremath{L_{\Delta x,\Delta y}}}
\newcommand{\tvbminmod}{\ensuremath{\Lambda \Pi_{\Delta x,\Delta y}}}
\newcommand{\DGmoment}[2]{\ensuremath{\bu_{#1,#2}}}
\newcommand{\MEDGmoment}[3]{\ensuremath{\bu_{#1,#2,#3}}}

\newcommand{\stageind}{s}
\newcommand{\stage}{S}
\newcommand{\RKsolution}{\bv}

\newcommand{\cutfun}[1]{h\left(#1\right)}
\newcommand{\hyperbolicity}{hyperbolicity}
\newcommand{\hyperbolic}{admissible}
\newcommand{\hypset}{hyperbolicity set}

\newcommand{\hdSG}{hDSG}

\newcommand{\randomElement}[1]{\ensuremath{D_{#1}}}
\newcommand{\indicatorVar}[2]{\ensuremath{\chi_{#1}(#2)}}
\newcommand{\inverseIndicator}[2]{\ensuremath{\chi_{#1}^{-1}(#2)}}
\newcommand{\MEIndex}{\ensuremath{l}}
\newcommand{\MEElements}{\ensuremath{{N_{\randomSpace}}}}
\newcommand{\localxiBasisPoly}[2]{\ensuremath{\phi_{#1,#2} }}

\newcommand{\E}{\ensuremath{\mathbb{E}}}
\DeclareMathOperator{\Var}{\text{Var}}

\pgfplotscreateplotcyclelist{color parula}{% 
	royalblue!70,every mark/.append style={solid,line width = 0pt,fill=royalblue!60!black},mark=ball\\%
	goldenrod!50!yelloworange,every mark/.append style={solid,fill=goldenrod!30!black},mark=*\\%
	green,every mark/.append style={black,fill=yellowgreen},mark=diamond*\\%
	bluegreen,every mark/.append style={fill=black},mark=triangle*\\%	
	royalblue!70,densely dashed,every mark/.append style={solid,line width = 0pt,fill=royalblue!60!black},mark=ball\\%
	goldenrod!50!yelloworange,densely dashed,every mark/.append style={solid,black,fill=goldenrod},mark=diamond*\\%
	yellowgreen,densely dashed,every mark/.append style={solid,fill=yellowgreen!60!royalblue},mark=square*, mark size= 1pt\\%
	bluegreen,densely dashed,mark size = 1.5pt,every mark/.append style={solid,fill=white},mark=otimes*\\%,densely dashed	
	royalblue,densely dashed,mark size = 1.5pt,every mark/.append style={solid,fill=royalblue!30!black},mark=pentagon*\\%,densely dashed
	goldenrod!40!yelloworange,every mark/.append style={solid,fill=goldenrod!30!black},mark=|\\%		
}

\pgfplotscreateplotcyclelist{color louisa}{% 
	blue,every mark/.append style={solid,line width = 0pt,fill=royalblue!60!black},mark=ball\\%
	red,every mark/.append style={solid,fill=orangered!30!black},mark=*\\%
	green,every mark/.append style={black,fill=yellowgreen},mark=diamond*\\%
	goldenrod!50!yelloworange,every mark/.append style={fill=black},mark=triangle*\\%	
	royalblue!70,densely dashed,every mark/.append style={solid,line width = 0pt,fill=royalblue!60!black},mark=ball\\%
	goldenrod!50!yelloworange,densely dashed,every mark/.append style={solid,black,fill=goldenrod},mark=diamond*\\%
	yellowgreen,densely dashed,every mark/.append style={solid,fill=yellowgreen!60!royalblue},mark=square*, mark size= 1pt\\%
	bluegreen,densely dashed,mark size = 1.5pt,every mark/.append style={solid,fill=white},mark=otimes*\\%,densely dashed	
	royalblue,densely dashed,mark size = 1.5pt,every mark/.append style={solid,fill=royalblue!30!black},mark=pentagon*\\%,densely dashed
	goldenrod!40!yelloworange,every mark/.append style={solid,fill=goldenrod!30!black},mark=|\\%		
}

\pgfplotscreateplotcyclelist{color fabian}{% 
	black,dashed\\%
	red,every mark/.append style={solid},mark=diamond*\\%
	blue,every mark/.append style={solid},mark=*\\%
	goldenrod!50!yelloworange,every mark/.append style={fill=black},mark=triangle*\\%	
	royalblue!70,densely dashed,every mark/.append style={solid,line width = 0pt,fill=royalblue!60!black},mark=ball\\%
	goldenrod!50!yelloworange,densely dashed,every mark/.append style={solid,black,fill=goldenrod},mark=diamond*\\%
	yellowgreen,densely dashed,every mark/.append style={solid,fill=yellowgreen!60!royalblue},mark=square*, mark size= 1pt\\%
	bluegreen,densely dashed,mark size = 1.5pt,every mark/.append style={solid,fill=white},mark=otimes*\\%,densely dashed	
	royalblue,densely dashed,mark size = 1.5pt,every mark/.append style={solid,fill=royalblue!30!black},mark=pentagon*\\%,densely dashed
	goldenrod!40!yelloworange,every mark/.append style={solid,fill=goldenrod!30!black},mark=|\\%		
}

\usepackage[colorinlistoftodos,obeyFinal]{todonotes}

\newcommand{\nstochdim}{\ensuremath{N}}
\newcommand{\nstochdimidx}{\ensuremath{n}}
\newcommand{\uncertaintydim}[1]{\xi^{#1}}
\newcommand{\xiPDFdim}[1]{\ensuremath{f_{\Xi_{#1}}}}
\newcommand{\randomSpacedim}[1]{\ensuremath{\Gamma_{#1}}}
\newcommand{\idxset}{\mathcal{K}}

\newcommand{\MEidxset}{\mathcal{L}}

\newcommand{\Mxidxset}{\mathcal{M}}
\newcommand{\Pidxset}{\mathcal{P}}
\newcommand{\totalidxset}{\N_0^{\nstochdim}}

\newcommand{\nbxiBasisPolyDim}{\ensuremath{N_{\SGtruncorder}}}
\newcommand{\SGapproach}{\ensuremath{\sum_{\SGsumIndex\in \idxset} \solution_\SGsumIndex \xiBasisPoly{\SGsumIndex}}}
\newcommand{\localuncertainty}[1]{\ensuremath{\uncertainty_{#1}}}
\newcommand{\localuncertaintydim}[2]{\ensuremath{\xi_{#1}^{#2}}}
\newcommand{\locxiPDF}[1]{\ensuremath{f}_{\randomElement{#1}}}

\newcommand{\MExiPDFdxi}{\ensuremath{\locxiPDF{_\MEIndex} \mathrm{d}\uncertainty_\MEIndex}}

\newcommand{\velocitydim}[1]{\ensuremath{v}_{#1}}

\newcommand{\spatialdimension}{\alpha} %moment vector
\newcommand{\inty}{\int \limits_{y_{j-\frac{1}{2}}}^{y_j+ \frac{1}{2}}}
\newcommand{\intx}{\int \limits_{x_{i-\frac{1}{2}}}^{x_i+ \frac{1}{2}}}
\newcommand{\dintx}{\mathrm{d}\x}
\newcommand{\dinty}{\mathrm{d}\y}

\newcommand{\gausslegendreorder}{\ensuremath{Q_L}}

%% file: Sections/introduction.tex
% !TeX spellcheck = en_US
\section{Introduction}
%\lsnote{Notation: Multi-Element, stochastic Galerkin, discontinuous Galerkin, SG, DG}
Non-deterministic effects may influence the validity of accurate approximations of deterministic systems %\cite{Poette2009,Kusch2015a}
because uncertain inputs in parameters, initial or boundary conditions yield a propagation of the uncertainty into the solution. Therefore, it is important to account for these uncertainties and hence Uncertainty Quantification (UQ) is becoming more and more important in numerical simulations.
The two main sources for these non-deterministic effects (or uncertainties) are the limitations in measuring  physical parameters exactly and the absence of knowledge of the underlying physical processes. 

In general there exist two major approaches to quantify the influence of uncertain parameters,
see for example \cite{Smith2014,Kolb2018,abgrall2017uncertainty, PetterssonIaccarinoNordstroem2015, MaitreKnio2010} for a general overview of different UQ methods.
Statistical approaches such as (multilevel) Monte Carlo (MC) type methods \cite{mishra2012multi,mishra2012sparse,mishra2016numerical} sample the random space to obtain statistical information,
like mean, variance or higher-order moments of the corresponding random field.
While they are very robust even for non-smooth problems like solutions of hyperbolic systems,
they suffer from a slow rate of convergence dictated
by the law of large numbers. 
In contrast, non-statistical approaches, like the intrusive and non-intrusive  Polynomial Chaos (PC) expansion \refone{\cite{XiuKarniadakis2002,Blatman2008,BabuskaTemponeZouraris2004,LeMaitreOlivier2002}}, or the Stochastic Collocation method \refone{\cite{Nobile2008,BabuskaNobileTempone2010,HesthavenXiu2005}}, approximate the random field
by a series of polynomials and derive deterministic models for the stochastic modes.
The theoretical foundation for the PC expansion has been laid in \cite{Wiener1938}
and can be described as a  polynomial approximation of Gaussian random variables to represent random processes.
Later, the approach has been generalized to a larger class of distributions \cite{XiuKarniadakis2002}, which is now known as
generalized Polynomial Chaos (gPC).
The intrusive PC expansion, also known as stochastic Galerkin (SG) approach,
considers a weak formulation of the partial differential equation with respect to the stochastic variable
and uses corresponding orthonormal polynomials as ansatz and test functions.

For many random elliptic and parabolic equations, the underlying random field is sufficiently smooth with respect to the stochastic variable and hence the use of
the SG method is superior compared to MC methods. This is mainly due to the fact that the gPC approximation exhibits spectral convergence \cite{GhanemSpanos1991, Yan2005, BabuskaTemponeZouraris2004}.  
However, the naive usage of the SG approach for nonlinear hyperbolic problems typically fails \cite{Poette2009,abgrall2017uncertainty} since the polynomial expansion of discontinuous data leads to huge oscillations (also known as Gibbs phenomenon). In the case of nonlinear systems of conservation laws, the resulting SG system might even \refone{lose} its hyperbolicity.

To reduce the Gibbs oscillations, the authors of \cite{WanKarniadakis2005} developed the so-called Multi-Element method
which subdivides the random space into smaller, disjoint elements. 
This corresponds to an $h$-refinement in the random space. 
Further developments of the Multi-Element approach encompass $h$- and $hp$-adaptive refinements in the stochastic space
(\cite{WanKarniadakis2005,WanKarniadakis2009, TryoenLeMaitreErn2012}) or a multi-resolution discretization using wavelets instead of gPC,
cf. \cite{BuergerKroekerRohde2014, PetterssonIaccarionNordstroem2014}.
Another approach which ensures hyperbolicity of the resulting nonlinear SG system
is the intrusive polynomial moment method. 
It bounds the oscillations of the Gibbs phenomenon by expanding the stochastic solution not in the conserved variables but 
in so-called entropic variables \cite{Poette2009}, which is well known in the radiative transfer community as minimum entropy models \cite{Lev84,Levermore1996,BruHol01}.
The resulting SG system is hyperbolic and has good approximation properties but requires to solve (typically) expensive nonlinear systems in every space-time cell.
Furthermore, it is necessary that the system possesses a strictly convex entropy function \refone{that ensures a hyperbolic system in order} to define the entropic variables.

In this article, we extend the hyperbolicity-preserving limiter developed in  \citep{Schlachter2017a} to (spatial) discontinuous Galerkin (DG) schemes ensuring the hyperbolicity of the resulting DG-SG system. Moreover, we combine the limiter with the Multi-Element approach to further decrease Gibbs oscillations and ensure a high-order approximation in both physical and stochastic space.
We compare the performance of our numerical scheme to that of the non-intrusive Stochastic Collocation method for a smooth
solution. In less than three stochastic dimensions 
our method proves to be more efficient than the Stochastic Collocation method. 
Furthermore, we apply our method to different one- and two-dimensional Riemann problems for the compressible Euler equations,
which shows that our method is able to handle complex flow problems.

%The mechanism that causes the classical stochastic Galerkin method to loose hyperbolicity has been observed before in the context of high-order discontinuous-Galerkin schemes for hyperbolic systems, especially for moment systems (see, e.g., \cite{Alldredge2015,Schneider2015b,Chidyagwai2017,Olbrant2012,Zhang2010,Zhang2011b,Schneider2016a}). We use a similar technique, a ``slope limiter'', to ``dampen'' the Gibbs oscillations in the stochastic expansion in such a way that the resulting system is always hyperbolic. Additionally, the method can be applied whenever the domain of hyperbolicity is explicitly available, even when there is no known entropy for the system.

The rest of the paper is organized as follows. In \secref{sec:model} we describe our hyperbolic model problem and the classical SG approach. Moreover, we explain the Multi-Element approach and define the domain of \hyperbolicity{}. Our modification of the SG method combined with DG and the hyperbolicity-preserving limiter is stated in \secref{sec:dg}, where we also prove its \hyperbolicity{}-preservation and embed the two-dimensional compressible Euler equations into our framework. \secref{sec:results} is devoted to a numerical convergence analysis and applications to the compressible Euler equations with different uncertain initial states in multiple stochastic dimensions, which we extensively investigate numerically.

%% file: Sections/model.tex
\section{Uncertainty Quantification}\label{sec:model}
Throughout this paper, we restrict ourselves to the study of two-dimensional random systems of hyperbolic conservation laws of the form
\begin{equation}\label{conslaw}\dt \solution +  \dx\flux_1 (\solution) + \dy\flux_2 (\solution) = 0,\end{equation}
with flux components $\flux_i({\solution}): \R^{\dimension} \rightarrow \R^{\dimension}$, $i=1,2$ in two spatial dimension $ (\x,\y)^\top
\in\Domain\subset\R^2$, and where the solution
$$\solution = \solution(\timevar,\x,\y,\uncertainty):\R_+ \times \R^2 \times \sampleSpace \rightarrow \R^{\dimension}$$
is depending on an $\nstochdim$-dimensional real-valued random vector
$\uncertainty=(\uncertaintydim{1},\ldots,\uncertaintydim{\nstochdim})$ defined on 
a probability space $(\sampleSpace,\sigmaAlgebra,\probabilityMeasure)$.
We assume that the entries of $\uncertainty$ are independent,
identically distributed random variables.
We denote the image of the random vector by $\randomSpace := \uncertainty(\sampleSpace)\subset \R^\nstochdim$ and its probability density function by
 $\xiPDF(\uncertainty) : \randomSpace \rightarrow \R_+$. Due to the independence of the entries of $\uncertainty$ 
we may write $\xiPDF(\uncertainty)= \prod \limits_{\nstochdimidx=1}^\nstochdim \xiPDFdim{\nstochdimidx}(\uncertaintydim{\nstochdimidx})$,
where  $\xiPDFdim{\nstochdimidx}$ is the probability density function of the random variable $\uncertaintydim{\nstochdimidx}$, for $\nstochdimidx=1,\ldots,\nstochdim$.
We further assume that the uncertainty is introduced via the initial conditions, namely
$$\solution(\timevar=0,\x,\y,\uncertainty) = \solution^0(\x,\y,\uncertainty).$$

We solve the system \eqref{conslaw} using a Runge--Kutta time-stepping method combined with a discontinuous Galerkin method in space (see \secref{sec:dg}).
The stochastic discretization uses the generalized Polynomial Chaos (gPC) theory \cite{Wiener1938, XiuKarniadakis2002}, which will be explained by the following stochastic Galerkin approach.

\subsection{Stochastic Galerkin} \label{subsec:Stochastic Galerkin}
The stochastic Galerkin (SG) method discretizes $\randomSpace$ by a suitable orthonormal basis.
For a multi-dimensional $\randomSpace\subset \R^\nstochdim$, we first introduce the family of multivariate orthonormal polynomials. 
To this end we consider the one-dimensional orthonormal polynomials $(\xiBasisPoly{\SGsumIndex}(\uncertaintydim{\nstochdimidx}))_{\SGsumIndex=0,\ldots,\infty}$,
$\nstochdimidx=1,\ldots \nstochdim$, w.r.t. 
to the inner product induced by the probability density function $\xiPDFdim{\nstochdimidx}$, i.e.,
\begin{equation} \label{eq:scalarproduct}\langle \mathbf{h}(\uncertaintydim{\nstochdimidx}),\,\bg(\uncertaintydim{\nstochdimidx})\rangle_{\randomSpacedim{\nstochdimidx}} := \int_{\omega\in\sampleSpace} \! \mathbf{h}(\uncertaintydim{\nstochdimidx}(\omega))\bg(\uncertaintydim{\nstochdimidx}(\omega)) \mathrm{d}\probabilityMeasure(\omega) = \int_{\uncertaintydim{\nstochdimidx}\in\randomSpacedim{\nstochdimidx}}\!\mathbf{h}(\uncertaintydim{\nstochdimidx})\bg(\uncertaintydim{\nstochdimidx})\xiPDFdim{\nstochdimidx}(\uncertaintydim{\nstochdimidx})\mathrm{d}{\uncertaintydim{\nstochdimidx}},\end{equation}
where $\randomSpacedim{\nstochdimidx} := \uncertaintydim{\nstochdimidx}(\Omega)$.
Thus, $(\xiBasisPoly{\SGsumIndex}(\uncertaintydim{\nstochdimidx}))_{\SGsumIndex=0,\ldots,\infty}$ form a basis of the (weighted) space $L_2(\randomSpacedim{\nstochdimidx})$.
For any multi-index $\SGsumIndex =( \SGsumIndex_1,\ldots,\SGsumIndex_{\nstochdim})^\top \in \totalidxset$,
we define the multivariate orthonormal polynomials  $\xiBasisPoly{\SGsumIndex}(\uncertainty):= \xiBasisPoly{\SGsumIndex_1}(\uncertaintydim{1})\cdot\ldots\cdot\xiBasisPoly{\SGsumIndex_\nstochdim}(\uncertaintydim{\nstochdim})$.
According to \cite{XiuKarniadakis2002}, we can write $\solution$ as the Fourier series
\begin{equation}\label{SGinfinitesum}\solution(\timevar,\x,\y,\uncertainty) = \sum_{\SGsumIndex \in \totalidxset} \solution_{\SGsumIndex}(\timevar,\x,\y) \xiBasisPoly{\SGsumIndex}(\uncertainty),\end{equation}
with deterministic coefficients $\solution_\SGsumIndex$. 
The solution $\solution$ of \eqref{conslaw} is now approximated by truncating the infinite sum \eqref{SGinfinitesum} at finite order $\SGtruncorder$.
\refone{We consider the complete polynomial space described by the following index set} $\SGsumIndex \in \idxset:=\{\SGsumIndex=(\SGsumIndex_1,\ldots\SGsumIndex_{\nstochdim})^\top \in \N_0^{\nstochdim} ~|~ \sum \limits_{\nstochdimidx=1}^{\nstochdim} \SGsumIndex_\nstochdimidx \leq \SGtruncorder \}$, cf. \cite{HesthavenXiu2005}.
Hence, the SG approximation reads as follows
\begin{equation}\label{SGapproach}\solution \approx \sum_{\SGsumIndex \in \idxset} \solution_{\SGsumIndex}(\timevar,\x,\y) \xiBasisPoly{\SGsumIndex}(\uncertainty),\end{equation}
which, by the Cameron-Martin theorem \refone{\cite{Ernst2012,Cameron1947}}, is converging to \eqref{SGinfinitesum} in $\Lp{2}(\randomSpace)$, as $\SGtruncorder\rightarrow\infty$.
The polynomial moments $\solution_{\SGsumIndex}$ of $\solution$ are determined by the Galerkin projection, given by
\begin{equation}\label{SGmoment}\solution_{\SGsumIndex} (\timevar,\x,\y) = \int_{\randomSpace} \solution(\timevar,\x,\y,\uncertainty) \xiBasisPoly{\SGsumIndex}(\uncertainty) \xiPDF \mathrm{d}\uncertainty.\end{equation}
We insert the ansatz \eqref{SGapproach} into the system of conservation laws \eqref{conslaw} and project the result onto the complete polynomial space.
We then obtain
$$\dt \intRS \left(\SGapproach\right)\!\xiBasisPoly{\SGeqIndex} \xiPDFdxi + \dx \intRS \flux_1 \!\left(\SGapproach\right) \!\xiBasisPoly{\SGeqIndex}\xiPDFdxi  + \dy \intRS \flux_2 \!\left(\SGapproach\right) \!\xiBasisPoly{\SGeqIndex}\xiPDFdxi= 0, \qquad \SGeqIndex \in \idxset.$$
The orthonormality of the basis functions yields the following SG system
\begin{equation}\label{SGsystem}
\dt \underbrace{\begin{pmatrix}
\solution_0\\ \vdots\\ \solution_{\nbxiBasisPolyDim}
\end{pmatrix}}_{=:\SGmomentvec} + 
 \dx\underbrace{\begin{pmatrix}
\intRS \flux_1\! \left( \SGapproach \right)\! \xiBasisPoly{0} \xiPDFdxi\\ \vdots\\ \intRS \flux_1\! \left( \SGapproach \right)\! \xiBasisPoly{{\nbxiBasisPolyDim}} \xiPDFdxi 
\end{pmatrix}}_{=:\SGflux_1\left(\SGmomentvec\right)}
+ 
 \dy\underbrace{\begin{pmatrix}
\intRS \flux_2\! \left( \SGapproach \right)\! \xiBasisPoly{0} \xiPDFdxi\\ \vdots\\ \intRS \flux_2\! \left( \SGapproach \right)\! \xiBasisPoly{{\nbxiBasisPolyDim}} \xiPDFdxi 
\end{pmatrix}}_{=:\SGflux_2\left(\SGmomentvec\right)}
= 0,
\end{equation}
with $\SGmomentvec,~\SGflux_\spatialdimension\!\left(\SGmomentvec\right)\in\R^{\dimension\nbxiBasisPolyDim}$, for $\spatialdimension=1,2$, 
where $\nbxiBasisPolyDim := {\nstochdim+ \SGtruncorder \choose \nstochdim}$ is the number of basis polynomials. The Jacobian matrix of this model reads
\begin{equation}\label{SGJacobian}\frac{\de \SGflux^\spatialdimension}{\de \SGmomentvec} = \begin{pmatrix}
\hat{\bF}_{00}^\spatialdimension & \cdots & \hat{\bF}_{0\nbxiBasisPolyDim}^\spatialdimension\\
\vdots & & \vdots\\
\hat{\bF}_{\nbxiBasisPolyDim 0}^\spatialdimension & \cdots & \hat{\bF}_{\nbxiBasisPolyDim\nbxiBasisPolyDim}^\spatialdimension
\end{pmatrix} ~ \in \Rhoch{\dimension\nbxiBasisPolyDim}{\dimension\nbxiBasisPolyDim},
\end{equation}
where
$$\hat{\bF}_{ji}^\spatialdimension =\intRS \frac{\de \flux_\spatialdimension}{\de \solution}\! \left(\SGapproach\right)\! \xiBasisPoly{j}\xiBasisPoly{i}\xiPDFdxi ~\in\Rhoch{\dimension}{\dimension}, \quad i,j = 0, \ldots, \nbxiBasisPolyDim,$$
for $\spatialdimension=1,2$.
The deterministic SG system \eqref{SGsystem} can then be solved by any suitable numerical method.
We will focus on the RKDG method, which will be described in detail in \secref{subsec:rkdgsg}. 
\refone{It has been shown in \cite{Poette2009} that the flux Jacobian \eqref{SGJacobian} may exhibit complex eigenvalues, hence the SG system \eqref{SGsystem} loses the hyperbolicity of the original conservation law.}

\subsection{Multi-Element Stochastic Galerkin}\label{subsec:Multielement Stochastic Galerkin}
\refone{As we have seen in the previous subsection, a major drawback of the plain SG approach for hyperbolic conservation laws is the possible loss of hyperbolicity. Moreover, for a discontinuous solution, the gPC approach may converge slowly or even fail to converge, due to Gibbs oscillations, cf. \cite{WanKarniadakis2005, Poette2009}}. To overcome this issue we employ the Multi-Element (ME) approach as presented in \cite{WanKarniadakis2006a}, where we subdivide $\randomSpace$ into disjoint elements and consider a local gPC approximation of  (\ref{conslaw}) on every random element.

For ease of presentation we assume that $\randomSpace=[0,1]^\nstochdim,$ and let $0= d_1 < d_2 < \ldots < d_{\MEElements+1}= 1$ be a decomposition of $[0,1]$.
We define $\randomElement{\nstochdimidx} := [d_{\nstochdimidx}, d_{\nstochdimidx+1} )$, for $ n= 1,\ldots, \MEElements-1,$ and 
$\randomElement{\MEElements} := [d_{\MEElements}, d_{\MEElements+1}]$.
Introducing the tensor-product index set $\MEidxset:= \{\MEIndex=(\MEIndex_1,\ldots,\MEIndex_\nstochdim)^\top  \in \N_0^{\nstochdim}: \MEIndex_{\nstochdimidx} \leq \MEElements,~ \nstochdimidx =1\ldots,\nstochdim \}
$ allows us to define the Multi-Element $\randomElement{\MEIndex} := \bigtimes \limits_{\nstochdimidx=1}^\nstochdim \randomElement{\MEIndex_\nstochdimidx}$ with $\Delta \uncertaintyD := \prod \limits_{\nstochdimidx=1}^\nstochdim (d_{\MEIndex_\nstochdimidx+1}-d_{\MEIndex_\nstochdimidx})$ and $\MEIndex\in \MEidxset$.
Moreover, we introduce an indicator variable $\chi_{\MEIndex}: \Omega \to  \{0,1\}$ on every random element
\begin{equation}\label{def:indicator}
\indicatorVar{\MEIndex}{\uncertainty(\omega)} :=
	\begin{cases}
	1, &\text{if } \uncertainty(\omega) \in \randomElement{\MEIndex}, \\
	0, &\text{else, }
	\end{cases}
\end{equation}
for $\MEIndex \in \MEidxset $. Using the indicator variable allows us to derive a disjoint partition of the sample space $\sampleSpace$ via
\begin{equation*}
\sampleSpace  = \bigcup \limits_{\MEIndex \in \MEidxset } \inverseIndicator{\MEIndex}{1}.
\end{equation*}
Hence, we can define a new local random vector $\localuncertainty{\MEIndex}=(\localuncertaintydim{\MEIndex}{1},\ldots,\localuncertaintydim{\MEIndex}{\nstochdim}): \inverseIndicator{\MEIndex}{1} \to \randomElement{\MEIndex}$ on the local probability space $(\inverseIndicator{\MEIndex}{1},\sigmaAlgebra\cap \inverseIndicator{\MEIndex}{1},\probabilityMeasure(\cdot | \indicatorVar{\MEIndex}{\uncertainty} =1))$. Using Bayes' rule we compute the following local probability 
density functions 
\begin{equation}\label{def:CondDensity}
\locxiPDF{\MEIndex}:=\locxiPDF{\MEIndex}(\localuncertainty{\MEIndex}|\indicatorVar{\MEIndex}{\uncertainty}=1) = \frac{\xiPDF(\localuncertainty{\MEIndex})}{\probabilityMeasure(\indicatorVar{\MEIndex}{\uncertainty}=1)}.
\end{equation}
\begin{remark}
	For a uniform distribution we have $\locxiPDF{\MEIndex}=\prod \limits_{\nstochdimidx=1}^\nstochdim \frac{1}{d_{\MEIndex_\nstochdimidx+1}-d_{\MEIndex_\nstochdimidx}}$.
\end{remark}
If we let $\{\localxiBasisPoly{\SGsumIndex}{\MEIndex} (\localuncertainty{\MEIndex})\}_{\SGsumIndex \in \totalidxset}$ be the orthonormal polynomials with respect to the conditional probability density function (\ref{def:CondDensity}), we may consider the local gPC approximation in element $\randomElement{\MEIndex}$,
\begin{equation}
\solution_{\MEIndex}(t,\x,\y,\localuncertainty{\MEIndex}) = \sum_{\SGsumIndex \in \totalidxset}  \solution_{\SGsumIndex, \MEIndex}(\timevar,\x,\y) \localxiBasisPoly{\SGsumIndex}{\MEIndex} (\localuncertainty{\MEIndex}) \approx
 \sum_{\SGsumIndex \in \idxset} \solution_{\SGsumIndex, \MEIndex}(\timevar,\x,\y) \localxiBasisPoly{\SGsumIndex}{\MEIndex} (\localuncertainty{\MEIndex}),
\end{equation}
for all $\MEIndex \in \MEidxset$.
The global approximation \ref{SGapproach} can then be written as
\begin{align}
\solution(t,\x,\y,\uncertainty) &= \sum_{\MEIndex \in \MEidxset} \solution_\MEIndex(t,\x,\y,\uncertainty) \indicatorVar{\MEIndex}{\uncertainty} \nonumber\\
& \refone{ \approx \sum_{\MEIndex \in \MEidxset}
  \sum_{\SGsumIndex\in \idxset} \solution_{\SGsumIndex, \MEIndex}(\timevar,\x,\y) \localxiBasisPoly{\SGsumIndex}{\MEIndex}(\uncertainty) \indicatorVar{\MEIndex}{\uncertainty},}\label{def:globalSG}
\end{align}
where the local approximation \eqref{SGapproach} converges to the global solution in $\Lp{2}(\randomSpace)$ as $\MEElements,\SGtruncorder\to\infty$, cf \cite{Alpert1993}.

\begin{remark}\label{rem:disjointME}
	Due to the disjoint decomposition of the random space, we may now apply the SG method from \secref{subsec:Stochastic Galerkin} on
	 every random element $\randomElement{\MEIndex}$, $\MEIndex \in \MEidxset$.
\end{remark}
The expected value of $\solution$ is given by its moment of zeroth order. We assume that
$ \localxiBasisPoly{0}{\MEIndex}(\localuncertainty{\MEIndex})=\prod_{\nstochdimidx=1}^{\nstochdim}\localxiBasisPoly{0}{\MEIndex}(\localuncertaintydim{\MEIndex}{\nstochdimidx})=1$ 
in $\randomElement{\MEIndex}$ and obtain the expected value and variance using the orthonormality of the basis polynomials as
\begin{align}
\E(\solution) &\approx  \int_{\randomSpace} \sum_{\MEIndex \in \MEidxset}   \sum_{\SGsumIndex\in \idxset}  \solution_{\SGsumIndex, \MEIndex}\, \localxiBasisPoly{\SGsumIndex}{\MEIndex}\, \indicatorVar{\MEIndex}{\uncertainty}\, \xiPDFdxi 
=  \sum_{\MEIndex \in \MEidxset}  \sum_{\SGsumIndex\in \idxset}  \solution_{\SGsumIndex, \MEIndex} \int_{\randomElement{\MEIndex}} \!\localxiBasisPoly{\SGsumIndex}{\MEIndex} \,\localxiBasisPoly{0}{\MEIndex} \, \probabilityMeasure(\indicatorVar{\MEIndex}{\uncertainty}=1)\,\MExiPDFdxi \nonumber 
\\ %\locxiPDF{\MEIndex}(\localuncertainty{\MEIndex}|\indicatorVar{\MEIndex}{\uncertainty}=1) \localdxi{\MEIndx}
&= \sum_{\MEIndex \in \MEidxset} \probabilityMeasure(\indicatorVar{\MEIndex}{\uncertainty}=1) \,\solution_{0,\MEIndex}, \label{SGEV}\\[0.2cm]
\Var(\solution) &\approx
\sum_{\MEIndex \in \MEidxset}\Big( \Var(\solution_{\MEIndex}) + \big(\solution_{0,\MEIndex} - \mathbb{E}(\solution)\big)^2\Big)\,\probabilityMeasure(\indicatorVar{\MEIndex}{\uncertainty}=1),\label{SGSD}
\end{align}
where the local variance $\Var(\solution_{\MEIndex})$ is given by
$$\Var(\solution_{\MEIndex})\approx \int_{\randomElement{\MEIndex}} \! \bigg( \sum_{\SGsumIndex\in \idxset} \solution_{\SGsumIndex,\MEIndex}\localxiBasisPoly{\SGsumIndex}{\MEIndex}\bigg)^2  \MExiPDFdxi - \solution_{0,\MEIndex}^2
%= \sum_{\SGsumIndex=0}^\SGtruncorder \solution_{\SGsumIndex,\MEIndex}^2 \int_{\randomElement{\MEIndex}} \!  \big(\localxiBasisPoly{\SGsumIndex}{\MEIndex}\big)^2  \, \MExiPDFdxi- \solution_{0,\MEIndex}^2 \int_{\randomElement{\MEIndex}} \!  \big(\localxiBasisPoly{0}{\MEIndex}\big)^2  \, \MExiPDFdxi 
=\sum_{\SGsumIndex\in \idxset \setminus (0,\ldots,0)} \solution_{\SGsumIndex,\MEIndex}^2.$$
\subsection{Hyperbolicity}
The solution of the system of equations \eqref{conslaw} has to fulfill certain physical properties. For example, the density or pressure should always be nonnegative for the Euler equations. \refone{We need to ensure the hyperbolicity of the system to satisfy these physical properties. This translates into the following definition of the hyperbolicity set.}
\begin{definition}
We call the set 
\begin{align*}
		\realizableSet := \left\{\solution\in\R^\dimension~\Big|~
		\alpha_1\frac{\partial\flux_1(\solution)}{\partial\solution} + \alpha_2\frac{\partial\flux_2(\solution)}{\partial\solution} \text{ has $\dimension$ real, distinct eigenvalues } \forall \alpha_1,\alpha_2\in\R, \alpha_1^2+\alpha_2^2=1
%		 \eigenvalue_\basisindx\left(\solution\right)\in\R~\forall\,\basisindx=1,\ldots,\dimension,~\eigenvalue_1<\eigenvalue_2 < \cdots < \eigenvalue_\dimension
		 \right\}
\end{align*}
	the \textbf{\hypset{}} and every solution vector $\solution\in\realizableSet$ \textbf{\hyperbolic{}}.
\end{definition}
\begin{assumption}
	\label{ass:RealizableSet}
	\refone{From now on, we assume} that the \hypset{} $\realizableSet$ is open and convex.
\end{assumption}
\refone{It has been shown in \cite{DespresPoetteLucor2013} that for non-scalar and non-symmetric models like the Euler equations, the SG system may lose its hyperbolicity and develops states outside $\realizableSet$. The shallow water example in  \cite{DespresPoetteLucor2013} even produced velocities that become greater than the speed of light by the usage of a non-hyperbolic model. This demonstrates the necessity of defining admissible states of $\realizableSet$ that are included in set of physical states. The following theorem states that the admissibility of the SG approximation translates to the hyperbolicity of the SG system.} 
\begin{theorem}\label{thm:SGsystemhyp}
	If the SG polynomial \eqref{def:globalSG} is admissible, \refone{ i.e. $$  \sum_{\MEIndex \in \MEidxset}
	\sum_{\SGsumIndex\in \idxset} \solution_{\SGsumIndex, \MEIndex}(\timevar,\x,\y) \localxiBasisPoly{\SGsumIndex}{\MEIndex}(\uncertainty) \indicatorVar{\MEIndex}{\uncertainty}\in\realizableSet,$$} then the SG system is hyperbolic.
\end{theorem}
\begin{proof}
	The proof can be found in \refone{\cite[Thm. 2.1]{Wu2017}}.
\end{proof}
In the following section we want to generalize the hyperbolic slope-limiter from \cite{Schlachter2017a}, which ensures hyperbolicity of the SG system, to high-order Runge--Kutta discontinuous Galerkin schemes.

%% file: Sections/dg.tex
% !TeX spellcheck = en_US
\section{Hyperbolicity-preserving discontinuous stochastic Galerkin scheme (\hdSG{})}\label{sec:dg}
\refone{In this section, we develop a hyperbolicity-preserving variant of the stochastic Galerkin method. Our aim is to construct a numerical scheme that is high order in space, time and the uncertainty. Additionally, we want to ensure that it preserves hyperbolicity of the resulting SG system. The idea is to guarantee the admissibility of the spatial-stochastic cell mean after one time step of the numerical scheme. This property is proven in \thmref{thm:realizable}. Using the slope limiting technique from \cite{Schlachter2017a}, we then provide a pointwise admissible SG approximation. In a first step to formulate the hdSG method we introduce the space-time discretization of the SG system using the Runge-Kutta discontinuous Galerkin scheme.}

\subsection{The Runge--Kutta discontinuous stochastic Galerkin scheme} \label{subsec:rkdgsg}
\refone{For the spatial discretization of the SG system \eqref{SGsystem}, we consider a discontinuous Galerkin (DG) scheme.}
Similar to \secref{subsec:Multielement Stochastic Galerkin}, where we subdivided the random space  $\randomSpace$ into Multi-Elements $\cellR{\cellindR}$, $\cellindR \in\MEidxset$, we now divide the spatial domain $\Domain \subset\R^2$ into a uniform rectangular mesh with cells $\cell{\cellind}= [\x_{i-\frac{1}{2}},\x_{i+\frac{1}{2}}]\times [\y_{{j-\frac{1}{2}}},\y_{j+\frac{1}{2}}]$ and cell-widths
$\Delta x:=(\x_{i+\frac{1}{2}}- \x_{i-\frac{1}{2}}) $, $\Delta y:= (\y_{j+\frac{1}{2}}- \y_{{j-\frac{1}{2}}})$.

For the proof of hyperbolicity-preservation in Theorem \ref{thm:realizable} we do not write down the DG spatial discretization of the moment-system \eqref{SGsystem} but rather consider the weak formulation of \eqref{conslaw} with respect to $(\x,\y)$ and also $\uncertainty$.
Therefore, we test \eqref{conslaw} with test functions $\Testfunction(\x,\y,\uncertainty)$ where $\support{\Testfunction}\subseteq\cell{\cellind}\times\cellR{\cellindR}$ and obtain, after one formal integration by parts, the following weak formulation
\begin{align}
\label{eq:weakformulation}
\dt \MEintxR{\solution\,\Testfunction} & = \MEintxR{\Big(\flux_1 (\solution)\,\dx\Testfunction+ \flux_2 (\solution)\,\dy\Testfunction \Big)} \\
& -\inty \MEintcellR{\MEIndex}{\bigg[\flux_1 (\solution)\Testfunction \bigg]_{\x_{\cellindx-\frac12}}^{\x_{\cellindx+\frac12}}}\dinty
-\intx \MEintcellR{\MEIndex}{\bigg[\flux_2 (\solution)\Testfunction \bigg]_{\y_{\cellindy-\frac12}}^{\y_{\cellindy+\frac12}}}\dintx.\nonumber
\end{align}
We  replace each component of the solution $\solution$ by a piecewise-polynomial approximation, i.e., $\solution_h \big|_{\cell{\cellind}\times\cellR{\cellindR}}\in\left(\SpaceOfPolynomials{\polydegree}(\cell{\cellind})\otimes\SpaceOfPolynomials{\polydegreeR}(\cellR{\cellindR})\right)^{\otimes \dimension}$ for
polynomial degrees $\polydegree,\polydegreeR \in \N_0$. As a basis of
$\SpaceOfPolynomials{\polydegree}(\cell{\cellind})$ we use tensor-products of one-dimensional polynomials.
Thus, we define the corresponding tensor-product index set $\xsumIndex\in \Mxidxset:=\{\xsumIndex=(\xsumIndex_1,\xsumIndex_2)^T\in\mathbb{N}_0^2~\big|~\xsumIndex_i\leq\polydegree, ~i=1,2\}$.
As a basis of 
$\SpaceOfPolynomials{\polydegreeR}(\cellR{\cellindR})$ we use the complete polynomial space from \secref{subsec:Stochastic Galerkin}.

For ease of notation, \refone{we drop the index $h$ and write $\solution$ instead of $\solution_h$}.
The local approximation in the bases of the spaces of piecewise polynomials then reads as follows
\begin{align}\label{localMEDGDGApproximation}
\localsolution{\cellind}{\cellindR}(\timevar,\x,\y,\localuncertainty{\MEIndex}) 
:= \solutionprojected \big|_{\cell{\cellind}\times\cellR{\cellindR}}(\timevar,\x,\y,\localuncertainty{\MEIndex}) = \sum_{\xsumIndex\in \Mxidxset} \sum_{\SGsumIndex\in\idxset}\MEDGmoment{\xsumIndex}{\SGsumIndex}{\cellind,\MEIndex}(\timevar)\polybasis{\xsumIndex,\cellind}(\x,\y)\localxiBasisPoly{\SGsumIndex}{\MEIndex}(\localuncertainty{\MEIndex}), 
\end{align}
for $(\x,\y)\in \cell{\cellind},~\localuncertainty{\MEIndex}\in\cellR{\cellindR}$.
Here, $\{\localxiBasisPoly{\SGsumIndex}{\MEIndex}\}_{\SGsumIndex\in\idxset}$ are the local basis polynomials on the random element $\cellR{\cellindR}$ and $\{\polybasis{\xsumIndex,\cellind}\}_{\xsumIndex\in \Mxidxset}$ are the basis polynomials on the physical cell $\cell{\cellind}$.
For a simpler notation we restrict ourselves to one fixed space-stochastic cell $\cell{\cellind}\times \cellR{\cellindR}$
and drop the additional indices $\cellind$ and $\MEIndex$ in \eqref{localMEDGDGApproximation}.

Plugging the ansatz \eqref{localMEDGDGApproximation} into the weak formulation \eqref{eq:weakformulation} with test functions $v=\polybasis{\xsumIndexvar}\xiBasisPoly{\SGsumIndexvar}$ yields
\begin{align}
\label{eq:DG}
 \sum_{\xsumIndex\in \Mxidxset} \sum_{\SGsumIndex\in\idxset}\dt\DGmoment{\xsumIndex}{\SGsumIndex}{} \MEintxR{\polybasis{\xsumIndex}\polybasis{\xsumIndexvar}\xiBasisPoly{\SGsumIndex}{}\xiBasisPoly{\SGsumIndexvar}} 
 &= \MEintxR{\Big( \flux_1 (\solutionprojected)\xiBasisPoly{\SGsumIndexvar}\,\dx\polybasis{\xsumIndexvar} +  \flux_2 (\solutionprojected)\xiBasisPoly{\SGsumIndexvar}\,\dy\polybasis{\xsumIndexvar} \Big)} 
 \\ &-
 \inty \MEintcellR{\cellindR}{\bigg[\flux_1(\solutionprojected)\,\polybasis{\xsumIndexvar} \bigg]_{\x_{\cellindx-\frac12}}^{\x_{\cellindx+\frac12}}\xiBasisPoly{\SGsumIndexvar}} \dinty
 \label{eq:fluxinty} \\
 &-
 \intx \MEintcellR{\cellindR}{\bigg[\flux_2(\solutionprojected)\,\polybasis{\xsumIndexvar} \bigg]_{\y_{\cellindy-\frac12}}^{\y_{\cellindy+\frac12}}\xiBasisPoly{\SGsumIndexvar}} \dintx.
 \label{eq:fluxintx}
\end{align}
\begin{remark}
Using the orthogonality of the local basis polynomials  $\{\localxiBasisPoly{\SGsumIndex}{\MEIndex}\}_{\SGsumIndex\in\idxset}$ on the random element $\cellR{\cellindR}$, 
we would come up with the DG spatial discretization of \eqref{SGsystem}.
\end{remark}
To approximate the integrals in \eqref{eq:DG},\eqref{eq:fluxinty} and \eqref{eq:fluxintx} numerically,
we use either Gau\ss-Legendre or Gau\ss-Lobatto quadrature rules.
We denote Gau\ss-Legendre quadrature rules with Latin letters and Gau\ss-Lobatto quadrature rules with Greek letters 
and add additional hats to the points and weights.
On the spatial cell $\cell{\cellind}$ we use a tensor-product of one-dimensional Gau\ss-Lobatto quadratures with $\nbxnodes+1 = \ceil*{\frac{\spatialorder+1}{2}}+1$ points and weights, denoted by  
$(\quadxpoint{\xQuadIndex},\quadxweight{\xQuadIndex})$ and 
$(\quadypoint{\yQuadIndex},\quadyweight{\yQuadIndex})$, for $\xQuadIndex,\yQuadIndex=0,\ldots,\nbxnodes$. 
To approximate the (spatially) one-dimensional integrals in \eqref{eq:fluxinty} and \eqref{eq:fluxintx}, we use Gau\ss-Legendre quadratures with sufficient accuracy.
Let us assume that we employ Gau\ss-Legendre quadratures with  
$\gausslegendreorder \in \N_0$ points and weights denoted by 
$(\quadxpointG{\xQuadIndexG},\quadxweightG{\xQuadIndexG})$ and 
$(\quadypointG{\yQuadIndexG},\quadyweightG{\yQuadIndexG})$, for $\xQuadIndexG,\yQuadIndexG=0,\ldots,\gausslegendreorder$. 

For a uniformly-distributed random variable and Legendre basis functions, we apply a Gau\ss{}-Lobatto quadrature rule on $\cellR{\cellindR}$ with order $\SGtruncorder$, i.e., $\nbRnodes+1$  points and weights $(\quadRpoint{\xiQuadIndex},\quadRweight{\xiQuadIndex})$, $\xiQuadIndex=0,\ldots,\nbRnodes$, where $\nbRnodes = \ceil*{\frac{\SGtruncorder+1}{2}}$. For a multi-dimensional random space $\randomSpace\subset \R^N$ we also use tensor-product quadrature rules
and introduce the multi-index set $\xiQuadIndex \in  \Pidxset:=\{\xiQuadIndex=(\xiQuadIndex_1,\ldots,\xiQuadIndex_N)^T\in\mathbb{N}_0^N~\big|~\rho_i\leq\nbRnodes, ~i=1,\ldots,N\}$.  For other probability distributions we use the corresponding Gauß quadrature based on the orthogonal basis polynomials and weighted by the probability density function $\locxiPDF{\MEIndex}$. We scale the quadrature weights such that
\begin{equation}\label{eq:quad}
\sum_{\xQuadIndex=0}^{\nbxnodes} \sum_{\yQuadIndexG=0}^{\gausslegendreorder}  \sum_{\xiQuadIndex\in\Pidxset}
\quadxweight{\xQuadIndex} \quadyweightG{\yQuadIndexG}\quadRweight{\xiQuadIndex}=1, \qquad 
\sum_{\xQuadIndexG=0}^{\gausslegendreorder} \sum_{\yQuadIndex=0}^{\nbxnodes}  \sum_{\xiQuadIndex\in\Pidxset}
\quadxweightG{\xQuadIndexG} \quadyweight{\yQuadIndex}\quadRweight{\xiQuadIndex}=1, \qquad
\MEintcellR{\cellindR}{\bg} \approx  \sum_{\xiQuadIndex\in\Pidxset}\bg(\quadRpoint{\xiQuadIndex})\quadRweight{\xiQuadIndex}.\end{equation}
\begin{remark}
The Gau\ss{}-Lobatto quadrature rule includes the endpoints, i.e. cell interfaces. This property will be used in the proof of \thmref{thm:realizable}.
\end{remark}
Since $\solutionprojected |_{\cell{\cellind}\times\cellR{\cellindR}}$ is discontinuous across the physical cell interfaces $\x_{\cellindx\pm\frac12}$, $\y_{\cellindy\pm\frac12}$, 
we replace the evaluation of the flux components $\flux_1$,$\flux_2$ at these points with a numerical flux function 
$\numericalFlux_1, \numericalFlux_2$, (approximately) solving the corresponding Riemann problem at the interface.
To this end we denote the spatial traces by  
\begin{align*}
\solution_{\cellindx +\frac12}^- 
:= \solution^-(\x_{\cellindx +\frac12},y,\localuncertainty{\MEIndex})
:= \lim\limits_{\x\uparrow\x_{\cellindx +\frac12}} \solutionprojected |_{\cell{\cellind}\times\cellR{\cellindR}}(\x,\y,\localuncertainty{\MEIndex}),
\\
\solution_{\cellindx +\frac12}^+ 
:= \solution^+(\x_{\cellindx +\frac12},y,\localuncertainty{\MEIndex}) 
:= \lim\limits_{\x\downarrow\x_{\cellindx +\frac12}} \solutionprojected |_{\cell{\cellind}\times\cellR{\cellindR}}(\x,\y,\localuncertainty{\MEIndex}), \\
\solution_{\cellindy +\frac12}^- 
:= \solution^-(\x,\y_{\cellindy +\frac12},\localuncertainty{\MEIndex})
:= \lim\limits_{\y\uparrow\y_{\cellindy +\frac12}} \solutionprojected |_{\cell{\cellind}\times\cellR{\cellindR}}(\x,\y,\localuncertainty{\MEIndex}),
\\
\solution_{\cellindy +\frac12}^+ 
:= \solution^+(\x,\y_{\cellindy +\frac12},\localuncertainty{\MEIndex})
:= \lim\limits_{\y\downarrow\y_{\cellindy +\frac12}} \solutionprojected |_{\cell{\cellind}\times\cellR{\cellindR}}(\x,\y,\localuncertainty{\MEIndex}).
\end{align*}
For our numerical experiments in \secref{sec:results} we choose either the Lax-Friedrichs numerical flux with 
\begin{itemize}
\item $\numericalFlux_1(\solution_{\cellindx +\frac12}^-,\solution_{\cellindx +\frac12}^+ ):= \frac{1}{2}\Big(\flux_1(\solution_{\cellindx +\frac12}^-)+
\flux_1(\solution_{\cellindx +\frac12}^+ )- \lambda_{\max}^1(\solution_{\cellindx +\frac12}^+-\solution_{\cellindx +\frac12}^-)\Big),$ 
\item $\numericalFlux_2(\solution_{\cellindy +\frac12}^-,\solution_{\cellindy +\frac12}^+ ):= \frac{1}{2}\Big(\flux_2(\solution_{\cellindy +\frac12}^-)+
\flux_1(\solution_{\cellindy +\frac12}^+ )- \lambda_{\max}^2(\solution_{\cellindy +\frac12}^+-\solution_{\cellindy +\frac12}^-)\Big),$
\end{itemize}
where the numerical viscosity constants $\lambda_{\max}^1,\lambda_{\max}^2$ are taken as the global estimate of the 
absolute value of the largest eigenvalue of  $\frac{\partial\flux_1(\solution)}{\partial\solution}$ and
$\frac{\partial\flux_2(\solution)}{\partial\solution}$.
Alternatively, we choose the HLLE numerical flux as in  \cite{EinfeldtMunzRoe1991}
\begin{itemize}
  \item $\numericalFlux_1(\solution_{\cellindx +\frac12}^-,\solution_{\cellindx +\frac12}^+ ):= 
\frac{b_{\cellindx+\frac12}^+ \flux_1 (\solution_{\cellindx +\frac12}^-)- b_{\cellindx+\frac12}^- \flux_1(\solution_{\cellindx +\frac12}^-)}{b_{\cellindx+\frac12}^+-b_{\cellindx+\frac12}^-} 
+ \frac{b_{\cellindx+\frac12}^+b_{\cellindx+\frac12}^-}{b_{\cellindx+\frac12}^+-b_{\cellindx+\frac12}^-}(\solution_{\cellindx +\frac12}^+ - \solution_{\cellindx +\frac12}^- )$,
\end{itemize}
with signal speed estimates
$b_{\cellindx+\frac12}^-:= \{ \lambda_{\min}^1(\bar{\solution}_{\cellindx +\frac12}),\lambda_{\min}^1({\solution}_{\cellindx +\frac12}^-),0 \}$,
$b_{\cellindx+\frac12}^+:= \{ \lambda_{\max}^1(\bar{\solution}_{\cellindx +\frac12}),\lambda_{\max}^1({\solution}_{\cellindx +\frac12}^+),0 \}$
 where $\lambda_{\min}^1(\solution), \lambda_{\max}^1(\solution)$ are the smallest and biggest eigenvalues of $\frac{\partial\flux_1(\solution)}{\partial\solution}$ 
and $\bar{\solution}_{\cellindx +\frac12}$ is the corresponding Roe mean value, cf. \cite{EinfeldtMunzRoe1991}. Analogously we define 
\begin{itemize}
  \item $\numericalFlux_2(\solution_{\cellindy +\frac12}^-,\solution_{\cellindy +\frac12}^+ ):= 
\frac{b_{\cellindy+\frac12}^+ \flux_2 (\solution_{\cellindy +\frac12}^-)- b_{\cellindy+\frac12}^- \flux_2(\solution_{\cellindy +\frac12}^-)}{b_{\cellindy+\frac12}^+-b_{\cellindy+\frac12}^-} 
+ \frac{b_{\cellindy+\frac12}^+b_{\cellindy+\frac12}^-}{b_{\cellindy+\frac12}^+-b_{\cellindy+\frac12}^-}(\solution_{\cellindy +\frac12}^+ - \solution_{\cellindy +\frac12}^- )$, 
\end{itemize}
with signal speed estimates
$b_{\cellindy+\frac12}^-:=  \{ \lambda_{\min}^2(\bar{\solution}_{\cellindx +\frac12}),\lambda_{\min}^2({\solution}_{\cellindx +\frac12}^-),0 \}$,
$b_{\cellindy+\frac12}^+:=  \{ \lambda_{\max}^2(\bar{\solution}_{\cellindy +\frac12}),\lambda_{\max}^2({\solution}_{\cellindy +\frac12}),0 \}$,
 where $\lambda_{\min}^2(\solution), \lambda_{\max}^2(\solution)$ are the smallest and biggest eigenvalues of $\frac{\partial\flux_2(\solution)}{\partial\solution}$.
 
The semi-discrete system \eqref{eq:DG}-\eqref{eq:fluxintx} can now be solved numerically by a $\spatialorder$-th order SSP Runge--Kutta method, see \cite{Schneider2015b,Schneider2016,Gottlieb2005,Gottlieb2003}. 
In writing down the method we denote by $$\rhs(\solutionprojected(t,\cdot,\cdot)) :\left(\SpaceOfPolynomials{\polydegree}(\cell{\cellind})\times\SpaceOfPolynomials{\polydegreeR}(\cellR{\cellindR})\right)^{\otimes \dimension}
\to
\left(\SpaceOfPolynomials{\polydegree}(\cell{\cellind})\times\SpaceOfPolynomials{\polydegreeR}(\cellR{\cellindR})\right)^{\otimes \dimension} $$
 the right-hand side of (\ref{eq:DG}). Furthermore, $$\tvbminmod :\R^{\dimension\nbxiBasisPolyDim} \to \R^{\dimension\nbxiBasisPolyDim}$$ is the TVBM minmod slope limiter from \cite{CockburnShu2001}.
\begin{remark}\label{rem:subcelllimiter}
  For our numerical example in \secref{subsec:dmr} we use
  the Finite-Volume sub-cell limiter from \cite{SonntagMunz2017}, which proved to be more robust 
  than the TVBM limiter.
\end{remark}

The complete time-marching algorithm for the $\stage$-stage RKDG scheme for a given $\timeind$-th time-iterate $\solutionprojected^{(\timeind)} $ is given in \algoref{algo:TVBMRungeKutta}.
 \begin{algorithm}[H] 
 \caption{TVBM Runge--Kutta Time-Step}
 \label{algo:TVBMRungeKutta}
 \begin{algorithmic}[1]
\State Set $\RKsolution^{(0)}$ = $\solutionprojected^{(\timeind)}$.\hfill \refone{\textit{\# initialization time step $n$}}
\For{$\stageind=1,\ldots, \stage$} {\refone{\hfill \textit{\# loop over Runge-Kutta stages}}}
\medskip
\State Compute 
$\auxiliaryprojected^{\stageind l}=\RKsolution^{(l)} + \frac{\beta_{\stageind l}}{\alpha_{\stageind l}} \Delta t_n \rhs(\RKsolution^{(l)}),$ \hfill \refone{\textit{\# time update}}
\State Compute $ \RKsolution^{(\stageind)} =  \tvbminmod\Big( \sum \limits_{l=0}^{\stageind-1} \alpha_{\stageind l}\auxiliaryprojected^{\stageind l}\Big).$   \hfill\refone{\textit{\# call of TVBM minmod slope limiter}}
\EndFor
\State Set $\solutionprojected^{(\timeind+1)} = \RKsolution^{(\stage)}.$ \hfill \refone{\textit{\# define solution at time step $n+1$}}
\end{algorithmic}
\end{algorithm}
\begin{remark}
The initial condition $\RKsolution^{(0)}$ also has to be limited by $\tvbminmod$.
\end{remark}
The parameters $\alpha_{\stageind l}$ satisfy the conditions $\alpha_{\stageind l}\geq 0$, $\sum \limits_{l=0}^{\stageind-1} \alpha_{\stageind l}=1$ and if $\beta_{\stageind l} \neq 0$, then $\alpha_{\stageind l} \neq 0$ for all $\stageind=1,\ldots, \stage$, $l=0,\ldots,\stageind$.

\subsection{Hyperbolicity-preservation}\label{subsec:Realizability preservation}
For the proof of \thmref{thm:realizable} we need to consider positivity-preserving numerical fluxes as for example 
in \cite{ZhangShu2010}.
For a simple definition of a positivity-preserving numerical flux  let us first assume that $\bar{\reftwo{\solution}}^n_{i}$ is the approximation of the
cell average of an exact solution \reftwo{$\solution(\timevar,\x)$} 
in cell $[\x_{i-\frac{1}{2}},\x_{i+\frac{1}{2}}]$ at time $\timevar_n$. After one time-step of forward Euler we obtain at time $\timevar_{n+1}$
\begin{align} \label{eq:FVscheme}
\bar{\reftwo{\solution}}^{n+1}_{i} = \bar{\reftwo{\solution}}^n_{i} 
- \frac{\Delta \timevar}{\Delta x}\Big(\numericalFlux_1(\bar{\reftwo{\solution}}^n_{i},\bar{\reftwo{\solution}}^n_{i+1})-\numericalFlux_1(\bar{\reftwo{\solution}}^n_{i-1},\bar{\reftwo{\solution}}^n_{i}) \Big).
\end{align}
\begin{definition}\label{def:pospreserving}
	\refone{The scheme \eqref{eq:FVscheme} and the numerical flux $\numericalFlux_1$ are called \textbf{positivity-preserving}, if $\bar{\reftwo{\solution}}^n_{i} \in \realizableSet$ for all $i=1,\ldots,\ncells$ at time $\timevar_\timeind$ implies
	that $\bar{\reftwo{\solution}}^{n+1}_{i} \in \realizableSet$ for all $i=1,\ldots,\ncells$ at time $\timevar_{\timeind+1}$. }
\end{definition}
This is in general achieved under a suitable CFL condition.
Let us assume that the numerical fluxes $\numericalFlux_1, \numericalFlux_2$ are positivity-preserving.
\begin{assumption} \label{ass:ppFlux}
	The numerical fluxes $\numericalFlux_1, \numericalFlux_2$, are positivity-preserving under a suitable CFL condition
	$$\lambda_{\max}^1 \frac{\Delta \timevar }{\Delta \x} + \lambda_{\max}^2\frac{\Delta \timevar }{\Delta \y}  \leq C,$$
	where $C \in (0,1]$.
\end{assumption}
\begin{remark}
  \refone{The positivity-preserving property from \assref{ass:ppFlux} is one ingredient of our hyperbolicity-preserving numerical scheme and the proof of \thmref{thm:realizable}.  We want to mention two numerical fluxes, used in our numerical experiments, which satisfy \assref{ass:ppFlux}. The Lax-Friedrichs flux is positivity-preserving with $C=1$ (see \cite[Remark 2.4]{ZhangShu2010}, where the preservation is shown for the deterministic Euler equations). The HLLE flux also fulfills this property for the deterministic Euler equations (cf. \cite{EinfeldtMunzRoe1991}), whereas we use $C=0.5$ for our numerical results in \secref{sec:results}.}
\end{remark}
\refone{We are now able to formulate the main theorem of this work, which states that, under a modified CFL condition, the pointwise admissible local approximation \eqref{localMEDGDGApproximation} retains an admissible spatial-stochastic cell mean after one time step of forward Euler .}
\begin{theorem}\label{thm:realizable}
	Assume that $\numericalFlux_1, \numericalFlux_2$, are positivity-preserving numerical fluxes 
	and assume that at time $\timepoint{\timeind}$ all point values satisfy 
	$\solution(\timevar_\timeind,\quadxpoint{\xQuadIndex},\quadypointG{\yQuadIndexG},\quadRpoint{\xiQuadIndex})\in\realizableSet$ 
	and $\solution(\timevar_\timeind,\quadxpointG{\xQuadIndexG},\quadypoint{\yQuadIndex},\quadRpoint{\xiQuadIndex})\in\realizableSet$
	for all cells $\cell{\cellind}\times\cellR{\cellindR}$. Then the \reftwo{cell mean} 
	\begin{align}\label{eq:cellmean}
		\cellmeant[\cellind,\cellindR]{\solution} := \frac{1}{\MEintxR{1}} \MEintxR{\solution(\timevar,\x,\y,\localuncertainty{\MEIndex})} &
	\end{align} 
	is admissible after one forward-Euler time-step of \eqref{eq:DG}
	 under the modified CFL condition
	\begin{align}
		\label{eq:CFL}
		 \lambda_{\max}^1 \frac{\Delta \timevar }{\Delta \x} + \lambda_{\max}^2\frac{\Delta \timevar }{\Delta \y}  \leq C\quadxweight{0},
	\end{align}
	where $\quadxweight{0}$ is the first weight of the $(\nbxnodes+1)$-point
	Gau\ss-Lobatto quadrature.
\end{theorem}
\begin{proof}
In the weak formulation \eqref{eq:weakformulation}, we choose the test functions as $\polybasis{\xsumIndexvar} = \frac{1}{\Delta \x \Delta \y}$ and $\localxiBasisPoly{\SGsumIndexvar}{\MEIndex} = \frac{1}{\Delta\uncertainty}$. We get
\begin{align*}
\partial_t \MEintxR{\frac{1}{\Delta\x\Delta\y\Delta\uncertainty}\solution} 
&= \MEintxR{\Big(\flux_1(\solution)\frac{1}{\Delta\uncertainty}\dx\frac{1}{\Delta\x\Delta y} + \flux_2(\solution)\frac{1}{\Delta\uncertainty}\partial_{\y}\frac{1}{\Delta\x \Delta y}\Big)}
 \\ 
 &-
 \inty \MEintcellR{\cellindR}{\frac{1}{\Delta\x\Delta y\Delta\uncertainty}\bigg[\flux_1(\solutionprojected) \bigg]_{\x_{\cellindx-\frac12}}^{\x_{\cellindx+ 
 \frac12}}} \dinty
 \\
 &-
 \intx \MEintcellR{\cellindR}{\frac{1}{\Delta\x\Delta y\Delta\uncertainty}\bigg[\flux_2(\solutionprojected) \bigg]_{\y_{\cellindy-\frac12}}^{\y_{\cellindy+
 \frac12}}} \dintx.
\end{align*}
The definition of the cell mean, the properties $\partial_{\x}\frac{1}{\Delta\x}=0$, $\partial_{\y}\frac{1}{\Delta\y}=0$ and using numerical quadrature yield
\begin{align*}
  \partial_t \cellmeant[\cellind,\cellindR]{\solution} 
= &-  \frac{1}{\Delta\x} \sum_{\yQuadIndexG=0}^{\gausslegendreorder} \sum_{\xiQuadIndex\in\Pidxset} 
\left(  \flux_1\big(\solution(\timevar,\x_{\cellindx+\frac{1}{2}},\y_\yQuadIndexG,\quadRpoint{\xiQuadIndex}) \big) 
- \flux_{\reftwo{1}}\big(\solution(\timevar,\x_{\cellindx-\frac{1}{2}},\y_\yQuadIndexG,\quadRpoint{\xiQuadIndex}) \big)\right) \quadyweightG{\yQuadIndexG} \quadRweight{\xiQuadIndex} 
 \\
 & -   \frac{1}{\Delta\y} \sum_{\xQuadIndexG=0}^{\gausslegendreorder} \sum_{\xiQuadIndex\in\Pidxset} 
\left(  \flux_{\reftwo{2}}\big(\solution(\timevar,\x_\xQuadIndexG,\y_{\cellindy+\frac{1}{2}},\quadRpoint{\xiQuadIndex}) \big) 
- \flux_{\reftwo{2}}\big(\solution(\timevar\x_\xQuadIndexG,y_{\cellindy-\frac{1}{2}},\quadRpoint{\xiQuadIndex}) \big)\right) \quadxweightG{\xQuadIndexG} \quadRweight{\xiQuadIndex}.
\end{align*}
In order to solve the Riemann problem at the cell interfaces we replace the flux components $\flux_i(\solution)$ 
by the corresponding numerical fluxes $\numericalFlux_{i}$, $i=1,2$.
\begin{align}
  \partial_t \cellmeant[\cellind,\cellindR]{\solution} =&
  -\frac{1}{\Delta\x} \sum_{\yQuadIndexG=0}^{\gausslegendreorder}  \sum_{\xiQuadIndex\in\Pidxset} 
   \left(  \numericalFlux_1\big(\solution(\timevar,\x^-_{\cellindx+\frac{1}{2}},\y_\yQuadIndexG,\quadRpoint{\xiQuadIndex}), \solution(\timevar,\x^+_{\cellindx+\frac{1}{2}},\y_\yQuadIndexG,\quadRpoint{\xiQuadIndex}) \big) 
	- \numericalFlux_1\big(\solution(\timevar,\x^-_{\cellindx-\frac{1}{2}},\y_\yQuadIndexG,\quadRpoint{\xiQuadIndex}), \solution(\timevar,\x^+_{\cellindx-\frac{1}{2}},\y_\yQuadIndexG,\quadRpoint{\xiQuadIndex}) \big)\right) 
	 \quadyweightG{\yQuadIndexG}\quadRweight{\xiQuadIndex}   \nonumber
	 \\ 
	 & -\frac{1}{\Delta\y} \sum_{\xQuadIndexG=0}^{\gausslegendreorder}  \sum_{\xiQuadIndex\in\Pidxset} 
   \left(  \numericalFlux_2\big(\solution(\timevar,\x_\xQuadIndexG,\y^-_{\cellindy+\frac{1}{2}},\quadRpoint{\xiQuadIndex}), \solution(\timevar,\x_\xQuadIndexG,\y^+_{\cellindy+\frac{1}{2}},\quadRpoint{\xiQuadIndex}) \big) 
	- \numericalFlux_2\big(\solution(\timevar,\x_\xQuadIndexG,\y^-_{\cellindy-\frac{1}{2}},\quadRpoint{\xiQuadIndex}), \solution(\timevar,\x_\xQuadIndexG,\y^+_{\cellindy-\frac{1}{2}},\quadRpoint{\xiQuadIndex}) \big)\right) 
	 \quadxweightG{\xQuadIndexG}\quadRweight{\xiQuadIndex}. \label{eq:odeCellmean}
\end{align}

We can write the cell mean evaluated at time step $\timevar_\timeind$ as
\begin{align*}
\cellmean[\cellind,\cellindR]{\solution}{\timeind} 
= 
\sum_{\xQuadIndexG=0}^{\gausslegendreorder}  \sum_{\yQuadIndex=0}^{\nbxnodes}\sum_{\xiQuadIndex\in\Pidxset}
\solution(\timevar_\timeind,\quadxpointG{\xQuadIndexG},\quadypoint{\yQuadIndex},\quadRpoint{\xiQuadIndex})\quadxweightG{\xQuadIndexG} \quadyweight{\yQuadIndex}\quadRweight{\xiQuadIndex}
=
\sum_{\xQuadIndex=0}^{\nbxnodes}  \sum_{\yQuadIndexG=0}^{\gausslegendreorder}\sum_{\xiQuadIndex\in\Pidxset}
\solution(\timevar_\timeind,\quadxpoint{\xQuadIndex},\quadypointG{\yQuadIndexG},\quadRpoint{\xiQuadIndex})\quadxweight{\xQuadIndex} \quadyweightG{\yQuadIndexG}\quadRweight{\xiQuadIndex}.
\end{align*}
Let us define $\delta_1:= \lambda_{\max}^1\frac{\Delta t}{\Delta x}, \delta_2:= \lambda_{\max}^2\frac{\Delta t}{\Delta y}$ and $\mu := \delta_1 + \delta_2$. This allows us to write
\begin{align} \label{eq:representationCellmean}
\cellmean[\cellind,\cellindR]{\solution}{\timeind} 
=& \frac{\delta_1}{\mu}\cellmean[\cellind,\cellindR]{\solution}{\timeind}  + \frac{\delta_2}{\mu}\cellmean[\cellind,\cellindR]{\solution}{\timeind} \nonumber \\
=& \frac{\delta_1}{\mu} \sum_{\xQuadIndex=0}^{\nbxnodes}  \sum_{\yQuadIndexG=0}^{\gausslegendreorder}\sum_{\xiQuadIndex\in\Pidxset}
\solution(\timevar_\timeind,\quadxpoint{\xQuadIndex},\quadypointG{\yQuadIndexG},\quadRpoint{\xiQuadIndex})\quadxweight{\xQuadIndex} \quadyweightG{\yQuadIndexG}\quadRweight{\xiQuadIndex}+ \frac{\delta_2}{\mu}
\sum_{\xQuadIndexG=0}^{\gausslegendreorder}  \sum_{\yQuadIndex=0}^{\nbxnodes}\sum_{\xiQuadIndex\in\Pidxset}
\solution(\timevar_\timeind,\quadxpointG{\xQuadIndexG},\quadypoint{\yQuadIndex},\quadRpoint{\xiQuadIndex})\quadxweightG{\xQuadIndexG} \quadyweight{\yQuadIndex}\quadRweight{\xiQuadIndex} \nonumber
\\
 = & \frac{\delta_1}{\mu}\sum_{\xQuadIndex=1}^{\nbxnodes-1}  \sum_{\yQuadIndexG=0}^{\gausslegendreorder}\sum_{\xiQuadIndex\in\Pidxset}
\solution(\timevar_\timeind,\quadxpoint{\xQuadIndex},\quadypointG{\yQuadIndexG},\quadRpoint{\xiQuadIndex})\quadxweight{\xQuadIndex} \quadyweightG{\yQuadIndexG}\quadRweight{\xiQuadIndex} \nonumber
\\
& + \frac{\delta_1}{\mu} \quadyweight{0}  \sum_{\yQuadIndexG=0}^{\gausslegendreorder}\sum_{\xiQuadIndex\in\Pidxset} \Big(
\solution(\timevar_\timeind,\x^+_{\cellindx-\frac{1}{2}},\quadypointG{\yQuadIndexG},\quadRpoint{\xiQuadIndex}) + \solution(\timevar_\timeind,\x^-_{\cellindx+\frac{1}{2}},\quadypointG{\yQuadIndexG},\quadRpoint{\xiQuadIndex}) \Big)\quadyweightG{\yQuadIndexG}\quadRweight{\xiQuadIndex}, \nonumber
\\
&+ \frac{\delta_2}{\mu} \sum_{\xQuadIndexG=0}^{\gausslegendreorder}  \sum_{\yQuadIndex=1}^{\nbxnodes-1}\sum_{\xiQuadIndex\in\Pidxset}
\solution(\timevar_\timeind,\quadxpointG{\xQuadIndexG},\quadypoint{\yQuadIndex},\quadRpoint{\xiQuadIndex})\quadxweightG{\xQuadIndexG} \quadyweight{\yQuadIndex}\quadRweight{\xiQuadIndex} \nonumber 
\\
& + \frac{\delta_2}{\mu} \quadyweight{0} \sum_{\xQuadIndexG=0}^{\gausslegendreorder} \sum_{\xiQuadIndex\in\Pidxset} \Big(
\solution(\timevar_\timeind,\quadxpointG{\xQuadIndexG},\y^+_{\cellindy-\frac{1}{2}},\quadRpoint{\xiQuadIndex}) + \solution(\timevar_\timeind,\quadxpointG{\xQuadIndexG},\y^-_{\cellindy+\frac{1}{2}},\quadRpoint{\xiQuadIndex}) \Big) \quadxweightG{\xQuadIndexG} \quadRweight{\xiQuadIndex}, 
\end{align}
where we used the fact that $\quadyweight{0} = \quadyweight{\nbxnodes}$, $\y^+_{\cellindy-\frac{1}{2}} = \quadypoint{0}\big|_{C_{\cellind}}$,
$\y^-_{\cellindy+\frac{1}{2}} = \quadypoint{\nbxnodes}\big|_{C_{\cellind}}$, $\x^+_{\cellindx-\frac{1}{2}} = \quadxpoint{0}\big|_{C_{\cellind}}$,
$\x^-_{\cellindx+\frac{1}{2}} = \quadxpoint{\nbxnodes}\big|_{C_{\cellind}}$.
One time-step of forward Euler of \eqref{eq:odeCellmean} reads as follows
\begin{align}
  \cellmean[\cellind,\cellindR]{\solution}{\timeind+1} 
  &= \cellmean[\cellind,\cellindR]{\solution}{\timeind} \nonumber\\
  & -\frac{\Delta \timevar}{\Delta\x} \sum_{\yQuadIndexG=0}^{\gausslegendreorder}  \sum_{\xiQuadIndex\in\Pidxset} 
   \left(  \numericalFlux_1\big(\solution(\timevar_n,\x^-_{\cellindx+\frac{1}{2}},\y_\yQuadIndexG,\quadRpoint{\xiQuadIndex}), \solution(\timevar_n,\x^+_{\cellindx+\frac{1}{2}},\y_
   \yQuadIndexG,\quadRpoint{\xiQuadIndex}) \big) 
   - \numericalFlux_1\big(\solution(\timevar_n,\x^-_{\cellindx-\frac{1}{2}},\y_\yQuadIndexG,\quadRpoint{\xiQuadIndex}), \solution(\timevar_n,\x^+_{\cellindx-\frac{1}{2}},\y_
   \yQuadIndexG,\quadRpoint{\xiQuadIndex}) \big)\right) 
    \quadyweightG{\yQuadIndexG}\quadRweight{\xiQuadIndex} \nonumber 
    \\
    & -\frac{\Delta \timevar}{\Delta\y} \sum_{\xQuadIndexG=0}^{\gausslegendreorder}  \sum_{\xiQuadIndex\in\Pidxset} 
    \left(  \numericalFlux_2\big(\solution(\timevar_n,\x_\xQuadIndexG,\y^-_{\cellindy+\frac{1}{2}},\quadRpoint{\xiQuadIndex}), \solution(\timevar_n,\x_\xQuadIndexG,\y^+_{\cellindy+\frac{1}{2}},\quadRpoint{\xiQuadIndex}) \big) 
	- \numericalFlux_2\big(\solution(\timevar_n,\x_\xQuadIndexG,\y^-_{\cellindy-\frac{1}{2}},\quadRpoint{\xiQuadIndex}), \solution(\timevar_n,\x_\xQuadIndexG,\y^+_{\cellindy-\frac{1}{2}},\quadRpoint{\xiQuadIndex}) \big)\right) \quadxweightG{\xQuadIndexG}\quadRweight{\xiQuadIndex}.
\end{align}
Using \eqref{eq:representationCellmean} gives
\begin{align*}
  \cellmean[\cellind,\cellindR]{\solution}{\timeind+1} 
  &= \frac{\delta_1}{\mu}\sum_{\xQuadIndex=1}^{\nbxnodes-1}  \sum_{\yQuadIndexG=0}^{\gausslegendreorder}\sum_{\xiQuadIndex\in\Pidxset}
\solution(\timevar_\timeind,\quadxpoint{\xQuadIndex},\quadypointG{\yQuadIndexG},\quadRpoint{\xiQuadIndex})\quadxweight{\xQuadIndex} \quadyweightG{\yQuadIndexG}\quadRweight{\xiQuadIndex} \nonumber
\\
& + \frac{\delta_1}{\mu} \quadyweight{0}  \sum_{\yQuadIndexG=0}^{\gausslegendreorder}\sum_{\xiQuadIndex\in\Pidxset} \Big(
\solution(\timevar_\timeind,\x^+_{\cellindx-\frac{1}{2}},\quadypointG{\yQuadIndexG},\quadRpoint{\xiQuadIndex}) + \solution(\timevar_\timeind,\x^-_{\cellindx+\frac{1}{2}},\quadypointG{\yQuadIndexG},\quadRpoint{\xiQuadIndex}) \Big)\quadyweightG{\yQuadIndexG}\quadRweight{\xiQuadIndex}, \nonumber
\\
&+ \frac{\delta_2}{\mu} \sum_{\xQuadIndexG=0}^{\gausslegendreorder}  \sum_{\yQuadIndex=1}^{\nbxnodes-1}\sum_{\xiQuadIndex\in\Pidxset}
\solution(\timevar_\timeind,\quadxpointG{\xQuadIndexG},\quadypoint{\yQuadIndex},\quadRpoint{\xiQuadIndex})\quadxweightG{\xQuadIndexG} \quadyweight{\yQuadIndex}\quadRweight{\xiQuadIndex} \nonumber 
\\
& + \frac{\delta_2}{\mu} \quadyweight{0} \sum_{\xQuadIndexG=0}^{\gausslegendreorder} \sum_{\xiQuadIndex\in\Pidxset} \Big(
\solution(\timevar_\timeind,\quadxpointG{\xQuadIndexG},\y^+_{\cellindy-\frac{1}{2}},\quadRpoint{\xiQuadIndex}) + \solution(\timevar_\timeind,\quadxpointG{\xQuadIndexG},\y^-_{\cellindy+\frac{1}{2}},\quadRpoint{\xiQuadIndex}) \Big) \quadxweightG{\xQuadIndexG} \quadRweight{\xiQuadIndex} 
\\
  & -\frac{\delta_1}{\lambda_{\max}^1}  \sum_{\yQuadIndexG=0}^{\gausslegendreorder}  \sum_{\xiQuadIndex\in\Pidxset} 
   \left(  \numericalFlux_1\big(\solution(\timevar_n,\x^-_{\cellindx+\frac{1}{2}},\y_\yQuadIndexG,\quadRpoint{\xiQuadIndex}), \solution(\timevar_n,\x^+_{\cellindx+\frac{1}{2}},\y_
   \yQuadIndexG,\quadRpoint{\xiQuadIndex}) \big) 
   - \numericalFlux_1\big(\solution(\timevar_n,\x^-_{\cellindx-\frac{1}{2}},\y_\yQuadIndexG,\quadRpoint{\xiQuadIndex}), \solution(\timevar_n,\x^+_{\cellindx-\frac{1}{2}},\y_
   \yQuadIndexG,\quadRpoint{\xiQuadIndex}) \big)\right) 
    \quadyweightG{\yQuadIndexG}\quadRweight{\xiQuadIndex}  
    \\
    & -\frac{\delta_2}{\lambda_{\max}^2} \sum_{\xQuadIndexG=0}^{\gausslegendreorder}  \sum_{\xiQuadIndex\in\Pidxset} 
    \left(  \numericalFlux_2\big(\solution(\timevar_n,\x_\xQuadIndexG,\y^-_{\cellindy+\frac{1}{2}},\quadRpoint{\xiQuadIndex}), \solution(\timevar_n,\x_\xQuadIndexG,\y^+_{\cellindy+\frac{1}{2}},\quadRpoint{\xiQuadIndex}) \big) 
	- \numericalFlux_2\big(\solution(\timevar_n,\x_\xQuadIndexG,\y^-_{\cellindy-\frac{1}{2}},\quadRpoint{\xiQuadIndex}), \solution(\timevar_n,\x_\xQuadIndexG,\y^+_{\cellindy-\frac{1}{2}},\quadRpoint{\xiQuadIndex}) \big)\right) \quadxweightG{\xQuadIndexG}\quadRweight{\xiQuadIndex}.
\end{align*}
Adding and subtracting $\numericalFlux_1\big(\solution(\timevar_n,\x^+_{\cellindx-\frac{1}{2}},\y_\yQuadIndexG,\quadRpoint{\xiQuadIndex}), \solution(\timevar_n,\x^-_{\cellindx+\frac{1}{2}},\y_\yQuadIndexG,\quadRpoint{\xiQuadIndex}) \big)$ and $\numericalFlux_2\big(\solution(\timevar_n,\x_\xQuadIndexG,\y^+_{\cellindy-\frac{1}{2}},\quadRpoint{\xiQuadIndex}), \solution(\timevar_n,\x_\xQuadIndexG,\y^-_{\cellindy+\frac{1}{2}},\quadRpoint{\xiQuadIndex}) \big)$ yields
\begin{align*}
\cellmean[\cellind,\cellindR]{\solution}{\timeind+1} 
&= \frac{\delta_1}{\mu}\sum_{\xQuadIndex=1}^{\nbxnodes-1}  \sum_{\yQuadIndexG=0}^{\gausslegendreorder}\sum_{\xiQuadIndex\in\Pidxset}
\solution(\timevar_\timeind,\quadxpoint{\xQuadIndex},\quadypointG{\yQuadIndexG},\quadRpoint{\xiQuadIndex})\quadxweight{\xQuadIndex} \quadyweightG{\yQuadIndexG}\quadRweight{\xiQuadIndex} 
+ 
\frac{\delta_2}{\mu} \sum_{\xQuadIndexG=0}^{\gausslegendreorder}  \sum_{\yQuadIndex=1}^{\nbxnodes-1}\sum_{\xiQuadIndex\in\Pidxset}
\solution(\timevar_\timeind,\quadxpointG{\xQuadIndexG},\quadypoint{\yQuadIndex},\quadRpoint{\xiQuadIndex})\quadxweightG{\xQuadIndexG} \quadyweight{\yQuadIndex}\quadRweight{\xiQuadIndex}  
\\
& + \frac{\delta_1}{\mu} \quadyweight{0}  \sum_{\yQuadIndexG=0}^{\gausslegendreorder}\sum_{\xiQuadIndex\in\Pidxset} \Big(
\solution(\timevar_\timeind,\x^+_{\cellindx-\frac{1}{2}},\quadypointG{\yQuadIndexG},\quadRpoint{\xiQuadIndex}) + \solution(\timevar_\timeind,\x^-_{\cellindx+\frac{1}{2}},\quadypointG{\yQuadIndexG},\quadRpoint{\xiQuadIndex}) \Big)\quadyweightG{\yQuadIndexG}\quadRweight{\xiQuadIndex}, \nonumber
\\
& + \frac{\delta_2}{\mu} \quadyweight{0} \sum_{\xQuadIndexG=0}^{\gausslegendreorder} \sum_{\xiQuadIndex\in\Pidxset} \Big(
\solution(\timevar_\timeind,\quadxpointG{\xQuadIndexG},\y^+_{\cellindy-\frac{1}{2}},\quadRpoint{\xiQuadIndex}) + \solution(\timevar_\timeind,\quadxpointG{\xQuadIndexG},\y^-_{\cellindy+\frac{1}{2}},\quadRpoint{\xiQuadIndex}) \Big) \quadxweightG{\xQuadIndexG} \quadRweight{\xiQuadIndex} 
\\
& -\frac{\delta_1}{\lambda_{\max}^1}  \sum_{\yQuadIndexG=0}^{\gausslegendreorder}  \sum_{\xiQuadIndex\in\Pidxset} 
\left(  \numericalFlux_1\big(\solution(\timevar_n,\x^-_{\cellindx+\frac{1}{2}},\y_\yQuadIndexG,\quadRpoint{\xiQuadIndex}), \solution(\timevar_n,\x^+_{\cellindx+\frac{1}{2}},\y_
\yQuadIndexG,\quadRpoint{\xiQuadIndex}) \big) - \numericalFlux_1\big(\solution(\timevar_n,\x^+_{\cellindx-\frac{1}{2}},\y_\yQuadIndexG,\quadRpoint{\xiQuadIndex}), \solution(\timevar_n,\x^-_{\cellindx+\frac{1}{2}},\y_\yQuadIndexG,\quadRpoint{\xiQuadIndex}) \big) \right)  \quadyweightG{\yQuadIndexG}\quadRweight{\xiQuadIndex}  
\\
& -\frac{\delta_1}{\lambda_{\max}^1}  \sum_{\yQuadIndexG=0}^{\gausslegendreorder}  \sum_{\xiQuadIndex\in\Pidxset} 
\left(\numericalFlux_1\big(\solution(\timevar_n,\x^+_{\cellindx-\frac{1}{2}},\y_\yQuadIndexG,\quadRpoint{\xiQuadIndex}), \solution(\timevar_n,\x^-_{\cellindx+\frac{1}{2}},\y_\yQuadIndexG,\quadRpoint{\xiQuadIndex}) \big) - \numericalFlux_1\big(\solution(\timevar_n,\x^-_{\cellindx-\frac{1}{2}},\y_\yQuadIndexG,\quadRpoint{\xiQuadIndex}), \solution(\timevar_n,\x^+_{\cellindx-\frac{1}{2}},\y_
\yQuadIndexG,\quadRpoint{\xiQuadIndex}) \big)\right) 
\quadyweightG{\yQuadIndexG}\quadRweight{\xiQuadIndex}  
\\
& -\frac{\delta_2}{\lambda_{\max}^2} \sum_{\xQuadIndexG=0}^{\gausslegendreorder}  \sum_{\xiQuadIndex\in\Pidxset} 
\left(  \numericalFlux_2\big(\solution(\timevar_n,\x_\xQuadIndexG,\y^-_{\cellindy+\frac{1}{2}},\quadRpoint{\xiQuadIndex}), \solution(\timevar_n,\x_\xQuadIndexG,\y^+_{\cellindy+\frac{1}{2}},\quadRpoint{\xiQuadIndex}) \big) 
- \numericalFlux_2\big(\solution(\timevar_n,\x_\xQuadIndexG,\y^+_{\cellindy-\frac{1}{2}},\quadRpoint{\xiQuadIndex}), \solution(\timevar_n,\x_\xQuadIndexG,\y^-_{\cellindy+\frac{1}{2}},\quadRpoint{\xiQuadIndex}) \big)\right) \quadxweightG{\xQuadIndexG}\quadRweight{\xiQuadIndex}
\\
& -\frac{\delta_2}{\lambda_{\max}^2} \sum_{\xQuadIndexG=0}^{\gausslegendreorder}  \sum_{\xiQuadIndex\in\Pidxset} 
\left( \numericalFlux_2\big(\solution(\timevar_n,\x_\xQuadIndexG,\y^+_{\cellindy-\frac{1}{2}},\quadRpoint{\xiQuadIndex}), \solution(\timevar_n,\x_\xQuadIndexG,\y^-_{\cellindy+\frac{1}{2}},\quadRpoint{\xiQuadIndex}) \big)
- \numericalFlux_2\big(\solution(\timevar_n,\x_\xQuadIndexG,\y^-_{\cellindy-\frac{1}{2}},\quadRpoint{\xiQuadIndex}), \solution(\timevar_n,\x_\xQuadIndexG,\y^+_{\cellindy-\frac{1}{2}},\quadRpoint{\xiQuadIndex}) \big)\right) \quadxweightG{\xQuadIndexG}\quadRweight{\xiQuadIndex}.
\end{align*}
If we rearrange the previous equation we obtain 
\begin{align} \label{eq:convexCombination}
\cellmean[\cellind,\cellindR]{\solution}{\timeind+1} 
&=  \frac{\delta_1}{\mu}\sum_{\xQuadIndex=1}^{\nbxnodes-1}  \sum_{\yQuadIndexG=0}^{\gausslegendreorder}\sum_{\xiQuadIndex\in\Pidxset}
\solution(\timevar_\timeind,\quadxpoint{\xQuadIndex},\quadypointG{\yQuadIndexG},\quadRpoint{\xiQuadIndex})\quadxweight{\xQuadIndex} \quadyweightG{\yQuadIndexG}\quadRweight{\xiQuadIndex} 
+ 
\frac{\delta_2}{\mu} \sum_{\xQuadIndexG=0}^{\gausslegendreorder}  \sum_{\yQuadIndex=1}^{\nbxnodes-1}\sum_{\xiQuadIndex\in\Pidxset}
\solution(\timevar_\timeind,\quadxpointG{\xQuadIndexG},\quadypoint{\yQuadIndex},\quadRpoint{\xiQuadIndex})\quadxweightG{\xQuadIndexG} \quadyweight{\yQuadIndex}\quadRweight{\xiQuadIndex}  
\\ \nonumber
& + \frac{\delta_1}{\mu} \quadyweight{0}  \sum_{\yQuadIndexG=0}^{\gausslegendreorder}\sum_{\xiQuadIndex\in\Pidxset} \quadyweightG{\yQuadIndexG}\quadRweight{\xiQuadIndex} \Big(
\solution(\timevar_\timeind,\x^+_{\cellindx-\frac{1}{2}},\quadypointG{\yQuadIndexG},\quadRpoint{\xiQuadIndex}) 
\\ \nonumber 
&-  \frac{\mu}{\lambda_{\max}^1 \quadyweight{0}} \big(\numericalFlux_1\big(\solution(\timevar_n,\x^+_{\cellindx-\frac{1}{2}},\y_\yQuadIndexG,\quadRpoint{\xiQuadIndex}), \solution(\timevar_n,\x^-_{\cellindx+\frac{1}{2}},\y_\yQuadIndexG,\quadRpoint{\xiQuadIndex}) \big) - \numericalFlux_1\big(\solution(\timevar_n,\x^-_{\cellindx-\frac{1}{2}},\y_\yQuadIndexG,\quadRpoint{\xiQuadIndex}), \solution(\timevar_n,\x^+_{\cellindx-\frac{1}{2}},\y_\yQuadIndexG,\quadRpoint{\xiQuadIndex}) \big)\big) \Big)
\\ \nonumber
& + \frac{\delta_1}{\mu} \quadyweight{0}  \sum_{\yQuadIndexG=0}^{\gausslegendreorder}\sum_{\xiQuadIndex\in\Pidxset} \quadyweightG{\yQuadIndexG}\quadRweight{\xiQuadIndex} \Big(
\solution(\timevar_\timeind,\x^-_{\cellindx+\frac{1}{2}},\quadypointG{\yQuadIndexG},\quadRpoint{\xiQuadIndex}) + \solution(\timevar_\timeind,\x^-_{\cellindx+\frac{1}{2}},\quadypointG{\yQuadIndexG},\quadRpoint{\xiQuadIndex})  
\\ \nonumber
& - \frac{\mu}{\lambda_{\max}^1\quadyweight{0}}
\big(  \numericalFlux_1\big(\solution(\timevar_n,\x^-_{\cellindx+\frac{1}{2}},\y_\yQuadIndexG,\quadRpoint{\xiQuadIndex}), \solution(\timevar_n,\x^+_{\cellindx+\frac{1}{2}},\y_
\yQuadIndexG,\quadRpoint{\xiQuadIndex}) \big) - \numericalFlux_1\big(\solution(\timevar_n,\x^+_{\cellindx-\frac{1}{2}},\y_\yQuadIndexG,\quadRpoint{\xiQuadIndex}), \solution(\timevar_n,\x^-_{\cellindx+\frac{1}{2}},\y_\yQuadIndexG,\quadRpoint{\xiQuadIndex}) \big) \big) \Big) 
\\ \nonumber
& + \frac{\delta_2}{\mu} \quadyweight{0} \sum_{\xQuadIndexG=0}^{\gausslegendreorder} \sum_{\xiQuadIndex\in\Pidxset} \quadxweightG{\xQuadIndexG} \quadRweight{\xiQuadIndex}  \Big(
\solution(\timevar_\timeind,\quadxpointG{\xQuadIndexG},\y^+_{\cellindy-\frac{1}{2}},\quadRpoint{\xiQuadIndex}) 
\\ \nonumber
& - \frac{\mu}{\lambda_{\max}^2 \quadyweight{0}} 
\big( \numericalFlux_2\big(\solution(\timevar_n,\x_\xQuadIndexG,\y^+_{\cellindy-\frac{1}{2}},\quadRpoint{\xiQuadIndex}), \solution(\timevar_n,\x_\xQuadIndexG,\y^-_{\cellindy+\frac{1}{2}},\quadRpoint{\xiQuadIndex}) \big)- \numericalFlux_2\big(\solution(\timevar_n,\x_\xQuadIndexG,\y^-_{\cellindy-\frac{1}{2}},\quadRpoint{\xiQuadIndex}), \solution(\timevar_n,\x_\xQuadIndexG,\y^+_{\cellindy-\frac{1}{2}},\quadRpoint{\xiQuadIndex}) \big) \big)\Big)
\\ \nonumber
& + \frac{\delta_2}{\mu} \quadyweight{0} \sum_{\xQuadIndexG=0}^{\gausslegendreorder} \sum_{\xiQuadIndex\in\Pidxset}  \quadxweightG{\xQuadIndexG} \quadRweight{\xiQuadIndex}
\Big(\solution(\timevar_\timeind,\quadxpointG{\xQuadIndexG},\y^-_{\cellindy+\frac{1}{2}},\quadRpoint{\xiQuadIndex}) 
\\ \nonumber
& -\frac{\mu}{\lambda_{\max}^2 \quadyweight{0}} \big(\numericalFlux_2\big(\solution(\timevar_n,\x_\xQuadIndexG,\y^-_{\cellindy+\frac{1}{2}},\quadRpoint{\xiQuadIndex}), \solution(\timevar_n,\x_\xQuadIndexG,\y^+_{\cellindy+\frac{1}{2}},\quadRpoint{\xiQuadIndex})) - \numericalFlux_2\big(\solution(\timevar_n,\x_\xQuadIndexG,\y^+_{\cellindy-\frac{1}{2}},\quadRpoint{\xiQuadIndex}), \solution(\timevar_n,\x_\xQuadIndexG,\y^-_{\cellindy+\frac{1}{2}},\quadRpoint{\xiQuadIndex}) \big)\big) \Big).
\end{align}
Now every term of the form 
\begin{align*}
\solution(\timevar_\timeind,\x^+_{\cellindx-\frac{1}{2}},\quadypointG{\yQuadIndexG},\quadRpoint{\xiQuadIndex}) -  \frac{\mu}{\lambda_{\max}^1 \quadyweight{0}} \big(\numericalFlux_1\big(\solution(\timevar_n,\x^+_{\cellindx-\frac{1}{2}},\y_\yQuadIndexG,\quadRpoint{\xiQuadIndex}), \solution(\timevar_n,\x^-_{\cellindx+\frac{1}{2}},\y_\yQuadIndexG,\quadRpoint{\xiQuadIndex}) \big) - \numericalFlux_1\big(\solution(\timevar_n,\x^-_{\cellindx-\frac{1}{2}},\y_\yQuadIndexG,\quadRpoint{\xiQuadIndex}), \solution(\timevar_n,\x^+_{\cellindx-\frac{1}{2}},\y_\yQuadIndexG,\quadRpoint{\xiQuadIndex}) \big)\big)
\end{align*}
and
\begin{align*}
\solution(\timevar_\timeind,\quadxpointG{\xQuadIndexG},\y^+_{\cellindy-\frac{1}{2}},\quadRpoint{\xiQuadIndex}) 
- \frac{\mu}{\lambda_{\max}^2 \quadyweight{0}} 
\big( \numericalFlux_2\big(\solution(\timevar_n,\x_\xQuadIndexG,\y^+_{\cellindy-\frac{1}{2}},\quadRpoint{\xiQuadIndex}), \solution(\timevar_n,\x_\xQuadIndexG,\y^-_{\cellindy+\frac{1}{2}},\quadRpoint{\xiQuadIndex}) \big)- \numericalFlux_2\big(\solution(\timevar_n,\x_\xQuadIndexG,\y^-_{\cellindy-\frac{1}{2}},\quadRpoint{\xiQuadIndex}), \solution(\timevar_n,\x_\xQuadIndexG,\y^+_{\cellindy-\frac{1}{2}},\quadRpoint{\xiQuadIndex}) \big) \big)
\end{align*}
is admissible under the modified CFL condition
\begin{align*}
\lambda_{\max}^1 \frac{\Delta \timevar }{\Delta \x} + \lambda_{\max}^2\frac{\Delta \timevar }{\Delta \y}  \leq C\quadxweight{0}
\end{align*}
\refone{because of the positivity-preserving property from \defref{def:pospreserving} under \assref{ass:ppFlux}.}
Hence, $\cellmean[\cellind,\cellindR]{\solution}{\timeind+1}$  is a convex combination of
admissible quantities  in $\realizableSet$ and therefore $\cellmean[\cellind,\cellindR]{\solution}{\timeind+1} \in \realizableSet$.
\end{proof}

\begin{remark}
In the one-dimensional case $\Domain\subset\mathbb{R}$, we only have to verify $\solution(\timevar_\timeind,\quadxpoint{\xQuadIndex},\quadRpoint{\xiQuadIndex})\in\realizableSet$ for all cells $\cell{i}\times\cellR{\cellindR}$ in order to apply \thmref{thm:realizable}.
\end{remark}

\subsection{Hyperbolicity limiter}\label{subsec: Realizability limiter}
To ensure that at time $\timepoint{\timeind}$ all point values satisfy $\solution\big(\timevar_\timeind,\quadxpoint{\xQuadIndex},\quadypointG{\yQuadIndexG},\quadRpoint{\xiQuadIndex}\big)\in\realizableSet$ and $\solution(\timevar_\timeind,\quadxpointG{\xQuadIndexG},\quadypoint{\yQuadIndex},\quadRpoint{\xiQuadIndex})\in\realizableSet$ for each cell $\cell{\cellind}\times\cellR{\cellindR}$, we define the slope-limited polynomial in $\cell{\cellind}\times\cellR{\cellindR}$ as
\begin{align}\label{eq:slopeLimitedSolution}
\hyperbolLimit{\limitervariable}\left(\localsolution{\cellind}{\cellindR}\right)(\timevar,\x,\y,\localuncertainty{\MEIndex}):= \limitervariable \,\cellmeant[\cellind,\cellindR]{\solution}+ (1-\limitervariable) \localsolution{\cellind}{\cellindR}(\timevar,\x,\y,\localuncertainty{\MEIndex}).
\end{align}
\refone{The variable $\limitervariable$ limits the polynomial towards the cell mean, which is \hyperbolic{} according to \thmref{thm:realizable}. }

The case $\limitervariable=0$ coincides with the unlimited solution and for $\limitervariable=1$ we have
$$\hyperbolLimit{\limitervariable=1} (\localsolution{\cellind}{\cellindR})(\timevar,\x,\y,\localuncertainty{\MEIndex}) =\cellmean[\cellind,\cellindR]{\solution}{\timeind},$$
which is supposed to be \hyperbolic{}. Because of this property and since $\realizableSet$ is convex, we can choose
\begin{align*}
\hat{\limitervariable}_{\cellind,\cellindR}(\timevar_\timeind) := \inf\Big\{ \tilde{\limitervariable}\in [0,1] ~\Big|~ &\hyperbolLimit{\tilde{\limitervariable}}(\localsolution{\cellind}{\cellindR})(\timevar_\timeind,\quadxpoint{\xQuadIndex},\quadypointG{\yQuadIndexG},\quadRpoint{\xiQuadIndex})  \in \realizableSet \wedge \hyperbolLimit{\tilde{\limitervariable}}(\localsolution{\cellind}{\cellindR})(\timevar_\timeind,\quadxpointG{\xQuadIndexG},\quadypoint{\yQuadIndex},\quadRpoint{\xiQuadIndex})\in\realizableSet~~\\
& \forall\, \xQuadIndex,\yQuadIndex=0,\ldots,\nbxnodes~,\reftwo{\xQuadIndexG,\yQuadIndexG=0,\ldots,\gausslegendreorder,}~ \xiQuadIndex=0,\ldots,\nbRnodes\Big\}.
\end{align*}
Due to the openness of $\realizableSet$, we need to modify $\limitervariable$ slightly in order to avoid placing the solution onto the boundary (if the limiter was active). Therefore we use 
\begin{align*}
\limitervariable = \begin{cases}
\hat{\limitervariable}, & \text{ if } \hat{\limitervariable} = 0,\\
\min(\hat{\limitervariable}+\limitertolerance,1), & \text{ if } \hat{\limitervariable} > 0,
\end{cases}
\end{align*}
where $0<\limitertolerance = 10^{-10}$ should be chosen small enough to ensure that the approximation quality is not influenced significantly.

Using the limited point values, we derive the updated coefficients via
$$\DGmoment{\xsumIndex}{\SGsumIndex}(\timevar_\timeind) =  \sum_{\xQuadIndex=0}^{\nbxnodes} \sum_{\yQuadIndex=0}^{\nbxnodes} \sum_{\xiQuadIndex\in\Pidxset} \hyperbolLimit{\limitervariable} (\localsolution{\cellind}{\cellindR} )(\timevar_\timeind,\quadxpoint{\xQuadIndex},\quadypoint{\yQuadIndex},\quadRpoint{\xiQuadIndex}) \polybasis{\xsumIndex}(\quadxpoint{\xQuadIndex},\quadypoint{\yQuadIndex})\polybasisR{\SGsumIndex}(\quadRpoint{\xiQuadIndex}) \quadxweight{\xQuadIndex}\quadyweight{\yQuadIndex}\quadRweight{\xiQuadIndex},\quad \forall \xsumIndex=0,\ldots,\polydegree,~\SGsumIndex=0,\ldots,\polydegreeR.$$
Note that the cell mean is preserved since
\begin{align*}\sum_{\xQuadIndex=0}^{\nbxnodes} \sum_{\yQuadIndex=0}^{\nbxnodes}\sum_{\xiQuadIndex\in\Pidxset} &\hyperbolLimit{\limitervariable} (\localsolution{\cellind}{\cellindR} )(\timevar_\timeind,\quadxpoint{\xQuadIndex},\quadypoint{\yQuadIndex},\quadRpoint{\xiQuadIndex}) \quadxweight{\xQuadIndex}\quadyweight{\yQuadIndex}\quadRweight{\xiQuadIndex}\\
&= \limitervariable  \,\cellmean[\cellind,\cellindR]{\solution}{\timeind} \sum_{\xQuadIndex=0}^{\nbxnodes} \sum_{\yQuadIndex=0}^{\nbxnodes}\sum_{\xiQuadIndex\in\Pidxset}\quadxweight{\xQuadIndex}\quadyweight{\yQuadIndex}\quadRweight{\xiQuadIndex}
+ (1-\limitervariable) \sum_{\xQuadIndex=0}^{\nbxnodes} \sum_{\yQuadIndex=0}^{\nbxnodes}\sum_{\xiQuadIndex\in\Pidxset}\solution\big(\timevar_\timeind,\quadxpoint{\xQuadIndex},\quadypoint{\yQuadIndex},\quadRpoint{\xiQuadIndex}\big)\quadxweight{\xQuadIndex}\quadyweight{\yQuadIndex}\quadRweight{\xiQuadIndex} \\
&=  \cellmean[\cellind,\cellindR]{\solution}{\timeind}.
\end{align*}

\refone{\begin{remark}
	After the application of the hyperbolic limiter, the SG system is guaranteed to be hyperbolic according to \thmref{thm:SGsystemhyp}.
\end{remark}}

\begin{remark}
\refone{From now on, we again} abuse notation and write $\hyperbolLimit{\limitervariable}(\solution)$ for the piecewise polynomial $\solution$ instead of the local polynomials $\localsolution{\cellind}{\cellindR}$. By this we mean the application of the hyperbolicity limiter on each cell $\cell{\cellind}\times\cellR{\cellindR}$ separately where we obtain independent values of the limiter variable $\limitervariable$.
\end{remark}

A complete Runge--Kutta time-step using the hyperbolicity-preserving limiter is shown in Algorithm \ref{alg:hdSG}.
\begin{algorithm}
\caption{TVBM Runge--Kutta time-step with hyperbolic limiter}
\label{alg:hdSG}
\begin{algorithmic}[1]
\State Set $\RKsolution^{(0)} = \solutionprojected^{(\timeind)}$. \hfill \refone{\textit{\# initialization time step $n$}}
\For{$\stageind=1,\ldots, \stage$} {\refone{\hfill \textit{\# loop over Runge-Kutta stages}}}
\medskip
\State Compute $\auxiliaryprojected^{\stageind l}=\RKsolution^{(l)} + \frac{\beta_{\stageind l}}{\alpha_{\stageind l}} \dtstepsize_n\rhs(\RKsolution^{(l)})$,\hfill\refone{\textit{\# time update}}
\State Compute
$ \RKsolution^{(\stageind)} =  \tvbminmod\Big( \sum \limits_{l=0}^{\stageind-1} \alpha_{\stageind l}\auxiliaryprojected^{\stageind l}\Big)$,
\hfill\refone{\textit{\# call of TVBM minmod slope limiter}}
\State Compute $\RKsolution^{(\stageind)}= \hyperbolLimit{\limitervariable}(\RKsolution^{(\stageind)})$. 
\hfill\refone{\textit{\# call of hyperbolicity limiter}}
\medskip
\EndFor
\State Set $\solutionprojected^{(\timeind+1)}= \RKsolution^{(\stage)}$. \hfill \refone{\textit{\# define solution at time step $n+1$}}
\end{algorithmic}
\end{algorithm}
\begin{remark}
The initial condition $\solutionprojected^0$ also has to be limited by both limiters $\tvbminmod$ and $\hyperbolLimit{\limitervariable}$.
\end{remark}
\subsubsection{Calculations for the Euler equation}\label{subsec: Euler}
The two-dimensional compressible Euler equations for the flow of an ideal gas are given by
\begin{equation}\label{eulereq}
\left.
\begin{alignedat}{3}
\hspace*{2cm}
\dt \density &+ \dx \momentumdim{1} &&+ \dy \momentumdim{2} && = 0, \\[0.05cm]
\dt \momentumdim{1} &+ \dx \left(\frac{\momentumdim{1}^2}{\density} + \pressure\right) &&+ \dy \left(\frac{\momentumdim{1}\momentumdim{2}}{\rho}\right) && = 0,\\[0.22cm]
\dt \momentumdim{2} &+ \dx \left(\frac{\momentumdim{1}\momentumdim{2}}{\rho}\right) &&+ \dy \left(\frac{\momentumdim{2}^2}{\density} + \pressure\right) && = 0, \\[0.22cm]
\dt \energy &+ \dx \left( (\energy + \pressure) \,\frac{\momentumdim{1}}{\density}\right) &&+ \dy\left( (\energy + \pressure) \,\frac{\momentumdim{2}}{\density}\right)&& = 0, \hspace*{2cm}
\end{alignedat}
\right\}
\end{equation}
where $\density$ describes the mass density, $\momentumdim{1}$ and $\momentumdim{2}$ the momentum in $\x$ and $\y$ direction and $\energy$ the energy of the gas. The three equations model the conservation of mass, momentum and energy. The pressure $\pressure$ reads
$$\pressure = (\eulergamma-1)\left(\energy - \frac{1}{2}\frac{\big(\momentumdim{1}^2+\momentumdim{2}^2\big)}{\density}\right),$$
with the adiabatic constant $\eulergamma>1$. 
The \hypset{} is given by
$$	\realizableSet = \left\{ \solution = \begin{pmatrix}
\density\\\momentumdim{1}\\ \momentumdim{2}\\\energy
\end{pmatrix}\,\Bigg|~\density>0,~\pressure = (\eulergamma-1)\left(\energy - \frac{1}{2}\frac{\big(\momentumdim{1}^2+\momentumdim{2}^2\big)}{\density}\right) >0\right\}.$$

\begin{lemma}
	Let $\widetilde{\solution}= (\densityvar,\momentumdimvar{1},\momentumdimvar{2},\energyvar)^T\in\realizableSet$ and $\solution= (\density,\momentumdim{1},\momentumdim{2},\energy)^T\in\R^4$ be arbitrary. Then the solution of the hyperbolicity limiter problem
	\begin{align*}
		\min~ &\limitervariable\\
		\text{s.t.}~ &\limitervariable\in[0,1]\\
		&\limitervariable\widetilde{\solution} + (1-\limitervariable)\solution\in\realizableSet
	\end{align*}
	for the two-dimensional Euler equations has the solution
	\begin{align}
		\limitervariable^\star &= \max\left(\cutfun{\limitervariable_1},\cutfun{\limitervariable_{2+}},\cutfun{\limitervariable_{2-}}\right),\label{eq:thetastar}\\
				\cutfun{x} &= x\indicator{[0,1]}(x),\nonumber\\
		\limitervariable_1 &= \frac{\density}{\density-\densityvar},\nonumber\\ 
		\limitervariable_{2\pm} &=  \frac{\density\energyvar - 2\density\energy + \densityvar\energy - \momentumdim{1}\momentumdimvar{1} - \momentumdim{2}\momentumdimvar{2} \pm \sqrt{\widetilde{\theta} }}
		{\momentumdim{1}^2 - 2\momentumdim{1}\momentumdimvar{1} + \momentumdim{2}^2 - 2\momentumdim{2}\momentumdimvar{2} + \momentumdimvar{1}^2 + \momentumdimvar{2}^2 - 2\density\energy + 2\density\energyvar + 2\densityvar\energy - 2\densityvar\energyvar}\,,\nonumber
	\end{align}
	where
\begin{align}\widetilde{\theta} &=\density^2\energyvar^2 - 2\density\densityvar\energy\energyvar + 2\density\energy\momentumdimvar{1}^2 + 2\density\energy\momentumdimvar{2}^2 - 2\density\energyvar\momentumdim{1}\momentumdimvar{1} 
	- 2\density\energyvar\momentumdim{2}\momentumdimvar{2} \nonumber\\
	&+ \densityvar^2\energy^2 - 2\densityvar\energy\momentumdim{1}\momentumdimvar{1} - 2\densityvar\energy\momentumdim{2}\momentumdimvar{2} + 2\densityvar\energyvar\momentumdim{1}^2 + 2\densityvar\energyvar\momentumdim{2}^2 - \momentumdim{1}^2\momentumdimvar{2}^2 + 2\momentumdim{1}\momentumdim{2}\momentumdimvar{1}\momentumdimvar{2} - \momentumdim{2}^2\momentumdimvar{1}^2. \nonumber
	\end{align}
\end{lemma}
\begin{proof}
	We derive the values of the limiter variable $\limitervariable$, for which the expression $\limitervariable\widetilde{\solution} + (1-\limitervariable)\solution$ is in the \hypset{} $\realizableSet$.  In order to have a positive density we require
	$$\limitervariable\densityvar + (1-\limitervariable)\density > 0,$$
	yielding 
	$$\limitervariable < \frac{\density}{\density-\densityvar}=\limitervariable_1,$$
	if $\densityvar\neq\density$. Otherwise, we have $\density=\densityvar>0$ and set $\limitervariable_1=0$.
	Analogously, we calculate the pressure term and check it for positivity
	$$\frac{\pressure}{\eulergamma-1} = \limitervariable\energyvar + (1-\limitervariable)\energy - \frac{1}{2}\frac{\Big(\big(\limitervariable\momentumdimvar{1} + (1-\limitervariable)\momentumdim{1} \big)^2 + \big(\limitervariable\momentumdimvar{2} + (1-\limitervariable)\momentumdim{2} \big)^2\Big)}
	{\limitervariable\densityvar + (1-\limitervariable)\density } \overset{!}{>}0.$$ 
	Since the density is supposed to be positive we can multiply it to the above equation
	$$
	\frac{\pressure \,\big( \limitervariable\densityvar + (1-\limitervariable)\density\big) }{\eulergamma-1} 
	= \big( \limitervariable\energyvar + (1-\limitervariable)\energy\big) \big( \limitervariable\densityvar + (1-\limitervariable)\density\big) 
	- \frac{1}{2}\Big(\big(\limitervariable\momentumdimvar{1} + (1-\limitervariable)\momentumdim{1} \big)^2 + \big(\limitervariable\momentumdimvar{2} + (1-\limitervariable)\momentumdim{2} \big)^2\Big) \overset{!}{>}0.$$
	Writing this inequality as a polynomial in $\limitervariable$ we obtain
    \begin{align*} \bigg((\energy - \energyvar) (\density - \densityvar) &-\frac{1}{2}\left(\momentumdim{1}-\momentumdimvar{1}\right)^2  -\frac{1}{2}\left(\momentumdim{2}-\momentumdimvar{2}\right)^2\bigg)\,\limitervariable^2\\ &+\Big(\momentumdim{1}\left(\momentumdim{1}-\momentumdimvar{1}\right) + \momentumdim{2}\left(\momentumdim{2}-\momentumdimvar{2}\right) - \density(\energy-\energyvar)-\energy(\density - \densityvar)\Big)\,\limitervariable \\
    &+\energy\density-\frac{\momentumdim{1}^2}{2} -\frac{\momentumdim{2}^2}{2} \overset{!}{>}0. 
    \end{align*}
	The corresponding equality has the roots $\limitervariable_{2+}$ and $\limitervariable_{2-}$. 
	Combining these results with $\limitervariable\in[0,\,1]$, we end up with \eqref{eq:thetastar}.
\end{proof}
\begin{remark}
	In the previous lemma, the quantity $\widetilde{\solution}$ plays the role of the cell mean $\cellmeant[\cellind,\cellindR]{\solution}$ in \thmref{thm:realizable}, whereas $\solution$ is given by the point values $\solution(\timevar,\quadxpoint{\xQuadIndex},\quadypointG{\yQuadIndexG},\quadRpoint{\xiQuadIndex})$ and $\solution(\timevar,\quadxpointG{\xQuadIndexG},\quadypoint{\yQuadIndex},\quadRpoint{\xiQuadIndex})$.
\end{remark}

%% file: Sections/results.tex
\section{Numerical Results}\label{sec:results}
In the following numerical experiments we analyze the hyperbolicity-preserving discontinuous stochastic Galerkin (\hdSG{}) method for different problems of the Euler equations.
In \secref{subsec:convTest} we apply it to a smooth solution and examine its convergence.
In \secref{subsec:compSC} we compare the performance of our scheme to the non-intrusive Stochastic Collocation method.
We evaluate both methods for a smooth solution which depends on three random variables.
Finally, in \secref{subsec:Euler Equations: Uncertain Sod Shock Test} and \secref{subsec:dmr}, 
we apply the \hdSG{} method  to a one- and two-dimensional Riemann problem and examine the behavior of the
hyperbolicity-preserving limiter. When we employ the Multi-Element (ME) ansatz from \secref{subsec:Multielement Stochastic Galerkin}, we call our method ME-\hdSG{}.

As a numerical solver for the SG system we use a modified version of the Runge--Kutta discontinuous Galerkin solver Flexi \cite{Hindenlang2012} with
a SSP-RK time-discretization of order four. We set the CFL-number from \assref{ass:ppFlux} to $C=0.45$. 
In the following numerical examples
we measure the error in mean and variance at time $\timevar=T$ in the $\Lp{1}(\Domain)$- or $\Lp{2}(\Domain)$-norm which we approximate 
by a tensor product Gau\ss ~quadrature rule with $15$ points (in one dimension) in every physical cell and $20$ points (in one dimension) in every Multi-Element. We denote the number of spatial cells by $\ncells$ and the number of MEs by $\nRcells$.
\subsection{Convergence tests}\label{subsec:convTest}
This example is devoted to the convergence of the \hdSG{} method for smooth solutions of
the Euler equations.
We let the spatial domain $\Domain = [0,\,1]^2$, where we use periodic boundary conditions
and set the adiabatic constant to $\eulergamma = 1.4$. 
As numerical flux we use Lax-Friedrichs. 
To obtain an analytical solution of the Euler equations, we consider the method of manufactured solutions. To this end we introduce an additional source term $S$  in \eqref{eulereq}, where
\begin{align}
	S(\timevar,\x,\y,\uncertainty) := \begin{pmatrix}
	\dt \density + \dx \momentumdim{1} + \dy \momentumdim{2} \\[0.05cm]
    \dt \momentumdim{1} + \dx \left(\frac{\momentumdim{1}^2}{\density} + \pressure\right) + \dy \left(\frac{\momentumdim{1}\momentumdim{2}}{\rho}\right) \\[0.22cm]
    \dt \momentumdim{2} + \dx \left(\frac{\momentumdim{1}\momentumdim{2}}{\rho}\right) + \dy \left(\frac{\momentumdim{2}^2}{\density} + \pressure\right)  \\[0.22cm]
   \dt \energy + \dx \left( (\energy + \pressure) \,\frac{\momentumdim{1}}{\density}\right) + \dy\left( (\energy + \pressure) \,\frac{\momentumdim{2}}{\density}\right)
	\end{pmatrix}.
\end{align}
This additional source term allows us to consider the following analytical solution of the Euler equations
\begin{align} \label{eq:euler2DSmoothSolution}
	\solution(\timevar,\x,\y,\uncertaintyD_1,\uncertaintyD_2,\uncertaintyD_3) = 
	\begin{pmatrix}
	\density(\timevar,\x,\y,\uncertaintyD_1,\uncertaintyD_2,\uncertaintyD_3) \\
	 \momentumdim{1}(\timevar,\x,\y,\uncertaintyD_1,\uncertaintyD_2,\uncertaintyD_3)  \\
	 \momentumdim{2}(\timevar,\x,\y,\uncertaintyD_1,\uncertaintyD_2,\uncertaintyD_3)  \\
	\energy(\timevar,\x,\y,\uncertaintyD_1,\uncertaintyD_2,\uncertaintyD_3)  
	\end{pmatrix}=
	\begin{pmatrix}
	\uncertaintyD_3+ \uncertaintyD_2\cos(2\pi (\x- \uncertaintyD_1\timevar)) \\[0.1cm]
	\uncertaintyD_3+ \uncertaintyD_2\cos(2\pi (\x- \uncertaintyD_1\timevar)) \\
	0 \\
	(\uncertaintyD_3+ \uncertaintyD_2\cos(2\pi (\x- \uncertaintyD_1\timevar)))^2
	\end{pmatrix},
\end{align}
with uniformly distributed random variables $\uncertaintyD_1\sim \mathcal{U}(0.1,1)$,
$\uncertaintyD_2\sim \mathcal{U}(0.1,0.3)$ and $\uncertaintyD_3\sim \mathcal{U}(1.8,2.5)$.

\subsubsection{Euler Equations: smooth solution, spatial refinement}\label{subsec:Euler Equations: Smooth Solution, spatial refinement}
As a first benchmark example we consider a refinement of the spatial domain  $\Domain$ for a SG polynomial degree 
of $\SGtruncorder=10$ and one Multi-Element, i.e. $\nRcells=1$, as well as DG polynomial degrees of one and three.
We choose the following analytical function
\begin{align} \label{eq:euler2DSmoothSolution1DStoch}
	\solution(\timevar,\x,\y,\uncertaintyD_1) = 
	\begin{pmatrix}
	\density(\timevar,\x,\y,\uncertaintyD_1) \\
	\momentumdim{1}(\timevar,\x,\y,\uncertaintyD_1)  \\
	\momentumdim{2}(\timevar,\x,\y,\uncertaintyD_1)  \\
	\energy(\timevar,\x,\y,\uncertaintyD_1)  
	\end{pmatrix}=
	\begin{pmatrix}
	2+ 0.1(2\pi (\x- \uncertaintyD_1\timevar)) \\
	2+ 0.1\cos(2\pi (\x- \uncertaintyD_1\timevar)) \\
	0 \\
	(2+ 0.1\cos(2\pi (\x- \uncertaintyD_1\timevar)))^2
	\end{pmatrix},
\end{align}
where $\uncertaintyD_1\sim \mathcal{U}(0.1,1)$. We compute the numerical approximation of 
\eqref{eq:euler2DSmoothSolution1DStoch} up to $T=1$.
\tabref{table:EulerSpatialRefinement} displays the error in mean and variance of the density.
\refone{The experimental order of convergence (eoc) is computed by 
\begin{equation}\label{eq:eoc}
	\text{eoc} = \frac{\log\Big(\frac{\text{error}(r)}{\text{error}(r+1)}\Big)}{\log\Big(\frac{\text{dof}(r+1)}{\text{dof}(r)}\Big)}, 
\end{equation}
where the degrees of freedom $\text{dof}(r)$ corresponds to the number of spatial cells $\ncells$ or Multi-Elements $\nRcells$ and $r$ represents the level of refinement.}

The error is clearly dominated by the spatial error as the resolution in the stochastic space is sufficiently high. 
Thus, it converges with the rate of the DG method, which is $(\polydegree+1)/2$ in two spatial dimensions. 

\begin{table}[htb!]
\centering
\begin{filecontents*}{Images/convtestEuler2D/hRefEuler2Dp1.csv}
n,l2mean,eocmean,l2variance,eocvariance
4,2.7033e-02,-,5.9515e-02,-
16,6.7904e-03,1.00,1.0958e-02,1.22
64, 1.1669e-03,1.27,2.466e-03,1.08
256,1.6904e-04,1.39,3.8439e-04,1.34
1024,2.4801e-05,1.38,6.967e-05,1.23
4096,4.2716e-06,1.27,1.5161e-05,1.10
\end{filecontents*}
\pgfplotstabletypeset[
    col sep=comma,
    string type,
    every head row/.style={%
        before row={\hline
            \multicolumn{5}{|c|}{\hdSG{}, $\polydegree=1$}  \\
            \hline
        },
        after row=\hline
    },
    every last row/.style={after row=\hline},
    columns/n/.style={column name=$\ncells$, column type={|c}},
    columns/l2mean/.style={column name=$L_2$-Mean, column type={|c}},
    columns/eocmean/.style={column name=eoc, column type={|c}},
    columns/l2variance/.style={column name=$L_2$-Variance, column type={|c}},
    columns/eocvariance/.style={column name=eoc, column type={|c|}},
    ]{Images/convtestEuler2D/hRefEuler2Dp1.csv}
\begin{filecontents*}{Images/convtestEuler2D/hRefEuler2Dp3.csv}
n,l2mean,eocmean,l2variance,eocvariance
4,9.6072e-04,-,3.4636e-03,-
16,4.1815e-05,2.26,1.0036e-04,2.55
64,1.7619e-06,2.28,4.3666e-06,2.26
256,7.1412e-08,2.31,3.9645e-07,1.73
1024,2.9046e-09,2.31,1.7748e-08,2.24
4096,1.6708e-10,2.06,9.6501e-10,2.10
\end{filecontents*}
\pgfplotstabletypeset[
    col sep=comma,
    string type,
    every head row/.style={%
        before row={\hline
            \multicolumn{5}{|c|}{\hdSG{},$\polydegree=3$}  \\
            \hline
        },
        after row=\hline
    },
    every last row/.style={after row=\hline},
    columns/n/.style={column name=$\ncells$, column type={|c}},
    columns/l2mean/.style={column name=$L_2$-Mean, column type={|c}},
    columns/eocmean/.style={column name=eoc, column type={|c}},
    columns/l2variance/.style={column name=$L_2$-Variance, column type={|c}},
    columns/eocvariance/.style={column name=eoc, column type={|c|}},
    ]{Images/convtestEuler2D/hRefEuler2Dp3.csv}
    \caption{$\Lp{2}$-errors and experimental order of convergence (eoc) for the Euler equations (density)
     with $\SGtruncorder=10$, $\nRcells=1$, for DG polynomial degrees $\polydegree=1,3$.
     Example \ref{subsec:Euler Equations: Smooth Solution, spatial    
     refinement}.}
    \label{table:EulerSpatialRefinement} 
\end{table}

\subsubsection{Euler Equations: smooth solution, stochastic refinement}\label{subsec:Euler Equations: Smooth Solution, stochastic refinement}
As next numerical example we consider a stochastic refinement, where we increase the SG polynomial degree $\SGtruncorder$.
We consider the same analytical function \eqref{eq:euler2DSmoothSolution1DStoch} from the previous numerical experiment.
In \tabref{table:EulerqRefinement} we show a $\SGtruncorder$-refinement for one random element, 
i.e. $\nRcells=1$.
Here, the physical mesh consists of $\ncells=400$ cells and a DG polynomial degree of six.
For the smooth solution \eqref{eq:euler2DSmoothSolution1DStoch} we observe that the error exhibits the expected spectral convergence when we increase the SG polynomial degree.
\begin{table}[htb!]
\centering
\begin{filecontents*}{Images/convtestEuler2D/pRefSG1D.csv}
n,l2mean,l2variance
1,6.7406e-06,1.3562e-05
2,3.1228e-06,8.5896e-06
3,1.6532e-06,4.5845e-06
4,5.5879e-07,1.6006e-06
5,1.2829e-07,4.3977e-07
6,2.1531e-08,8.7639e-08
7,2.7923e-09,1.2051e-08
8,2.9392e-10,1.4562e-09
9,3.5925e-11,1.131e-10
10,2.4548e-11,4.7307e-11
\end{filecontents*}
\pgfplotstabletypeset[
    col sep=comma,
    string type,
    every head row/.style={%
        before row={\hline
            \multicolumn{3}{|c|}{\hdSG{}} \\
            \hline
        },
        after row=\hline
    },
    every last row/.style={after row=\hline},
    columns/n/.style={column name=$\SGtruncorder$, column type={|c}},
    columns/l2mean/.style={column name=$L_2$-Mean, column type={|c}},
    columns/l2variance/.style={column name=$L_2$-Variance, column type={|c|}},
    ]{Images/convtestEuler2D/pRefSG1D.csv}
    \caption{$\Lp{2}$-errors for the Euler equations (density)
    with $\nRcells=1$, $\ncells=400$ for DG polynomial degree $\polydegree=6$. 
    Example \ref{subsec:Euler Equations: Smooth Solution, stochastic refinement}.}
    \label{table:EulerqRefinement} 
\end{table}

\subsection{Comparison with Stochastic Collocation method} \label{subsec:compSC}
In this section we compare the \hdSG{}-method with a common non-intrusive method, namely the Stochastic Collocation (SC) 
method as described in \cite{HesthavenXiu2005}.
When we use SC in combination with the ME scheme from \secref{subsec:Multielement Stochastic Galerkin},
we call this method ME-SC.
For the collocation points in a multi-dimensional random space $\randomSpace$
we use tensor-products of one-dimensional Gau\ss-Legendre quadrature points. 
The number of collocation points in one dimension is always $\SGtruncorder+1$.

To compare both methods, we measure the total elapsed time (including post-processing for SC)
on a workstation equipped  with an AMD Ryzen ThreadRipper 2950x processor with sixteen kernels and 128 GB RAM.
As a deterministic solver for the SC method we use the standard version of Flexi \cite{Hindenlang2012}.
Both methods are parallelized using Open MPI and we measure the total elapsed time on sixteen kernels.
\refone{We denote the total elapsed time by wall time.}

As a benchmark solution we refer to the analytical solution \eqref{eq:euler2DSmoothSolution},
which is computed up to $T=1$ and 
we successively increase the number of random variables starting from $\uncertaintyD_1$ 
while fixing $\uncertaintyD_2=0.1 $, $\uncertaintyD_3=2$ and so on.
The number of physical cells is always $\ncells=400 $ and the DG polynomial degree is six. 
As numerical flux we choose Lax-Friedrichs. 
\subsubsection{Comparison with SC: $\SGtruncorder$-refinement} \label{subsec:KRefinement}
In this example we consider a global $\SGtruncorder$-refinement for one fixed ME, i.e. $\nRcells=1$.
In \figref{fig:pRefwalltime} and Tables \ref{table:KRefinement1D}, \ref{table:KRefinement2D}, \ref{table:KRefinement3D} we
compare the error vs. wall time for the \hdSG{} and the SC method in one, two and three random dimensions. 
In all three cases the \hdSG{} method yields a smaller absolute error for the same number of degrees of freedoms,
however, only in the one-dimensional case the \hdSG{} scheme proves to be more efficient (in terms of error vs. wall time)
than the SC method. 

\begin{table}[htb!]
\centering
\begin{filecontents*}{Images/pRef2DEuler/pRefSG1D.csv}
degree,l2mean,l2variance, walltime
1,6.7406e-06,1.3562e-05,6.0
2,3.1228e-06,8.5896e-06,6.48
3,1.6532e-06,4.5845e-06,7.88
4,5.5879e-07,1.6006e-06,9.70
5,1.2829e-07,4.3977e-07,10.58
6,2.1531e-08,8.7639e-08,12.83
7,2.7923e-09,1.2051e-08,14.75
8,2.9392e-10,1.4562e-09,18.16
9,3.5925e-11,1.131e-10,20.91
10,2.4548e-11,4.7307e-11,24.85
\end{filecontents*}
\pgfplotstabletypeset[
    col sep=comma,
    string type,
    every head row/.style={%
        before row={\hline
            \multicolumn{4}{|c|}{\hdSG{}, $\nstochdim=1$}  \\
            \hline
        },
        after row=\hline
    },
    every last row/.style={after row=\hline},
    columns/degree/.style={column name=$\SGtruncorder$, column type={|c}},
    columns/l2mean/.style={column name=$L_2$-Mean, column type={|c}},
    columns/l2variance/.style={column name=$L_2$-Variance, column type={|c}},
    columns/walltime/.style={column name=wall time [s], column type={|c|}},
    ]{Images/pRef2DEuler/pRefSG1D.csv}   
\begin{filecontents*}{Images/pRef2DEuler/pRefSC1D.csv}
degree,l2mean,l2variance, walltime
1,4.5363e-01,7.9044e-02,7.0
2,8.7027e-02,0.044842,6.0
3,6.4479e-03,1.8557e-02,6.0
4,2.4543e-04,3.371e-03,11.0
5,5.6894e-06,3.4898e-04,12.0
6,8.8712e-08,2.3519e-05,11.0
7,1.1123e-09,1.1181e-06,16.0
8,1.6271e-10,3.96e-08,17.0
9,1.6612e-10,1.0709e-09,16.0
10,1.661e-10,4.546e-11,16.0
\end{filecontents*}
\pgfplotstabletypeset[
    col sep=comma,
    string type,
    every head row/.style={%
        before row={\hline
            \multicolumn{4}{|c|}{SC, $\nstochdim=1$}  \\
            \hline
        },
        after row=\hline
    },
    every last row/.style={after row=\hline},
    columns/degree/.style={column name=$\SGtruncorder$, column type={|c}},
    columns/l2mean/.style={column name=$L_2$-Mean, column type={|c}},
    columns/l2variance/.style={column name=$L_2$-Variance, column type={|c}},
    columns/walltime/.style={column name=wall time [s], column type={|c|}},
    ]{Images/pRef2DEuler/pRefSC1D.csv}   
    \caption{$\Lp{2}$-errors and wall time for the Euler equations (density) for \hdSG{} and SC method in one random dimension.
     Example \ref{subsec:KRefinement}.}
    \label{table:KRefinement1D}
\end{table}

\begin{table}[htb!]
\centering
\begin{filecontents*}{Images/pRef2DEuler/pRefSG2D.csv}
degree,l2mean,l2variance, walltime
1,1.2374e-04,2.5015e-04,79.67
2,6.041e-05,0.00018524,93.53
3,3.5254e-05,0.00012008,121.93
4,1.4209e-05,5.7537e-05,172.90
5,5.2814e-06,2.1894e-05,220.53
6,1.5886e-06,6.3609e-06,336.04
7,3.5052e-07,1.4504e-06,467.0
8,5.8279e-08,2.4925e-07,700.04
\end{filecontents*}
\pgfplotstabletypeset[
    col sep=comma,
    string type,
    every head row/.style={%
        before row={\hline
            \multicolumn{4}{|c|}{\hdSG{}, $\nstochdim=2$}  \\
            \hline
        },
        after row=\hline
    },
    every last row/.style={after row=\hline},
    columns/degree/.style={column name=$\SGtruncorder$, column type={|c}},
    columns/l2mean/.style={column name=$L_2$-Mean, column type={|c}},
    columns/l2variance/.style={column name=$L_2$-Variance, column type={|c}},
    columns/walltime/.style={column name=wall time [s], column type={|c|}},
    ]{Images/pRef2DEuler/pRefSG2D.csv}   
\begin{filecontents*}{Images/pRef2DEuler/pRefSC2D.csv}
degree,l2mean,l2variance, walltime
1,4.5363e-01,8.5711e-02,9.0
2,8.7027e-02,4.8613e-02,15.0
3,6.4479e-03,2.0116e-02,20.0
4,2.4543e-04,3.6523e-03,30.0
5,5.6894e-06,3.7807e-04,45.0
6,8.8714e-08,2.5479e-05,60.0
7,1.1108e-09,1.2113e-06,81.0
8,1.6736e-10,4.2899e-08,100.0
\end{filecontents*}
\pgfplotstabletypeset[
    col sep=comma,
    string type,
    every head row/.style={%
        before row={\hline
            \multicolumn{4}{|c|}{SC, $\nstochdim=2$}  \\
            \hline
        },
        after row=\hline
    },
    every last row/.style={after row=\hline},
    columns/degree/.style={column name=$\SGtruncorder$, column type={|c}},
    columns/l2mean/.style={column name=$L_2$-Mean, column type={|c}},
    columns/l2variance/.style={column name=$L_2$-Variance, column type={|c}},
    columns/walltime/.style={column name=wall time [s], column type={|c|}},
    ]{Images/pRef2DEuler/pRefSC2D.csv}   
    \caption{$\Lp{2}$-errors and wall time for the Euler equations (density) for \hdSG{} and SC method in two random dimensions. 
    Example \ref{subsec:KRefinement}.}
    \label{table:KRefinement2D}
\end{table}

\begin{table}[htb!]
\centering
\begin{filecontents*}{Images/pRef2DEuler/pRefSG3D.csv}
degree,l2mean,l2variance, walltime
1,1.1315e-03,3.51281e-03,1844.0
2,2.192e-04,8.2764e-04,2722.0
3,5.3598e-05,2.4934e-04,5406.0
4,1.2283e-05,5.4753e-05,14583.0
\end{filecontents*}
\pgfplotstabletypeset[
    col sep=comma,
    string type,
    every head row/.style={%
        before row={\hline
            \multicolumn{4}{|c|}{\hdSG{}, $\nstochdim=3$}  \\
            \hline
        },
        after row=\hline
    },
    every last row/.style={after row=\hline},
    columns/degree/.style={column name=$\SGtruncorder$, column type={|c}},
    columns/l2mean/.style={column name=$L_2$-Mean, column type={|c}},
    columns/l2variance/.style={column name=$L_2$-Variance, column type={|c}},
    columns/walltime/.style={column name=wall time [s], column type={|c|}},
    ]{Images/pRef2DEuler/pRefSG3D.csv}   
\begin{filecontents*}{Images/pRef2DEuler/pRefSC3D.csv}
degree,l2mean,l2variance, walltime
1,4.5363e-01,2.4904e-01,86.0
2,8.7027e-02,4.8613e-02,96.0
3,6.4479e-03,2.0116e-02,126.0
4,2.4543e-04,3.6523e-03,182.0
\end{filecontents*}
\pgfplotstabletypeset[
    col sep=comma,
    string type,
    every head row/.style={%
        before row={\hline
            \multicolumn{4}{|c|}{SC, $\nstochdim=3$}  \\
            \hline
        },
        after row=\hline
    },
    every last row/.style={after row=\hline},
    columns/degree/.style={column name=$\SGtruncorder$, column type={|c}},
    columns/l2mean/.style={column name=$L_2$-Mean, column type={|c}},
    columns/l2variance/.style={column name=$L_2$-Variance, column type={|c}},
    columns/walltime/.style={column name=wall time [s], column type={|c|}},
    ]{Images/pRef2DEuler/pRefSC3D.csv}   
    \caption{$\Lp{2}$-errors and wall time for the Euler equations (density) for \hdSG{} and SC method in three random dimensions. 
    Example \ref{subsec:KRefinement}.}
    \label{table:KRefinement3D}
\end{table}
\begin{figure}[htb!]
\centering
  \input{Images/pRef2DEuler/walltimepRef2D.tex}
  \caption{$\Lp{2}$-errors and wall time for the Euler equations (density) for \hdSG{} and SC method.
    Example \ref{subsec:KRefinement}.}
    \label{fig:pRefwalltime}
\end{figure}
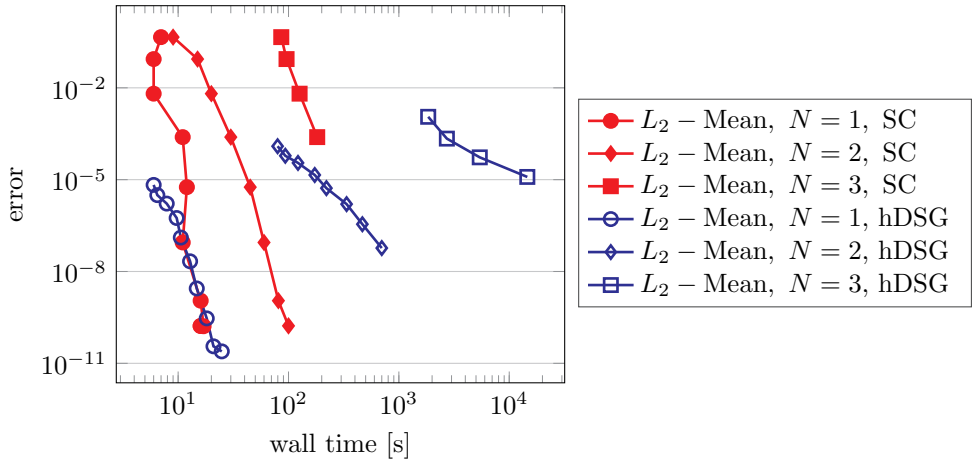

\subsubsection{Comparison with SC: ME-refinement} \label{subsec:MERefinement}
In this numerical experiment we compare error vs. wall time for the ME-\hdSG{} and ME-SC methods, while increasing the number of
MEs. 
For both methods we consider a linear interpolation,
i.e. $\SGtruncorder=1$. Again, \figref{fig:MEwalltime} and Tables \ref{table:MERefinement1D}, \ref{table:MERefinement2D}, \ref{table:MERefinement3D}
show the error vs. wall time for both methods. Similar to the $\SGtruncorder$-refinement in \secref{subsec:KRefinement},
we observe for a one-dimensional random space $\randomSpace$, that ME-\hdSG{} is clearly more efficient than 
the ME-SC method. In the two-dimensional case and a small number of MEs we deduce similar results. However, for three random dimensions,
ME-\hdSG{} stands no chance against the ME-SC method in terms of efficiency.
This is exactly the behavior that we expect from \hdSG{}, because computing the orthogonal projection of
the fluxes in \eqref{SGsystem} becomes more and more expensive when we increase the number of random dimensions.
Hence, for a one and maybe a two-dimensional random space, the \hdSG{} method yields an efficiency gain compared to SC and is not competitive beyond two dimensions.
\begin{table}[htb!]
\centering
\begin{filecontents*}{Images/MERef2DEuler/MERefSG1D.csv}
degree,l2mean,l2variance, walltime
1,6.7125e-06,1.2071e-05,3.71
2,1.3863e-06,4.3691e-06,7.23
3,3.5026e-07,1.162e-06,11.42
4,1.2083e-07,4.3255e-07,14.36
5,5.1526e-08,1.7694e-07,19.85
6,2.5395e-08,8.9625e-08,26.79
7,1.3884e-08,4.8357e-08,28.32
8,8.2032e-09,2.9729e-08,28.98
9,5.1462e-09,1.8097e-08,46.52
10,3.3861e-09,1.1955e-08,51.49
\end{filecontents*}
\pgfplotstabletypeset[
    col sep=comma,
    string type,
    every head row/.style={%
        before row={\hline
            \multicolumn{4}{|c|}{ME-\hdSG{}, $\nstochdim=1$}  \\
            \hline
        },
        after row=\hline
    },
    every last row/.style={after row=\hline},
    columns/degree/.style={column name=$\nRcells$, column type={|c}},
    columns/l2mean/.style={column name=$L_2$-Mean, column type={|c}},
    columns/l2variance/.style={column name=$L_2$-Variance, column type={|c}},
    columns/walltime/.style={column name=wall time [s], column type={|c|}},
    ]{Images/MERef2DEuler/MERefSG1D.csv}   
\begin{filecontents*}{Images/MERef2DEuler/MERefSC1D.csv}
degree,l2mean,l2variance, walltime
1,8.7027e-02,4.4842e-02,6.0
2,1.0882e-03,8.3017e-02,12.0
3,1.8291e-04,4.2943e-04,11.0
4,5.4915e-05,1.0471e-04,17.0
5,2.1967e-05,3.859e-05,17.0
6,1.046e-05,1.7628e-05,22.0
7,5.6031e-06,9.2177e-06,27.0
8,3.2683e-06,5.295e-06,27.0
9,2.0335e-06,3.2607e-06,31.0
10,1.3309e-06,2.1186e-06,32.0
\end{filecontents*}
\pgfplotstabletypeset[
    col sep=comma,
    string type,
    every head row/.style={%
        before row={\hline
            \multicolumn{4}{|c|}{ME-SC, $\nstochdim=1$}  \\
            \hline
        },
        after row=\hline
    },
    every last row/.style={after row=\hline},
    columns/degree/.style={column name=$\nRcells$, column type={|c}},
    columns/l2mean/.style={column name=$L_2$-Mean, column type={|c}},
    columns/l2variance/.style={column name=$L_2$-Variance, column type={|c}},
    columns/walltime/.style={column name=wall time [s], column type={|c|}},
    ]{Images/MERef2DEuler/MERefSC1D.csv}   
    \caption{$\Lp{2}$-errors and wall time for the Euler equations (density) for ME-\hdSG{} and ME-SC method in one random dimension. 
    Example \ref{subsec:MERefinement}.}
    \label{table:MERefinement1D}
\end{table}

\begin{table}[htb!]
\centering
\begin{filecontents*}{Images/MERef2DEuler/MERefSG2D.csv}
degree,l2mean,l2variance, walltime
1,1.2374e-04,2.5016e-04,33.26
4,5.7001e-05,2.9659e-04,134.32
9,1.6892e-05,1.1411e-04,432.19
16,6.1867e-06,4.6866e-05,526.07
\end{filecontents*}
\pgfplotstabletypeset[
    col sep=comma,
    string type,
    every head row/.style={%
        before row={\hline
            \multicolumn{4}{|c|}{ME-\hdSG{}, $\nstochdim=2$}  \\
            \hline
        },
        after row=\hline
    },
    every last row/.style={after row=\hline},
    columns/degree/.style={column name=$\nRcells$, column type={|c}},
    columns/l2mean/.style={column name=$L_2$-Mean, column type={|c}},
    columns/l2variance/.style={column name=$L_2$-Variance, column type={|c}},
    columns/walltime/.style={column name=wall time [s], column type={|c|}},
    ]{Images/MERef2DEuler/MERefSG2D.csv}   
\begin{filecontents*}{Images/MERef2DEuler/MERefSC2D.csv}
degree,l2mean,l2variance, walltime
1,8.7027e-02,4.8613e-02,15.0
4,1.0882e-03,8.9915e-03,30.0
9,1.8291e-04,4.6488e-04,60.0
16,5.4915e-05,1.1334e-04,101.0
\end{filecontents*}
\pgfplotstabletypeset[
    col sep=comma,
    string type,
    every head row/.style={%
        before row={\hline
            \multicolumn{4}{|c|}{ME-SC, $\nstochdim=2$}  \\
            \hline
        },
        after row=\hline
    },
    every last row/.style={after row=\hline},
    columns/degree/.style={column name=$\nRcells$, column type={|c}},
    columns/l2mean/.style={column name=$L_2$-Mean, column type={|c}},
    columns/l2variance/.style={column name=$L_2$-Variance, column type={|c}},
    columns/walltime/.style={column name=wall time [s], column type={|c|}},
    ]{Images/MERef2DEuler/MERefSC2D.csv}   
    \caption{$\Lp{2}$-errors and wall time for the Euler equations (density) for ME-\hdSG{} and ME-SC method in two random dimensions. 
    Example \ref{subsec:MERefinement}.}
    \label{table:MERefinement2D}
\end{table}
\begin{table}[!htb]
\centering
\begin{filecontents*}{Images/MERef2DEuler/MERefSG3D.csv}
degree,l2mean,l2variance, walltime
1,1.1316e-03,3.5136e-03,239.63
8,1.5013e-04,7.9673e-04,1799.0
27,3.5281e-05,2.3186e-04,10432.13
\end{filecontents*}
\pgfplotstabletypeset[
    col sep=comma,
    string type,
    every head row/.style={%
        before row={\hline
            \multicolumn{4}{|c|}{ME-\hdSG{}, $\nstochdim=3$}  \\
            \hline
        },
        after row=\hline
    },
    every last row/.style={after row=\hline},
    columns/degree/.style={column name=$\nRcells$, column type={|c}},
    columns/l2mean/.style={column name=$L_2$-Mean, column type={|c}},
    columns/l2variance/.style={column name=$L_2$-Variance, column type={|c}},
    columns/walltime/.style={column name=wall time [s], column type={|c|}},
    ]{Images/MERef2DEuler/MERefSG3D.csv}   
\begin{filecontents*}{Images/MERef2DEuler/MERefSC3D.csv}
degree,l2mean,l2variance, walltime
1,8.7027e-02,4.8613e-02,98.0
8,1.0882e-03,8.9915e-03,183.0
27,1.8291e-04,4.6488e-04,414.0
\end{filecontents*}
\pgfplotstabletypeset[
    col sep=comma,
    string type,
    every head row/.style={%
        before row={\hline
            \multicolumn{4}{|c|}{ME-SC, $\nstochdim=3$}  \\
            \hline
        },
        after row=\hline
    },
    every last row/.style={after row=\hline},
    columns/degree/.style={column name=$\nRcells$, column type={|c}},
    columns/l2mean/.style={column name=$L_2$-Mean, column type={|c}},
    columns/l2variance/.style={column name=$L_2$-Variance, column type={|c}},
    columns/walltime/.style={column name=wall time [s], column type={|c|}},
    ]{Images/MERef2DEuler/MERefSC3D.csv}   
    \caption{$\Lp{2}$-errors and wall time for the Euler equations (density) for ME-\hdSG{} and ME-SC method in three random dimensions. 
    Example \ref{subsec:MERefinement}.}
    \label{table:MERefinement3D}
\end{table}
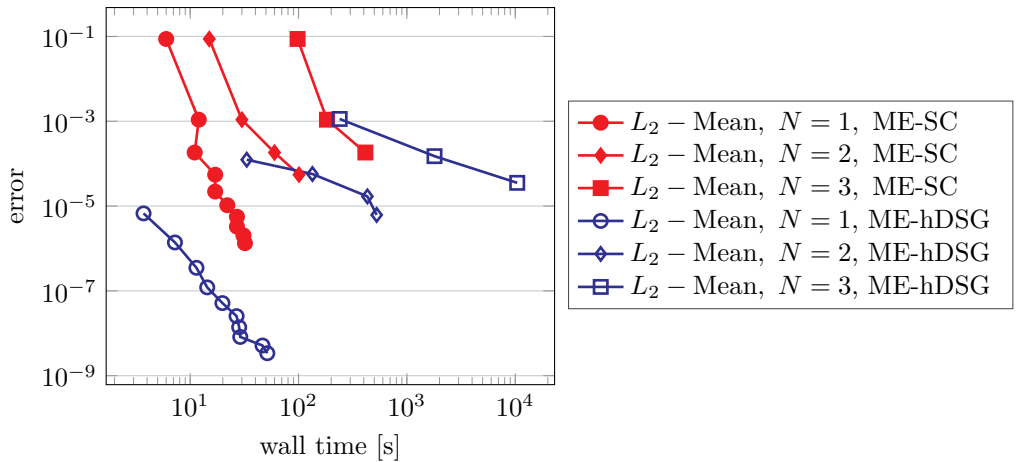
\begin{figure}[htb!]
\centering
  \input{Images/MERef2DEuler/walltimeMERef2D.tex}
  \caption{$\Lp{2}$-errors and wall time for the Euler equations (density) for ME-\hdSG{} and ME-SC method. Example \ref{subsec:MERefinement}}
  \label{fig:MEwalltime}
\end{figure}

\subsection{Uncertain Sod Shock Test}\label{subsec:Euler Equations: Uncertain Sod Shock Test}
In this numerical test we study the behavior of the hyperbolicity limiter $\hyperbolLimit{\limitervariable}$ when it is applied to
the one-dimensional uncertain Sod shock problem from \cite{Poette2009,Schlachter2017a}. 
We consider an uncertain position of the initial discontinuity, i.e., we have the following set of initial conditions
\begin{equation} \label{eq:initialEulerSod}
\left.
\hspace*{2cm}
\begin{alignedat}{1}
 \density(\timevar=0,\x,\uncertaintyD) &= \begin{cases}
1, \qquad &\x < 0.5 + 0.05\uncertaintyD, \hspace*{2cm}\\
0.125, \qquad &\x\geq 0.5 + 0.05\uncertaintyD,
\end{cases}\\
\momentum(\timevar=0,\x,\uncertaintyD) &= 0,\\
\energy(\timevar =0,\x,\uncertaintyD) &= \begin{cases}
2.5, \qquad &\x < 0.5 + 0.05\uncertaintyD,\\
0.25, \qquad &\x\geq 0.5 + 0.05\uncertaintyD,
\end{cases}
\end{alignedat}
\right\}
\end{equation}
where $\uncertaintyD\sim\mathcal{U}(-1,1)$.
The numerical solution is computed up to $T=0.2$ and we define the spatial domain as $\Domain = [0,\,1]$. At the boundary we prescribe
exact boundary conditions. We choose the Lax-Friedrichs numerical flux and the TVBM minmod limiter from \cite{CockburnShu2001} 
as spatial limiter $\tvbminmod$. 

We divide $\Domain$ into $500$ cells, set the DG polynomial degree to three and consider the \hdSG{} and ME-\hdSG{} method.
For the \hdSG{} scheme we use a truncation order of $\SGtruncorder=10$ and for ME-\hdSG{} we consider $\nRcells=10$ random elements and a linear approximation, i.e. $\SGtruncorder=1$. 
The number of quadrature nodes is set to $\nbxiQuadNodes=20$ and both methods are compared to a Monte Carlo simulation using an exact Riemann solver \cite{Backus2017} with $200\,000$ samples.

In \figref{fig:IC1_DG_dx500_p3} we compare mean and variance obtained by the \hdSG{} and ME-\hdSG{} methods against the  ``exact'' solution given by Monte Carlo sampling. The expected value in \figref{fig:Euler_a_a} and \figref{fig:Euler_a_b} indicate a good agreement between Monte Carlo, \hdSG{} and ME-\hdSG{}. However, for the \hdSG{} method we can see in \figref{fig:Euler_a_b} and \figref{fig:Euler_a_d}, especially around the shock at $\x\approx0.8$, that the \hdSG{} solution exhibits ($\SGtruncorder+1=11$) small shocks, which has also been observed for example in \cite{DespresPoetteLucor2013}. 
Because of the discontinuities in $\uncertaintyD$, the plain \hdSG{} approach suffers from Gibbs' oscillations and hence using a piecewise linear interpolation, as in the Multi-Element approach, yields a far better resolution of the mean and variance compared to the plain \hdSG{} approach.

\begin{figure}[htb!]
  \centering
  \settikzlabel{fig:Euler_a_a} \settikzlabel{fig:Euler_a_b} \settikzlabel{fig:Euler_a_c}    \settikzlabel{fig:Euler_a_d}
  \externaltikz{IC1_DG_dx500_p3}{\input{Images/IC1_DG_dx500_p3.tex}}
  \caption{Euler equations with initial state \eqref{eq:initialEulerSod}, $\ncells=500$ and DG polynomial degree $\polydegree=3$. For \hdSG{}, $\SGtruncorder=10$ and for ME-\hdSG{} $\SGtruncorder=1$, $\nRcells=10$. Example \ref{subsec:Euler Equations: Uncertain Sod Shock Test}.}
  \label{fig:IC1_DG_dx500_p3}
\end{figure}

\begin{table}[htb!]
\centering
  \begin{tabular}{|l|l|l|l|l|l|l|l|l|l|l|}
  	\cline{1-5}	\cline{7-11}
    \multicolumn{5}{|c|}{ME-\hdSG{}, $\SGtruncorder=0$}   && \multicolumn{5}{c|}{ME-\hdSG{}, $\SGtruncorder=1$}   \\  
    \cline{1-5}	\cline{7-11}
	$\nRcells$ & $\Lp{1}$-Mean & eoc & $\Lp{1}$-Variance & eoc & &
	$\nRcells$ & $\Lp{1}$-Mean &eoc & $\Lp{1}$-Variance & eoc  \\
	\cline{1-5}	\cline{7-11}
	2 &  0.0086  & --   & 8.8862e-04 & --    && 2 &0.0054 & --   & 5.5152e-04 & --     \\
	4 &  0.0043  & 1.02 & 4.3359e-04 & 1.04  && 4 & 0.0027 & 0.98 & 2.8091e-04 & 0.97   \\ 
	8 &  0.0021  & 0.99 & 2.2043e-04 & 0.98  && 8 & 0.0016 & 0.75 & 1.6526e-04 & 0.77   \\
	16 & 0.0012  & 0.83 & 1.2997e-04 & 0.76  && 16 & 0.0011 & 0.55 & 1.1586e-04 & 0.51   \\
	\cline{1-5}	\cline{7-11}
	\end{tabular}

	  \begin{tabular}{|l|l|l|l|l|} 
	  	\multicolumn{5}{c}{ }\\
	   \cline{1-5}	
	  	\multicolumn{5}{|c|}{\hdSG{}}  \\  
    \cline{1-5}	
	$\SGtruncorder$ & $\Lp{1}$-Mean & eoc & $\Lp{1}$-Variance & eoc \\
	\hline
	2 & 0.0075      & --    &  7.3846e-04 & --      \\
	4 & 0.0049      & 0.61  &  5.7026e-04 & 0.37    \\
	8 & 0.0038      & 0.37  &  4.4814e-04 & 0.35    \\
	16 & 0.0037     & 0.02  &  4.3632e-04 & 0.04    \\
	\hline
	\end{tabular}
	\caption{$\Lp{1}$-errors and experimental order of convergence (eoc) for the Euler equations (density) for $\ncells=500$ and DG polynomial degree $\polydegree=3$. Example \ref{subsec:Euler Equations: Uncertain Sod Shock Test}.}
	\label{table:convMEIC1}
\end{table}
\begin{figure}[htb!]
	\centering
	\externaltikz{surfaceIC1_L}{\input{Images/surfaceIC1_L.tex}}
	\caption{Expected value of the density $\density$ at $\timevar=0.2$ for Euler equations with initial state \eqref{eq:initialEulerSod}, $\ncells=500$, DG polynomial degree $\polydegree=3$. For \hdSG{}, $\SGtruncorder=10$ and for ME-\hdSG{} $\SGtruncorder=1$, $\nRcells=10$. Example \ref{subsec:Euler Equations: Uncertain Sod Shock Test}.}
	\label{fig:surf_IC1_DG_dx500_p3}
\end{figure} 

Spurious oscillations for the \hdSG{} method can also be observed in the $\x-\uncertaintyD$-diagram in \figref{fig:surf_IC1_DG_dx500_p3}, especially in the vicinity of the shock-curve around $\x\approx0.8$.
Furthermore, the influence of the $\x-\uncertaintyD$ discontinuities can be seen in Table \ref{table:convMEIC1}, where we show the error in mean and variance between the \hdSG-, ME-\hdSG-approximation and the Monte Carlo solution. 
We deduce that the error for the \hdSG-approach quickly starts to stagnate, whereas the error for the Multi-Element approach is still decreasing.
However, for all three methods the computed order of convergence in this example is smaller than one which is due to the discontinuities in $\uncertaintyD$.% \refone{We derived the order of convergence with \eqref{eq:eoc}, adapting the equation to a refinement in $\nRcells$ instead of $\ncells$.}

In \figref{fig:limIC1} we plot the values of the limiter variable $\limitervariable$ for the \hdSG{} and the ME-\hdSG{} method. In \figref{fig:limIC1SG} we can see that the limiter clearly follows the discontinuity in $\x$. 
For the ME-\hdSG{} scheme we plot the maximum value of $\limitervariable$ over all time-steps, random and physical cells in \figref{fig:limIC1ME}. We observe that the limiter is only active 
for the initial condition. Afterwards, the SG solution does not leave the hyperbolicity set which illustrates the strength of the Multi-Element approach. The superiority of the
ME-\hdSG{} is finally demonstrated in \tabref{table:usageLimiterIC1}. Here, we display the percentage of limited cells for both methods compared to all space-time-stochastic cells. The percentage of limited cells for ME-\hdSG{} is one order of magnitude lower than for the \hdSG{} method.
\begin{figure}[htb]
  \centering
  \settikzlabel{fig:limIC1SG} \settikzlabel{fig:limIC1ME}
  \externaltikz{limIC1_L}{\input{Images/limIC1_L.tex}}
  \caption{Limiter plot for Euler equations with initial state \eqref{eq:initialEulerSod}, $\ncells=500$ and DG polynomial degree $\polydegree=3$. For \hdSG{}, $\SGtruncorder=2$ and for ME-\hdSG{} $\SGtruncorder=1$, $\nRcells=10$. For a better visualization of the \hdSG{} method, we have used a logarithmic z-scale. Example \ref{subsec:Euler Equations: Uncertain Sod Shock Test}}
  \label{fig:limIC1}
\end{figure}
\begin{table}[htb!]
\centering
  \begin{tabular}{|l|l|l|l|l|l|l|l|l|l|}
  \hline 
    $\SGtruncorder$/$\nRcells$  & 1 &2 &3 &4 &5 &6&7 &8 &9  \\ \hline
    \hdSG{}     \hfill [\%]    & 0.0146 & 0.1473 & 0.0721  &0.0457 &0.0376  &0.0386&0.0335 & 0.0251 &0.0199   \\
    ME-\hdSG{} \hfill [\%]     & 0.0146 & 0.0028  &0.0021   &0.0034  &0.0037   &0.0022 & 0.0024 &0.0021 & 0.0008  \\
    \hline
	\end{tabular}
	\caption{Percentage of limited cells over all time-steps for the Euler equations with $\ncells=500$ and DG polynomial degree $\polydegree=3$. For ME-\hdSG{} we use $\SGtruncorder=1$.  Example \ref{subsec:Euler Equations: Uncertain Sod Shock Test}  }
	\label{table:usageLimiterIC1}
\end{table}

\subsection{Double Mach Reflection with Uncertain Angle} \label{subsec:dmr}
As a final numerical test for the hyperbolicity-preserving limiter,
we consider the Double Mach Reflection test case suggested by Woodward and Colella  \cite{WoodwardColella1984}.
It consists of a Mach 10 shock wave that hits a ramp which is inclined by 30 degrees. 
The Double Mach reflection poses a very challenging problem for the 
\hdSG{} method because the solution is very likely to leave the hyperbolicity set due to the high jump in pressure.
We choose the angle of the ramp uncertain,
i.e., we let $\uncertaintyD \sim \mathcal{U}(28^\circ,32^\circ)$ and consider the following Riemann data in primitive variables
\begin{equation} \label{eq:initialDMR}
\left.
\hspace*{2cm}
\begin{alignedat}{1}
\density(\timevar=0,\x,\y,\uncertaintyD) &= \begin{cases}
8, \qquad & \x< \overline{x} + \tan\left(\frac{\uncertaintyD \pi }{180^\circ}\right)\y, \hspace*{2cm}\\
0.125, \qquad  &\x \geq \overline{x} + \tan\left(\frac{\uncertaintyD \pi }{180^\circ}\right)\y,
\end{cases}\\
\velocitydim{1}(\timevar=0,\x,\y,\uncertaintyD) &= \begin{cases}
8.25 \cos\left(\frac{\uncertaintyD \pi}{180^\circ}\right)& \x< \overline{x} + \tan\left(\frac{\uncertaintyD \pi }{180^\circ}\right)\y, \hspace*{2cm}\\
0 , &\x \geq \overline{x} + \tan\left(\frac{\uncertaintyD \pi }{180^\circ}\right)\y,
\end{cases}\\
\velocitydim{2}(\timevar=0,\x,\y,\uncertaintyD) &= \begin{cases}
-8.25 \cos\left(\frac{\uncertaintyD \pi}{180^\circ}\right),\qquad &\x< \overline{x} + \tan\left(\frac{\uncertaintyD \pi }{180^\circ}\right)\y, \hspace*{2cm}\\
0 ,\qquad & \x \geq \overline{x} + \tan\left(\frac{\uncertaintyD \pi }{180^\circ}\right)\y,
\end{cases}\\
\pressure(\timevar =0,\x,\y,\uncertaintyD) &= \begin{cases}
116.5,\qquad &\x< \overline{x} + \tan\left(\frac{\uncertaintyD \pi }{180^\circ}\right)\y, \hspace*{2cm}\\
1, \qquad & \x \geq \overline{x} + \tan\left(\frac{\uncertaintyD \pi }{180^\circ}\right)\y,
\end{cases}
\end{alignedat}
\right\}
\end{equation}
where $\overline{x}=\frac{1}{6}$ is the start of the ramp. The computational domain is $\Domain=[0,4]\times[0,1]$ and we set $\timeint=0.2$.
At the bottom of the domain we employ reflective boundary conditions whereas we prescribe outflow boundary conditions at the right.
At the remaining boundaries we apply Dirichlet boundary conditions, which correspond 
to the physical values. We use the HLLE  numerical flux and according to Remark \ref{rem:subcelllimiter} we choose the FV sub-cell limiter as spatial limiter $\tvbminmod$. To detect troubled cells we implement the modified JST indicator as described in \cite{SonntagMunz2017}.

For this numerical example we apply the ME-\hdSG{} method with $\nRcells=8$ MEs, SG polynomial degree $\SGtruncorder=4$ and $\nbxiQuadNodes=20$.
The physical mesh consists of $\ncells=190\times 40$ cells and we use a DG polynomial degree of $\polydegree = 4$.
In \figref{fig:dmrMean} we plot mean and standard deviation of the density at time  $\timeint=0.2$.
We can see that the shock front is smeared out because of the variable angle.
Thanks to the high-order resolution in physical and stochastic space, 
small-scale flow features in mean and variance of density are clearly visible.
A high standard deviation can be identified around the position of
the shock wave. Around $\overline{x}=\frac{1}{6}$, which corresponds to the start position of the ramp,
we also observe a high standard deviation.

\figref{fig:dmrLimiter} shows the values of the limiter variable $\limitervariable$ in ME one and eight at time $\timeint=0.2$. We see
that the limiter is only active around the shock front. Furthermore, the limited cells vary with the uncertain angle.
\refone{We want to emphasize that the plain SG and even the plain ME-SG approach without the hyperbolicity-preserving limiter crash immediately after initialization of the initial condition.}
This shows that our proposed scheme is a reliable and robust method to compute complex flow problems with a high resolution in
space, time and stochasticity.
	
\begin{figure}[!htb]
  \centering
  \includegraphics[width = 0.8 \textwidth]{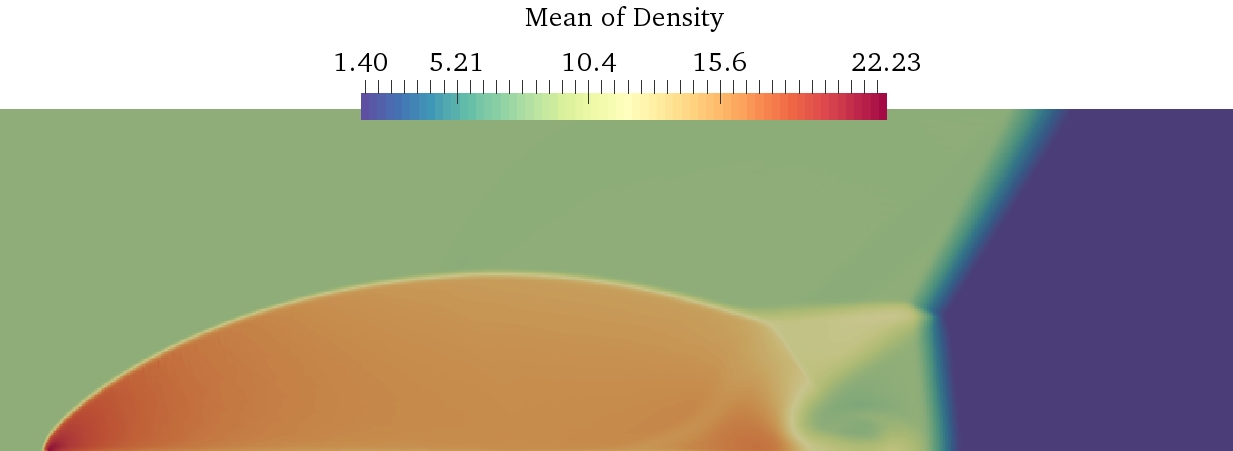}
  \includegraphics[width = 0.8 \textwidth]{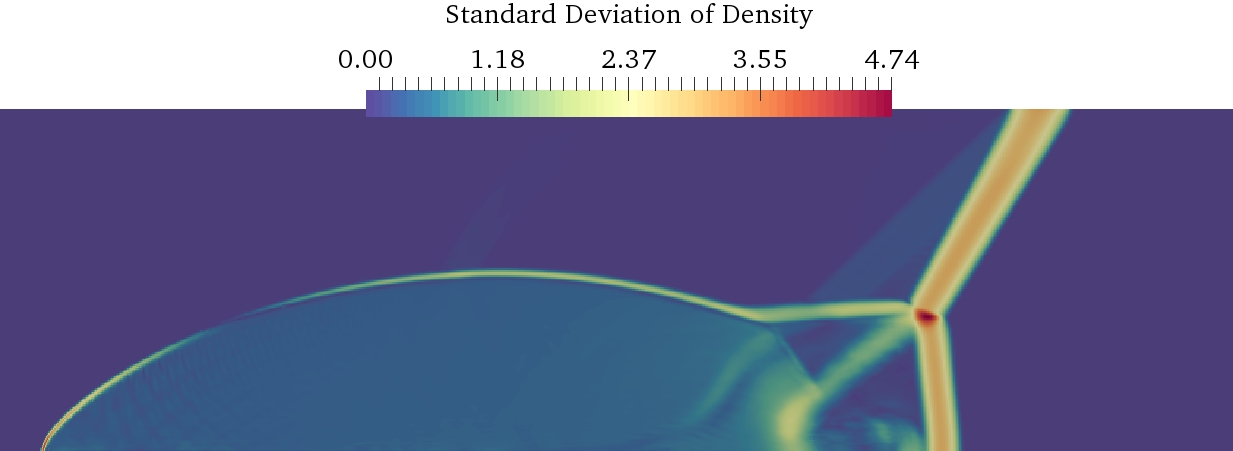}
  \caption{Euler equations with initial state \eqref{eq:initialDMR}, $\ncells=7600$ and DG polynomial degree $\polydegree=4$. For ME-\hdSG{}, $\SGtruncorder=1$, $\nRcells=8$. Example \ref{subsec:dmr} }
  \label{fig:dmrMean}
\end{figure}

\begin{figure}[!htb]
  \centering
  \includegraphics[width = 0.9 \textwidth]{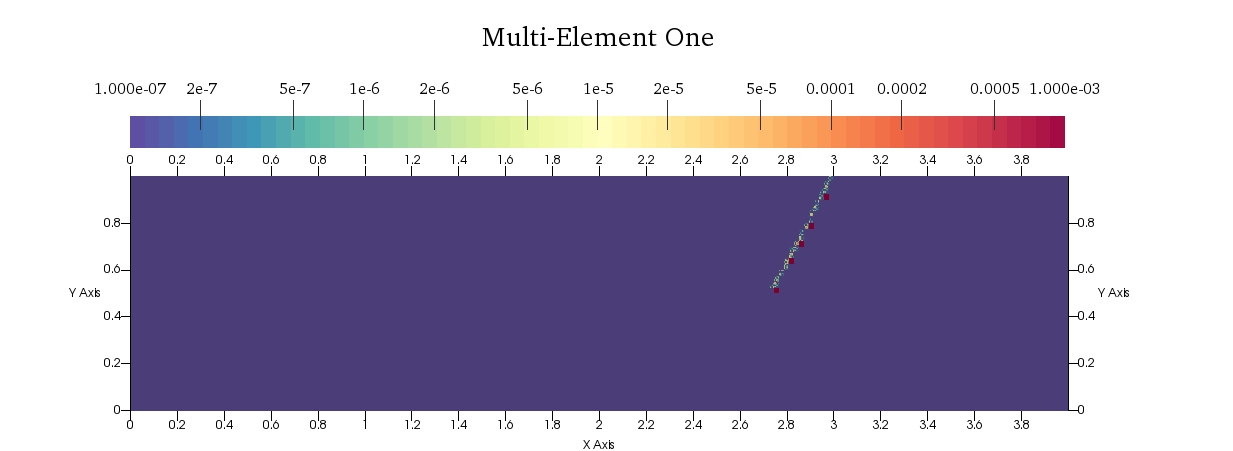}
  \includegraphics[width = 0.9 \textwidth]{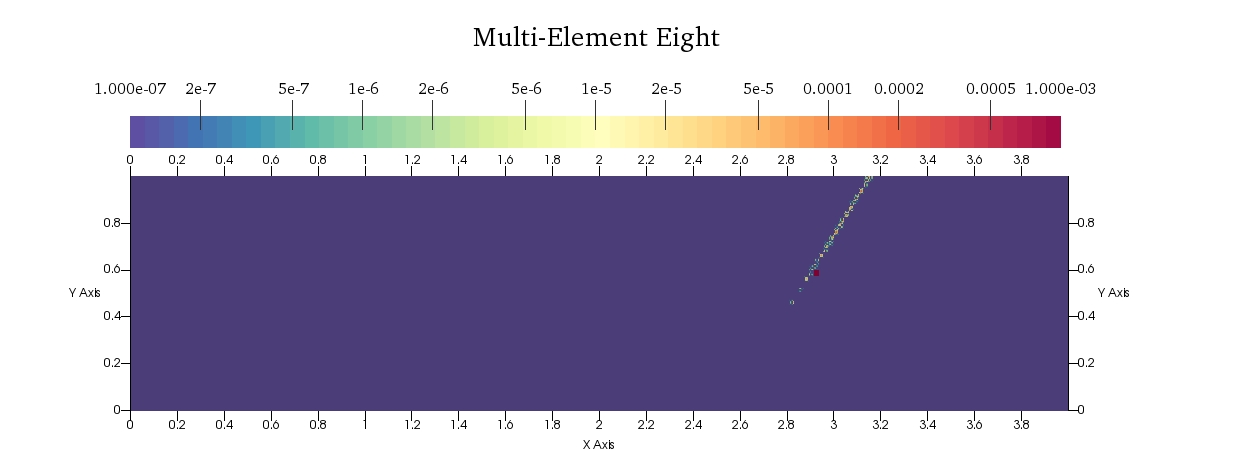}
  \caption{Limiter plot for Euler equations with initial state \eqref{eq:initialDMR}, $\ncells=7600$ and DG polynomial degree $\polydegree=4$. We plot the usage of the limiter in ME one and eight
  at time $\timeint=0.2$. Example \ref{subsec:dmr}}
  \label{fig:dmrLimiter}
\end{figure}

%% file: Images/pRef2DEuler/walltimepRef2D.tex
\begin{tikzpicture}
\begin{axis}[%
width=1\figurewidth,
height=\figureheight,
scale only axis,
xmode=log,
xminorticks=true,
xlabel={wall time [s]},
ymode=log,
yminorticks=false,
ymajorgrids=true,
yminorgrids = true,
ylabel={error},
legend pos = north east,
legend style={at={(1.03,0.2)},anchor=south west,legend cell align=left, align=left, draw=white!15!black},
]
\addplot [color=red, line width=1.0pt, mark size=2.5pt, mark=*, mark options={solid, red}] table [x=walltime, y=l2mean, col sep=comma] {Images/pRef2DEuler/pRefSC1D.csv};
\addlegendentry{$\Lp{2}-\text{Mean},~\nstochdim=1,~\text{SC}$}
\addplot [color=red, line width=1.0pt, mark size=2.5pt, mark=diamond*, mark options={solid, red}] table [x=walltime, y=l2mean, col sep=comma] {Images/pRef2DEuler/pRefSC2D.csv};
\addlegendentry{$\Lp{2}-\text{Mean},~\nstochdim=2,~\text{SC}$}
\addplot [color=red, line width=1.0pt, mark size=2.5pt, mark=square*, mark options={solid, red}] table [x=walltime, y=l2mean, col sep=comma] {Images/pRef2DEuler/pRefSC3D.csv};
\addlegendentry{$\Lp{2}-\text{Mean},~\nstochdim=3,~\text{SC}$}
\addplot [color=blue, line width=1.0pt, mark size=2.5pt, mark=o, mark options={solid, blue}] table [x=walltime, y=l2mean, col sep=comma] {Images/pRef2DEuler/pRefSG1D.csv};
\addlegendentry{$\Lp{2}-\text{Mean},~ \nstochdim=1$, \hdSG{} }
\addplot [color=blue, line width=1.0pt, mark size=2.5pt, mark=diamond, mark options={solid, blue}] table [x=walltime, y=l2mean, col sep=comma] {Images/pRef2DEuler/pRefSG2D.csv};
\addlegendentry{$\Lp{2}-\text{Mean},~ \nstochdim=2$, \hdSG{} }
\addplot [color=blue, line width=1.0pt, mark size=2.5pt, mark=square, mark options={solid, blue}] table [x=walltime, y=l2mean, col sep=comma] {Images/pRef2DEuler/pRefSG3D.csv};
\addlegendentry{$\Lp{2}-\text{Mean},~ \nstochdim=3$, \hdSG{} }

\end{axis}
\end{tikzpicture}%

%% file: Images/MERef2DEuler/walltimeMERef2D.tex
\begin{tikzpicture}
\begin{axis}[%
width=1\figurewidth,
height=\figureheight,
scale only axis,
xmode=log,
xminorticks=true,
xlabel={wall time [s]},
ymode=log,
yminorticks=false,
ymajorgrids=true,
yminorgrids = true,
ylabel={error},
legend pos = north east,
legend style={at={(1.03,0.2)},anchor=south west,legend cell align=left, align=left, draw=white!15!black},
]
\addplot [color=red, line width=1.0pt, mark size=2.5pt, mark=*, mark options={solid, red}] table [x=walltime, y=l2mean, col sep=comma] {Images/MERef2DEuler/MERefSC1D.csv};
\addlegendentry{$\Lp{2}-\text{Mean},~\nstochdim=1,~\text{ME-SC}$}
\addplot [color=red, line width=1.0pt, mark size=2.5pt, mark=diamond*, mark options={solid, red}] table [x=walltime, y=l2mean, col sep=comma] {Images/MERef2DEuler/MERefSC2D.csv};
\addlegendentry{$\Lp{2}-\text{Mean},~\nstochdim=2,~\text{ME-SC}$}
\addplot [color=red, line width=1.0pt, mark size=2.5pt, mark=square*, mark options={solid, red}] table [x=walltime, y=l2mean, col sep=comma] {Images/MERef2DEuler/MERefSC3D.csv};
\addlegendentry{$\Lp{2}-\text{Mean},~\nstochdim=3,~\text{ME-SC}$}
\addplot [color=blue, line width=1.0pt, mark size=2.5pt, mark=o, mark options={solid, blue}] table [x=walltime, y=l2mean, col sep=comma] {Images/MERef2DEuler/MERefSG1D.csv};
\addlegendentry{$\Lp{2}-\text{Mean},~ \nstochdim=1$, ME-\hdSG{} }
\addplot [color=blue, line width=1.0pt, mark size=2.5pt, mark=diamond, mark options={solid, blue}] table [x=walltime, y=l2mean, col sep=comma] {Images/MERef2DEuler/MERefSG2D.csv};
\addlegendentry{$\Lp{2}-\text{Mean},~ \nstochdim=2$, ME-\hdSG{} }
\addplot [color=blue, line width=1.0pt, mark size=2.5pt, mark=square, mark options={solid, blue}] table [x=walltime, y=l2mean, col sep=comma] {Images/MERef2DEuler/MERefSG3D.csv};
\addlegendentry{$\Lp{2}-\text{Mean},~ \nstochdim=3$, ME-\hdSG{} }

\end{axis}
\end{tikzpicture}%

%% file: Images/IC1_DG_dx500_p3.tex
	
	\begin{tikzpicture}		
	\begin{groupplot}[group style={group size=2 by 2, horizontal sep = 2cm,  vertical sep = 2cm},
	width=\figurewidth,
	height=\figureheight,
	scale only axis,
	cycle list name=color fabian,
	]
	\nextgroupplot[
	title = \tikztitle{Expected value $\density$ at $t=0.1$},
	xlabel= {$\x$},
	ylabel = {$\density$},
	ylabel style = {rotate=-90},			
	]
	\addplot+ [smooth,mark repeat = 222, thick] file {Images/IC1_DG_dx500_p3/MCt1.txt};
	\addplot+ [smooth,each nth point=2,mark repeat = 222, mark phase = 60] file {Images/IC1_DG_dx500_p3/hSGDGt1.txt};
	\addplot+ [smooth,each nth point=2,mark repeat = 222, mark phase = 100] file {Images/IC1_DG_dx500_p3/MEhSGDGt1.txt};
	
	%(densely / loosely) dashed / dotted
	\legend{ MC, \hdSG{}, ME-\hdSG{}};

	\nextgroupplot[
	title = \tikztitle{Expected value $\density$ at $t=0.2$},
	xlabel= {$\x$},
	ylabel style = {rotate=-90},			
	]
	\addplot+ [smooth,mark repeat = 222, thick] file {Images/IC1_DG_dx500_p3/MCt2.txt};
	\addplot+ [smooth,each nth point=2,mark repeat = 222, mark phase = 60] file {Images/IC1_DG_dx500_p3/hSGDGt2.txt};
	\addplot+ [smooth,each nth point=2,mark repeat = 222, mark phase = 100] file {Images/IC1_DG_dx500_p3/MEhSGDGt2.txt};
	\legend{ \small MC, \hdSG{}, ME-\hdSG{}};
	
	\nextgroupplot[
	title = \tikztitle{Variance $\density$ at $t=0.1$},
	xlabel= {$\x$},
	ylabel = {$\density$},
	ylabel style = {rotate=-90},
	scaled y ticks = false,
	y tick label style={ /pgf/number format/.cd, fixed,precision = 4,/tikz/.cd},
	]
	\addplot+ [smooth,mark repeat = 222,thick]  file {Images/IC1_DG_dx500_p3/MCt1_var.txt};
	\addplot+ [smooth,each nth point=2,mark repeat = 222, mark phase = 60] file {Images/IC1_DG_dx500_p3/hSGDGt1_var.txt};
	\addplot+ [smooth,each nth point=2,mark repeat = 222, mark phase = 100] file {Images/IC1_DG_dx500_p3/MEhSGDGt1_var.txt};
    \legend{ MC, \hdSG{}, ME-\hdSG{}};
	
    \nextgroupplot[
	title = \tikztitle{Variance $\density$ at $t=0.2$},
	xlabel= {$\x$},
	ylabel style = {rotate=-90},
	scaled y ticks = false,
	y tick label style={ /pgf/number format/.cd, fixed,precision = 4,/tikz/.cd},
	]
	\addplot+ [smooth,each nth point=2,mark repeat = 222,thick]  file {Images/IC1_DG_dx500_p3/MCt2_var.txt};
	\addplot+ [smooth,each nth point=2,mark repeat = 222, mark phase = 60] file {Images/IC1_DG_dx500_p3/hSGDGt2_var.txt};
	\addplot+ [smooth,each nth point=2,mark repeat = 222, mark phase = 100] file {Images/IC1_DG_dx500_p3/MEhSGDGt2_var.txt};
	
    \legend{ MC, \hdSG{}, ME-\hdSG{}};
	
	\end{groupplot}
	\end{tikzpicture}
	

%% file: Images/surfaceIC1_L.tex
\begin{tikzpicture}
    \begin{groupplot}[
group style={group size=2 by 1, horizontal sep = 2cm,  vertical sep = 2cm},
        axis on top,
        scale only axis,
        enlargelimits=false,
        width = \figurewidth, 
        height = \figureheight, 
        colormap name ={jet},
        colorbar horizontal,
        title style = {yshift = 0.6cm} ,
        scaled x ticks=false,
        ylabel = {$\mathds{E}(\density)$},
        colorbar style={,
             at={(0,1.02)},
             anchor=south west,
             height=0.02\textwidth,
             width=\figurewidth, 
             xticklabel style={font=\footnotesize,anchor=south, /pgf/number format/.cd, fixed,precision = 2,/tikz/.cd},
             xticklabel shift = -7pt,
                       },
        ]
\nextgroupplot[
        xmin=0.000000,
        xmax=1.000000,
        ymin=-1.000000,
        ymax=1.000000,
        xlabel={$x$},
        ylabel={$\xi$},
        ylabel style = {rotate=-90},
        title = \tikztitle{\hdSG{}},
        point meta min = 0.085585,
        point meta max = 1.006629,
        ]
      \addplot graphics[xmin=0.000000,xmax=1.000000,ymin=-1.000000,ymax=1.000000] {Images/surfaceIC1/IC1N10.png};
      
      \nextgroupplot[
      xmin=0.000000,
      xmax=1.000000,
      ymin=-1.000000,
      ymax=1.000000,
      xlabel={$x$},
      ylabel={},
      title = \tikztitle{ME-\hdSG{}},
      point meta min = 0.085585,
      point meta max = 1.006629,
%      point meta min = 0.125000,
%      point meta max = 1.000000,
      ]
      \addplot graphics[xmin=0.000000,xmax=1.000000,ymin=-1.000000,ymax=1.000000] {Images/surfaceIC1/IC1ME.png};
    \end{groupplot}
\end{tikzpicture}

%% file: Images/limIC1_L.tex
\begin{tikzpicture}
    \begin{groupplot}[
group style={group size=2 by 1, horizontal sep = 2cm,  vertical sep = 2cm},
        axis on top,
        scale only axis,
        enlargelimits=false,
        width = \figurewidth, 
        height = \figureheight, 
        colormap name ={myhot},
        colorbar horizontal,
        title style = {yshift = 0.6cm} ,
        scaled x ticks=false,
        ylabel = {$\mathds{E}(\density)$},
        colorbar style={,
             at={(0,1.02)},
             anchor=south west,
             height=0.02\textwidth,
             width=\figurewidth, 
             xticklabel style={font=\footnotesize,anchor=south, /pgf/number format/.cd, fixed,precision = 5,/tikz/.cd},
             xticklabel shift = -7pt,
                       },
        ]
\nextgroupplot[
        xmin=0.000000,
        xmax=1.000000,
        ymin=0.000000,
        ymax=0.200000,
        xlabel={$x$},
        ylabel={$t$},
        ylabel style = {rotate=-90},
        yticklabel style={/pgf/number format/.cd, fixed,precision = 2,/tikz/.cd},
        title = \tikztitle{\hdSG{}},
        point meta min = -14.000000,
        point meta max = -0.556721,
		colorbar style={ %
	xticklabel style={font=\footnotesize,anchor=west,/pgf/number format/.cd,
		/tikz/.cd},
	xticklabel = $10^{\pgfmathparse{\tick}\pgfmathprintnumber\pgfmathresult}$, 
	xticklabel shift = -14pt,                                                 
}, 
        ]
      \addplot graphics[xmin=0.000000,xmax=1.000000,ymin=0.0006000,ymax=0.200000] {Images/limIC1/limIC1SG.png};
      
      \nextgroupplot[
      xmin=0.000000,
      xmax=1.000000,
      ymin=0.000000,
      ymax=0.200000,
      xlabel={$x$},
      ylabel={},
      yticklabel style={/pgf/number format/.cd, fixed,precision = 2,/tikz/.cd},
      title = \tikztitle{ME-\hdSG{}}, 
      point meta min = 0.000000,
	  point meta max = 0.517412,
      ]
      \addplot graphics[xmin=0.000000,xmax=1.000000,ymin=0.00000,ymax=0.200000] {Images/limIC1/limIC1ME.png};
    \end{groupplot}
\end{tikzpicture}

%% file: Sections/conclusions.tex
\section{Conclusions and Outlook}
Throughout this paper, we have extended the hyperbolicity-preserving stochastic slope limiter from \cite{Schlachter2017} to multiple dimensions in  the physical and stochastic space.
To derive a high-order method, we combined the hyperbolicity-preserving
stochastic Galerkin scheme with a Multi-Element ansatz in the stochastic space and  a Runge--Kutta discontinuous Galerkin spatial discretization.
The resulting  ME-\hdSG{} scheme significantly improved the results compared to the standard SG scheme, 
especially when discontinuities are present in the uncertainty. 
Moreover, we compared the performance of our presented numerical scheme in multiple stochastic dimensions to that of 
the non-intrusive Stochastic Collocation method.
Our results show that our method is competitive in less than three stochastic dimensions,
since the cost of the SG system in higher dimensions is increasing significantly.
As a final numerical example, we applied our numerical scheme to the Double Mach reflection problem, \refone{where even the plain ME-SG approach without hyperbolic limiter fails. This shows in particular that our method is capable to handle challenging flow problems.}

Future work should incorporate adaptive refinements in space and stochasticity to further improve the efficiency of the ME-\hdSG{} method 
and to extend the range of dimensions in which the ME-\hdSG{} scheme outperforms non-intrusive methods. 
Moreover, we want to \refone{adapt} our methodology to the intrusive polynomial moment method, which promises to require less hyperbolicity limiting.
\section*{Acknowledgments}
Funding by the Deutsche Forschungsgemeinschaft (DFG) within the RTG GrK 1932 ``Stochastic Models for Innovations in the Engineering Science'' is gratefully acknowledged.
J.D., T.K. and F.M thank the Baden-W{\"u}rttemberg Stiftung for support via the project ``SEAL''.